\documentclass[11pt]{article}
\usepackage{bbm}
 \usepackage{amssymb}
\usepackage{amssymb, amsthm, amsmath, amscd}
\usepackage[usenames,dvipsnames]{color}

\newtheorem{thm}{Theorem}[section]
\newtheorem{lem}[thm]{Lemma}
\newtheorem{prop}[thm]{Proposition}
\newtheorem{cor}[thm]{Corollary}
\theoremstyle{definition}

\theoremstyle{definition}
\newtheorem{df}[thm]{Definition}
\theoremstyle{definition}
\newtheorem{rem}[thm]{Remark}
\newtheorem{nota}[thm]{Notation}
\theoremstyle{definition}

\newcommand{\red}{\textcolor{red}}
\newcommand{\blue}{\color{blue}}
\newcommand{\green}{\color{green}}
\newcommand{\Green}{\color{Green}}

\renewcommand{\phi}{\varphi}

\newcommand{\wtilde}{\widetilde}

\definecolor{purple}{RGB}{150,10,200} 
\newcommand{\xc}{\color{purple}} 
\newcommand{\xrule}{\bigskip \hrule \bigskip}

\newcommand{\CAs}{$C^*$-algebras}

\newcommand{\N}{\mathbb{N}}
\newcommand{\Z}{\mathbb{Z}}
\newcommand{\Q}{\mathbb{Q}}
\newcommand{\R}{\mathbb{R}}
\newcommand{\C}{\mathbb{C}}

\numberwithin{equation}{section}

\newcommand{\Aff}{\operatorname{Aff}}

\newcommand{\id}{\operatorname{id}}

\newcommand{\Her}{\mathrm{Her}}

\newcommand{\hm}{homomorphism}
\newcommand{\dt}{\delta}
\newcommand{\ep}{\varepsilon}

\newcommand{\sg}{\sigma}

\newcommand{\la}{\langle}
\newcommand{\ra}{\rangle}
\newcommand{\andeqn}{\,\,\,{\rm and}\,\,\,}
\newcommand{\rforal}{\,\,\,{\rm for\,\,\,all}\,\,\,}
\newcommand{\CA}{$C^*$-algebra}
\newcommand{\SCA}{$C^*$-subalgebra}

\newcommand{\af}{{\alpha}}
\newcommand{\bt}{{\beta}}

\newcommand{\wtd}{\widetilde}

\newcommand{\diag}{{\rm diag}}

\newcommand{\wilog}{without loss of generality}
\newcommand{\Wlog}{Without loss of generality}

\newcommand{\D}{\mathbb D}
\newcommand{\beq}{\begin{eqnarray}}
\newcommand{\eneq}{\end{eqnarray}}
\newcommand{\tforal}{\,\,\,\text{for\,\,\,all}\,\,\,}

\newcommand{\p}{\mathfrak{p}}
\newcommand{\q}{\mathfrak{q}}
\newcommand{\fd}{\mathfrak{d}}

\usepackage{amsfonts}
\usepackage{mathrsfs}
\usepackage{textcomp}
\usepackage[all]{xy}

\title{Non-amenable simple \CA s with tracial approximation}
\author{Xuanlong Fu
and 
Huaxin Lin
 }
\date{
}

\begin{document}

\maketitle

\begin{abstract}
We construct two types of  unital separable simple \CA s $A_z^{C_1}$ and $A_z^{C_2},$
one is exact but not amenable, and the other is non-exact. Both
have  the same Elliott invariant as the Jiang-Su algebra, namely,
$A_z^{C_i}$
has  a unique tracial state,
$$(K_0(A_z^{C_i}), K_0(A_z^{C_i})_+, [1_{A_z^{C_i}}
])=(\Z, \Z_+,1)$$
and $K_{1}(A_z^{C_i})=\{0\}$ ($i=1,2$).
We show that $A_z^{C_i}$ ($i=1,2$) is essentially tracially in the class of separable ${\cal Z}$-stable \CA s
of nuclear dimension 1.
$A_z^{C_i}$ has stable rank one,
strict comparison for positive elements
and no 2-quasitrace other than the unique tracial state.
We also produce models of unital separable simple non-exact \CA s which
are essentially tracially in the  class of simple separable nuclear ${\cal Z}$-stable \CA s
and the  models exhaust all possible weakly unperforated Elliott invariants.
We also discuss some basic properties of essential tracial approximation.
\end{abstract}

\section{Introduction}

Simple unital projectionless 
amenable  \CA s were  first constructed by B. Blackadar (\cite{Bexample}).
The \CA\ $A$ constructed by   Blackadar  has property that $K_0(A)=\Z$ with the usual order but with non-trivial $K_1(A).$
{{The}} Jiang-Su algebra ${\cal Z}$ given by X. Jiang and H. Su
(\cite{JS1999})  is  a 
unital infinite dimensional separable amenable simple \CA\ with the Elliott invariant exactly
{{the same as
that of}} the
complex field $\C.$
Let $A$ be any $\sigma$-unital \CA. Then 
$K_i(A)=K_i(A\otimes {\cal Z})$ ($i=0,1$) as abelian groups
and $T(A)\cong T(A\otimes {\cal Z}).$
If $A$ is a separable simple \CA\, then $A\otimes {\cal Z}$ has  nice regularities.
For example, $A\otimes {\cal Z}$ is either purely infinite, or stably finite (\cite{Rorjs}).
In fact, {{if}} $A\otimes {\cal Z}$ is not purely infinite,
then it has stable rank one  when
$A$ is not stably projectionless (\cite{Rorjs}),
or it almost has 
stable rank one when it is  stably projectionless (\cite{Rob16}).  Also $A\otimes {\cal Z}$ has
weakly unperforated $K_0$-group (\cite{GJS}).
Another important regularity  is that $A\otimes {\cal Z}$ has strict comparison
(\cite{Rorjs}) (see also Definition \ref{2-quasitrace}).
If $A$ has weakly unperforated $K_0(A),$ then $A$ and $A\otimes {\cal Z}$ have
the same Elliott invariant.  In other words, $A$ and  $A\otimes {\cal Z}$  are not distinguishable 
from the Elliott invariant.

The Jiang-Su algebra ${\cal Z}$ is an inductive limit of one-dimensional non-commutative CW complexes.
In fact ${\cal Z}$ is the unique infinite dimensional separable simple \CA\, with finite nuclear dimension in the UCT class
which has the same Elliott invariant as that of the complex field $\C$ 
(see  Corollary 4.12 of  \cite{EGLN}).
These
properties give
${\cal Z}$ a prominent role in the study of structure of \CA s, in particular, in the study of classification of amenable simple \CA s.

Attempts to construct a non-exact Jiang-Su type \CA\ have been in horizon for over a decade.
In particular, after M.~D\u{a}d\u{a}rlat
constructed non-amenable models for
non type I separable unital  AF-algebras {(\cite{D2000}), this should be possible.
The construction in 
\cite{D2000} generalized
some of earlier constructions
of simple \CA s of real rank zero such as
that of Goodearl (\cite{G93}).  Jiang and Su's construction has  a quite different feature.
To avoid producing 
any non-trivial projections, Jiang and Su did not use any
finite dimensional representations.  The construction used
prime dimension drop algebras and connecting maps are highly inventive so that the traces eventually collapse
to one. In fact, M. R\o rdam and W. Winter had another approach 
(\cite{RW-JS-Revisited}) using  a \SCA\, of
$C([0,1], M_{\p}\otimes M_\q),$ where $\p$ and $\q$ are relatively prime supernatural numbers.
One possible attempt 
to construct a non-amenable Jiang-Su type \CA\,
 would use $C([0,1], B_\p\otimes B_q),$ where $B_\p$ and $B_\q$  are non-amenable models for
$M_\p$ and $M_\q$
constructed in \cite{D2000},
 respectively.  However, one usually would avoid computation of $K$-theory  of tensor products
of non-exact simple \CA s such as $B_\p$ and $B_\q.$  Moreover, R\o rdam and Winter's construction depend
on knowing the existence of the Jiang-Su algebra ${\cal Z}.$  On the other hand, if  one considers  non-exact
interval ``dimension drop 
algebras",  aside of controlling $K$-theory, one has additional  issues such as each fiber
of the ``dimension drop 
algebra"  is not simple (unlike the usual
dimension 
drop algebras whose
fibers are simple matrix algebras).

We will present  some  non-exact (or non-nuclear) unital separable simple \CA s $A_z^C$ which have
the property that their  Elliott invariants are the same as that the Jiang-Su algebra ${\cal Z},$ namely,
$(K_0(A_z^C), K_0(A_z^C)_+, [1_{A_z^C}])=(\Z, \Z_+,1),$ $K_1(A_z^C) 
=\{0\},$ and $A_z^C$ has a unique tracial state.
Moreover, $A_z^C$ has stable rank one and  has strict comparison for positive elements.
$A_z^C$ has no (nonzero) 2-quasitrace  other than the unique tracial state.
Even though $A_z^C$ may not be exact, it is essentially tracially  approximated by ${\cal Z}.$ In particular,
it is essentially tracially approximated by unital simple \CA s with nuclear dimension 1.

In this 
paper, we will also study the tracial approximation.
We  will make it precise 
what we mean by   that $A_z^C$ is essentially tracially approximated by ${\cal Z}$
(Definition \ref{DTrapp} {\blue{and}}  Lemma \ref{LCCtrnz}).
We believe that regular properties such as stable rank one, strictly comparison for positive elements,
or almost 
unperforated Cuntz semigroup, as well as approximate divisibility are preserved by
tracial approximation. In fact, we show that if a unital separable simple \CA\, $A$ which is essentially tracially
in ${\cal C}_{\cal Z},$ the class of ${\cal Z}$-stable \CA s, then, as far  as usual regularities  concerned,
$A$ behaves just like \CA s in ${\cal C}_{\cal Z}.$  More precisely, we show that
if $A$ is  simple and essentially tracially in ${\cal C}_{\cal Z},$ then
$A$ is tracially approximately divisible. If $A$ is not purely infinite, then $A$ has stable rank one (or almost has stable rank one, if $A$ is not unital),
has strict comparison, and its Cuntz semigroup is almost 
unperforated. If $A$ is essentially tracially in
the class of exact \CA s, then every 2-quasitrace of $\overline{aAa},$ for any $a$ in the Pedersen
ideal of $A,$  is in fact a trace.

Using $A_z^C,$ we present a large class of non-exact
unital separable simple \CA s
which exhaust all possible weakly unperforated Elliott invariant.
 Moreover, every \CA\, in the class is essentially tracially in  the class
of unital 
separable simple \CA s which are ${\cal Z}$-stable and has nuclear dimension at most  1.

The paper is organized as follows:
Section 2 serves as 
preliminaries where some of frequently used notations and definitions
are listed. Section 3 introduces the notion of essential tracial approximation for simple \CA s.
In Section 4, we present some basic properties of essential tracial approximation.
For example, we show that, if $A$ is a simple \CA\, and is essentially tracially approximated by
\CA s whose Cuntz semigroups are  almost unperforated, then the Cuntz semigroup of $A$ is almost unperforated 
(Theorem \ref{Tstcomp}). 
In particular, $A$ has strict comparison for positive elements.
In Section 5, we study the separable simple \CA s which are essentially tracially approximated
by ${\cal Z}$-stable \CA s. We show such \CA s are either purely infinite, or
almost has stable rank one (or stable rank one if the \CA s are unital).  These simple \CA s
are tracially approximately divisible and have strict comparison for positive elements. In Section 6,
we begin the construction of $A_z^C.$ In  Section 7, we show that the construction in Section 6
can be made simple and the Elliott invariant of $A_z^C$ is precisely the same as that of complex
filed just as 
the Jinag-Su algebra ${\cal Z}.$ In Section 8, we show that $A_z^C$ has all expected regularity
properties. 
Moreover, $A_z^C$ is essentially tracially approximated by ${\cal Z}.$
Using $A_z^C,$ we also produce, for each weakly unperforated Elliott invariant, a unital separable
simple non-exact \CA\ $B$ which has the said Elliott invariant,  has stable rank one,
is essentially tracially approximated by \CA s with nuclear dimension 
at most 1, has almost unperforated
Cuntz semigroup,  has strict comparison for positive elements and has no 2-quasitraces which
are not traces.

{\bf Acknowledgement}:
The first named author was 
supported by China Postdoctoral Science Foundation, grant  $\#$ KLH1414009, 
and 
partially supported by an NSFC grant (NSFC 11420101001). 
The second named author was partially supported by a NSF grant (DMS-1954600). Both authors would like to acknowledge the support during their visits
to the Research Center of Operator Algebras at East China Normal University
which is partially supported by Shanghai Key Laboratory of PMMP, Science and Technology Commission of Shanghai Municipality (STCSM), grant \#13dz2260400 and a NNSF grant (11531003).

\section{Preliminary}
In this paper,
the set of all positive integers is denoted by $\N.$
If $A$ is unital, $U(A)$ is the unitary group of $A.$ 

\begin{nota}
Let  $A$ 
be a \CA\ and ${\cal F}\subset A$ be a subset. Let  $\epsilon>0$.
Let $a,b\in A,$
we  write $a\approx_{\epsilon}b$ if
$\|a-b\|< \epsilon$.
We write $a\in_\ep{\cal F},$ if there is $x\in{\cal F}$ such that
$a\approx_\ep x.$
\end{nota}

\begin{nota}
Let $A$ be a $C^*$-algebra and
let $S\subset A$ be a subset of $A.$
%
Denote by
${\rm Her}_A(S)$ (or just $\Her(S),$ when $A$ is clear)
the hereditary $C^*$-subalgebra of $A$ generated by $S.$
Denote by $A^1$ the closed unit ball of $A,$
by $A_+$ the set of all positive elements in $A,$
by $A_+^1:=A_+\cap A^1,$
and by
$A_{sa}$ the set of all self-adjoint elements in $A.$
Denote by $\wtd A$ the minimal unitization of $A.$
When $A$ is unital, denote by $GL(A)$ the set of invertible elements of $A.$
\end{nota}

\begin{nota}
Let $\epsilon>0.$ Define a continuous function
$f_{\epsilon}: {{[0,+\infty)}}
\rightarrow [0,1]$ by
$$
f_{\epsilon}(t)=
\left\{\begin{array}{ll}
0  &t\in {{[0,\epsilon/2]}},\\
1  &t\in [\epsilon,\infty),\\
\mathrm{linear } &{t\in[\epsilon/2, \epsilon].}
\end{array}\right.
$$

\end{nota}

\begin{df}\label{Dcuntz}
Let $A$ be a \CA\
and  let $M_{\infty}(A)_+:=\bigcup_{n\in\mathbb{N}}M_n(A)_+$.
For $x\in M_n(A),$
we identify $x$ with ${\rm diag}(x,0)\in M_{n+m}(A)$
for all $m\in \N.$
Let $a\in M_n(A)_+$ and $b\in M_m(A)_+$.
Define $a\oplus b:=\mathrm{diag}(a,b)\in M_{n+m}(A)_+$.
If $a, b\in M_n(A),$
we write $a\lesssim_A b$ if there are
$x_i\in M_n(A)$
such that
$\lim_{i\rightarrow\infty}\|a-x_i^*bx_i\|=0$.
We write $a\sim_A b$ if $a\lesssim_A b$ and $b\lesssim_A a$ hold.
We also write $a\lesssim b$ and $a\sim b,$
when $A$ is clear.
The Cuntz relation $\sim$ is an equivalence relation.
Set $W(A):=M_{\infty}(A)_+/\sim_A$.
Let $\la a\ra$ denote the equivalence class of $a$.
We write $\la a\ra\leq \la b\ra $ if $a\lesssim_A b$.
$(W(A),\leq)$ is a partially ordered abelian semigroup.
$W(A)$ is called almost unperforated,
if for any $\la a \ra, \la b\ra\in W(A)$,
and for any $k\in\N$,
if $(k+1)\la a\ra \leq k\la b\ra$,
then $\la a \ra \leq \la b\ra$
(see \cite{Rordam-1992-UHF2}).


\end{df}


If $B\subset A$ is a hereditary $C^*$-subalgebra,
$a,b\in B_+$, then $a\lesssim_A b \Leftrightarrow a\lesssim_B b$.



\begin{df}
\label{2-quasitrace}
Denote by $QT(A)$ the set of  {{2-quasitraces}} of $A$ with {{$\|\tau\|=1$}} 
(see {{\cite[II 1.1, II 2.3]{BH}}}) and by
$T(A)$ the set of all tracial states  on $A.$
We will also use $T(A)$ as well as $QT(A)$ for the extensions on $M_k(A)$ for each $k.$
In fact $T(A)$ and $QT(A)$ may be extended to lower semicontinuous traces and 
lower semicontinuous quasitraces on $A\otimes {\cal K}$ (see lines above Proposition 4.2 of \cite{ERS} and  Remark 2.27 (viii) of \cite{BlK}).

Let $A$ be a \CA. Denote by ${\rm Ped}(A)$ the Pedersen ideal of $A$ 
(see 5.6 of \cite{Pedbook}).
Suppose that $A$ is  a $\sigma$-unital simple \CA. Choose $b\in {\rm Ped}(A)_+$ with $\|b\|=1.$
Put $B:=\overline{bAb}={\rm Her}(b).$ Then, by \cite{Br}, $A\otimes {\cal K}\cong B\otimes {\cal K}.$
For each $\tau\in T(B),$ define a lower semi-continuous function
$d_\tau:  A\otimes {\cal K}_+ \rightarrow [0,+\infty]
$,
$x\mapsto \lim_{n\rightarrow \infty}\tau(f_{1/n}(x))$.
The function $d_{\tau}$ is called
the
dimension function
induced by $\tau$.

We say $A$ has strict comparison (for positive elements), if for any $a, b\in A\otimes {\cal K}_+,$ 
$d_\tau(a)<d_\tau(b)$ for all $\tau\in QT(B)$ implies that $a\lesssim b.$

\end{df}

%

\section{Tracial approximation}


%
%
\begin{df}\label{DTrapp}
Let ${\cal P}$ be a class of \CAs\ and let $A$ be a simple \CA.
We say  %
$A$ is essentially tracially in ${\cal P}$
{{(abbreviated as e.~tracially in ${\cal P}$),}}
if for any finite subset ${\cal F}\subset A,$ any $\ep>0,$
any  
$s\in A_+\setminus \{0\},$  
 there exist an element $e\in A_+^1$ and a non-zero \SCA\ $B$ of $A$
which is in ${\cal P}$ such that

(1) 
$\|ex-xe\|<\ep\rforal {{x}} 
\in {\cal F},$

(2)  $(1-e)x\in_{\ep} B$   
and
$\|(1-e)x\|\ge \|x\|-\ep$ for all $x\in {\cal F},$  and

(3)  $e\lesssim s.$



\end{df}


{{
\begin{prop}
\label{f-prop-arbitrary-power}
Let ${\cal P}$ be a class of {{\CA s}} and let
$A$ be a simple \CA. Then $A$ is e.~tracially in ${\cal P}$ if and only if the following hold:
For  any $\ep>0,$
any finite subset ${\cal F}\subset A,$
any  
$a \in 
 A_+\setminus \{0\},$
and any finite subset ${\cal G}\subset C_0((0,1]),$
there exist an element $e\in A_+^1$
and a  non-zero \SCA\ $B$ of $A$
such that $B$ in $\mathcal{P}$, and

(1)  $\|ex-xe\|<\ep$
for all $ x\in\cal F$,

(2) $g(1-e)x\in_\ep B$  for all $g\in {\cal G}$ 
and $\|(1-e)x\|\ge\|x\|-\ep$
for all $x\in {\cal F},$ 
%
and

(3) $e\lesssim a.$

\end{prop}

}}

\begin{proof}
The ``if" part follows easily by taking
${\cal G}=\{\iota\},$ where $\iota(t)=t$ for all  $t\in [0,1].$

{{We now show  the ``only if'' part.}}

Suppose that $A$ is e.~tracially in ${\cal P}.$
{{Let $\ep>0$ and let}}
${\cal F}\subset A$ be a finite subset and,
without loss of generality, we may assume that ${\cal F}\subset A^1.$
{{Moreover, \wilog\, (omitting an error within $\ep/16,$ say), we may further assume that there is $e_A\in A_+^1$
such that
\beq
\label{F-P3.8-1}
e_Ax=x=xe_A.
\eneq}}
Let 
$a\in A_+\setminus \{0\},$
let $\ep>0,$ 
and let ${\cal G}=\{g_1,g_2,...g_n\}\subset C_0((0,1])$ be a finite subset.

By The Weierstrass Theorem,
there is $m\in\N$
and 
polynomials $p_i(t)=\sum_{k=1}^m \beta_k^{(i)}t^k$
such that
\beq
\label{f-prop-3-7-2}
\text{$|p_i(t)-g_i(t)|<\ep/4$ for all $t\in[0,1]$ and all
$i\in\{1,2,...,n\}.$}
\eneq
Let $M=1+\max\{|\beta_k^{(i)}|:i=1,2,...,n, \, k=1,2,...,m\}$ and  $\dt:=
\frac{\ep}{32m^3M}.$

Now, since $A$ is e.~tracially in ${\cal P},$
there exist an element $e\in A_+^1$
and a  non-zero \SCA\ 
$B\subset A$
such that $B$ in $\mathcal{P}$, and

($1$) $\|ex-xe\|<\dt
\rforal x\in {\cal F}\cup\{e_A\},$

($2'$)  $(1-e)x\in_{\dt} B$
and $\|(1-e)x)\|\ge \|x\|-\dt$
for all $x\in {\cal F}\cup\{e_A\},$   and 

($3$)  $e\lesssim a.$

It remains to show that $g_i(1-e)x\in_{\ep/2} B$ for all $x\in {\cal F},$ $i=1,2,...,n.$

Claim: For all $x\in {\cal F}$ and all $k\in\{1,2,...,m\},$
$(1-e)^kx\in_{\frac{\ep}{16mM}} B.$
In fact,
\beq
{{(1-e)^kx\overset{\eqref{F-P3.8-1}}{=}(1-e)^ke_A^{k-1}x
\overset{(1)}{\approx_{k^2\dt}}\overbrace{(1-e)e_A(1-e)e_A\cdots (1-e)e_A}^{k-1}(1-e)x
\overset{(2')}{\in_{k\dt}} B.}}
\eneq
Note that $2k^2\dt\le 2m^2\dt<\ep/16mM.$ The claim follows.





By \eqref{f-prop-3-7-2} and the claim above,
for $x\in{\cal F}$ and $i\in\{1,2,...,n\},$
we have
\beq
\label{f-prop-3-7-1}
g_i(1-e)x\approx_{\ep/4}p_i(1-e)x
=\sum_{k=1}^m\bt_k^{(i)} (1-e)^k x\in_{\ep/4} B.
\eneq




\end{proof}

\begin{rem}

(1)  A similar notion as in Definition \ref{DTrapp}   could also be defined 
for non-simple \CA s.
However, in the present paper, we will be only interested in the simple case.

(2) Note in \ref{f-prop-arbitrary-power}, $g(1-e)$ is an element in ${\wtd A}.$ 
But $g(1-e)x\in A.$  In the case that $A$ is unital,  the  condition $\|(1-e)x\|\ge \|x\|-\ep$
for all $x\in {\cal F}$ in (2)  of the definition is redundant for most cases (we leave 
the discussion to \cite{FLIII}).



(3)   The origin of the notion of tracial approximation were first introduced in \cite{LnTAF} (see also \cite{Linplm}). 
Current definition also related to the notion of  ``centrally 
large subalgebra" in Definition 4.1 of \cite{Phi} (see also 
Definition 2.1 of \cite{ABP}) but not the same.

(4) In \cite{FL} a notion of asymptotically tracial approximation is introduced
which  studies the  phenomena of asymptotic preserving nature.   It also mainly studies the unital simple 
\CA s with rich structure of projections.  It is different from the Definition \ref{DTrapp}.
However,  if $A$ is a unital simple \CA\, which is asymptotically tracially in 
the class ${\cal C}$ of 1-dimensional non-commutative CW complexes, then $A$ is also 
essentially tracially in the same class ${\cal C}.$  Moreover,  many classes  ${\cal P}$ of \CA s  are preserved 
by asymptotically tracial approximation (see Section 4  of \cite{FL}). For these classes ${\cal P},$
of course, a simple \CA\, $A$ is asymptotically tracially in ${\cal P}$ implies 
that $A$ is also 
essentially tracially in ${\cal P}.$  Some more discussion may be found in a forthcoming paper (\cite{FLIII}).

%
\end{rem}



{{\begin{df}
Let ${\cal P}$ be a class of \CAs.
{{The class}}
${\cal P}$ is {{said to  have}} property (H),
if for any nonzero $A$ in ${\cal P}$ and any nonzero hereditary $C^*$-subalgebra
$B\subset A,$ $B$ is also in ${\cal P}.$

\end{df}}}


{\begin{prop}\label{Phered}
Let ${\cal P}$ be a class of \CA s which has property (H).
Suppose that $A$ is  a {{simple}}
\CA\, which is e.~tracially in ${\cal P}.$ Then every
nonzero hereditary \SCA\,
$B\subset A$  is also
e.~tracially in ${\cal P}.$
\end{prop}
\begin{proof}
Assume ${\cal P}$ has property (H) and $A$ is e.~tracially in ${\cal P}.$
Let $B\subset A$ be a nonzero hereditary $C^*$-subalgebra of $A.$
Let ${\cal F}\subset B$ and  {{$s\in B_+\setminus \{0\},$}}
and let $\ep\in(0,1/4).$


Without loss of generality,
we may assume that ${\cal F}\subset B_+^1.$
Let $d\in B_+^1$ {{be}} such that $dx\approx_{{\ep/32}} x\approx_{{\ep/32}} xd$
and $x\approx_{{{\ep/32}}}dxd$ for all $x\in{\cal F}.$

{{Put $\ep_1=\ep/32.$}}
By Lemma 3.3 of \cite{eglnp}, there is $\dt_1\in(0,{{\ep_1}})$
such that for any \CA\ $E$ and any $x,y\in E_+^1,$ if
$x\approx_{\dt_1} y,$ then there is an injective  \hm\,
$\psi: {\rm Her}_{E}(f_{{{\ep_1}}/2}(x))\to {\rm Her}_{E}(y)$
satisfying $z\approx_{{\ep_1}} \psi(z)$
for all $z\in {\rm Her}_{E}(f_{{{\ep_1}}/2}(x))^1.$

Note that there is $\dt_2\in(0,\dt_1)$ such that,
for any \CA\ $E,$ any $x,y\in E_+^1,$
if
$xy\approx_{\dt_2} yx,$
then $x^{1/4}y\approx_{\dt_1/2} yx^{1/4},$
$x^{1/8}y^{1/2}\approx_{\dt_1/2} y^{1/2}x^{1/8},$
and
$x^{1/8}y\approx_{\dt_1/2 
}yx^{1/8}.$


Let $\dt=\min\{\dt_1/2,\dt_2/2\}.$
Let ${\cal G}=\{t,t^{1/4},t^{1/8}\}\subset C_0((0,1]).$
Since $A$ is e.~tracially in ${\cal P},$ 
by Proposition \ref{f-prop-arbitrary-power},
there exist
a positive element
$a\in A_+^1$
and a non-zero $C^*$-subalgebra $C\subset A$ which is in ${\cal P}$ such that

(1) $\|ax-xa\|<{\dt}$ for all $x\in{\cal F}\cup\{d,d^{1/2},d^2\},$

(2) $g(1-a) x\in_{\dt} C$  {{for all $g\in {\cal G}$}} 
and   $\|(1-a)x\| \ge \|a\|-\dt$
for all $x\in{\cal F}\cup\{d,d^{1/2},d^2\},$
and

(3) $a\lesssim s.$ 

\noindent
By (2), there is $c\in C$ such that
$
c\approx_{\dt_1/2} (1-a)^{1/4}d.$ 
By (1) and the choice of $\dt_2$, we have $c\approx_{\dt_1}d^{1/2}(1-a)^{1/4}d^{1/2}.$
%
Then by Lemma 3.3 of \cite{eglnp} {{and by the choice of $\dt_1,$}} there is  a monomorphism
$$\phi: 
{\rm Her}_A(f_{{{\ep_1}}/2}(c))\to
{\rm Her}_A(d^{1/2}(1-a)^{1/4}d^{1/2})\subset B$$
satisfying $\|\phi(x)-x\|<{{\ep_1}}$ for all
$x\in
{\rm Her}_C(f_{{{\ep_1}}/2}(c))^1.$
Define $D:=\phi({\rm Her}_C(f_{{{\ep_1}}/2}(c))
)\subset B.$
Since $C$ is in ${\cal P}$ and ${\cal P}$ has property (H),
we have $D\cong 
{\rm Her}_C(f_{\ep/2}(c))$ is in ${\cal P}.$
Set $b:=dad\in B_+^1.$ Then by (1) and by the choice of $d,$ we have
\beq
\label{ET-her-1}
\|bx-xb\|=\|dadx-xdad\|
\approx_{4{{\ep_1}}}
\|adxd-dxda\|
\approx_{2{{\ep_1}}}
\|ax-xa\|<\dt
\text{ for all }x\in{\cal F}.
\eneq
By (2), for any $x\in{\cal F},$
there is $\bar x\in C$ such that
$(1-a)^{1/4}x(1-a)^{1/4}\approx_{2{{\ep_1}}} \bar x.$
Then
\beq\nonumber
(1-b)x
&=&
(1-dad)x 
\approx_{3{{\ep_1}}} (1-a)dxd
\nonumber\\
&\approx_{4{{\ep_1}}}&
(1-a)^{1/8}d(1-a)^{1/8}
\cdot(1-a)^{1/4}x(1-a)^{1/4}\cdot
(1-a)^{1/8}d(1-a)^{1/8}
\nonumber\\
&\approx_{4{{\ep_1}}}&
c\bar xc
\approx_{2{{\ep_1}}}
(c-{{\ep_1}})_+\bar x (c-{{\ep_1}})_+
\nonumber\\
&\approx_{{{\ep_1}}}&
\phi((c-{{\ep_1}})_+\bar x (c-{{\ep_1}})_+)\in D.
\eneq
{{In other words,
\beq\label{ET-her-2}
(1-b)x\in_{\ep} D.
\eneq}}
{{Therefore,}} for all $x\in {\cal F},$
\beq\nonumber
&&\hspace{-0.8in}\|(1-b)x\|=\|(1-dad)x\|
\ge \|(1-a d^2)x)\|-\dt \\\label{ET-her-2nn}
&&\hspace{0.2in}\ge \|(1-a)x\|-3\ep_1
\ge  \|x\|-\dt-3\ep_1\ge \|x\|-\ep.
\eneq
%
By (3),
we have
$b=dad\lesssim_A 
s.$
Note that $b, s\in B.$ Since $B$ is a hereditary SCA\, we have $b\lesssim_B s.$
By \eqref{ET-her-1} {{and}} \eqref{ET-her-2}, 
we see that $B$ is also e.~tracially in ${\cal P}.$
\end{proof}


\section{Basic properties}

\begin{nota}\label{Nota-WCZ}
Let ${\cal W}$ be the class of  \CA s $A$ such that $W(A)$
is  almost unperforated. 

Let ${\cal Z}$ be the Jiang-Su algebra (see \cite{JS1999}). 
A \CA\ $A$ is called ${\cal Z}$-stable if $A\otimes {\cal Z}\cong A.$
Let ${\cal C}_{\cal Z}$ be the class of separable ${\cal Z}$-stable \CAs. 
\end{nota}

\begin{lem}\label{L410}
\label{tracially-almost-unperforated-lemma}
Let  $A$ be  a  simple \CA\, which  is
{{e.~tracially}} in ${\cal W}$ and
$a,b,c\in A_+\backslash\{0\}.$
Suppose  that there exists $n\in\N$ satisfying $(n+1)\la a\ra\leq n\la b\ra.$
Then, for any $\ep>0$,
there exist $a_1,a_2\in A_+$
such that

(1) $a\approx_{\epsilon}a_1+a_2$,

(2) $a_1\lesssim _A b$, {{and}}

(3) $a_2 \lesssim_A c$.

\end{lem}


\begin{proof}
{{Without loss of generality, {{one}} may assume that
$a,b,c\in A_+^1\backslash\{0\}$ and $\epsilon<1/2$.
%
{{T}}hen $(n+1)\la a\ra\leq n\la b\ra$ implies that there exists
$r=\sum_{i,j=1}^{n+1}r_{i,j}\otimes e_{i,j}\in A\otimes M_{n+1}$
such that
\beq\label{410-n1}
a\otimes \sum_{i=1}^{n+1}e_{i,i}
\approx_{\epsilon/128}r^*\left(b
\otimes \sum_{i=1}^{{{n}}}e_{i,i}\right)r.
\eneq

Set
${\cal F}:=\{a,b\}\cup\{r_{i,j},r_{i,j}^*:i,j=1, 2 ,..., n+1\}.$
Let $M:=1+\|r\|.$
Let $\sg=\frac{\ep}{32M^2(n+1)^4}.$
Since $A$ is {{e.~tracially}} in ${\cal W},$
by Proposition \ref{f-prop-arbitrary-power},
for
any $\dt\in(0, \frac{\ep}{256M(n+1)^2}),$
there exist  {{$f\in A_+^1\setminus \{0\}$ and a \SCA\, $B\subset A$}}
which has almost unperforated
$W(B),$
such that

($1'$) $\|fx-xf \|<\dt$ for $x\in {\cal F},$

($2'$)  $(1-f)^{ 1/4} x,
(1-f)^{ 1/2} a(1-f)^{ 1/2 },
(1-f)^{1/4}x (1-f)^{ 1/4} \in_\dt B$
for all $x\in {\cal F},$

($3'$)
{{$f\lesssim c.$}}

Put $g=1-f.$
Let
$$
{\cal G}:=\{g^{ {1/4}}x, g^{1/2}xg^{1/2}, g^{1/4}xg^{1/4}
: x\in {\cal F}\}.$$
By ($2'$), there is a map $\af:{\cal G}\to B$ such that
$\af({\cal G}\cap A_+)\subset B_+,$
and
\beq\label{L410-nn-1}
 x \approx_{2\dt} \af( x )
\text{ for all }x\in{\cal G}.
\eneq
From ($1'$) and ($2'$),
one can  choose $\dt$
sufficiently small such that,
\beq\label{L410-n1}
&&a\approx_{\ep/16}
g^{{1/2}}ag^{{1/2}}+ (1-g)^{1/2}a(1-g)^{1/2},
{\rm \ and}
\\\label{L410-n1+}
&&(g^{{1/2}}{{a}}g^{{1/2}}-\ep/8)_+
\approx_{\ep/16}
(\alpha(g^{{1/2}}{{a}}
g^{{1/2}})-\epsilon/8)_+.
\eneq
\noindent
{{By}} ($1'$)  and  \eqref{410-n1}
(with $\dt$ sufficiently small),
one can also assume that
\beq
\label{L410-2-f}
g^{{1/2}}ag^{{1/2}}\otimes \sum_{i=1}^{n+1}e_{i,i}
\approx_{\ep/64}
R^* \left(g^{{1/4}} b g^{{1/4}}\otimes  \sum_{i=1}^{n}e_{i,i}\right)R,
\eneq
where  $R:=\sum_{i,j=1}^{n+1}({{g^{1/4}}}{{r_{i,j}}}
)\otimes e_{i,j}.$
By \eqref{L410-2-f}, \eqref{L410-nn-1},
and 
$\dt<\frac{\ep}{256M(n+1)^2},$
one has
\beq
\alpha(g^{{1/2}}{{a}}g^{{1/2}})
\otimes \sum_{i=1}^{n+1}e_{i,i}
\approx_{\ep/32}
{{\bar R^*}}
\left(
\af(g^{1/4}{{b}}g^{1/4})
\otimes  \sum_{i=1}^{n}e_{i,i}\right){{\bar R}},
\eneq
where $\bar R:=
\sum_{i,j=1}^{n+1}\af(g^{1/4}{{r_{i,j}}}
)\otimes e_{i,j}.$
Then by the choice of $\sg,$
\beq
\label{L410-3-f}
\alpha(g^{1/2}{a}g^{1/2})
\otimes \sum_{i=1}^{n+1}e_{i,i}
\approx_{\ep/16}
{{\bar R^*}}
\left(
(\af(g^{1/4}{{b}}g^{1/4})-\sg)_+
\otimes  \sum_{i=1}^{n}e_{i,i}\right){{\bar R}}.
\eneq
By \eqref{L410-3-f} and \cite[Proposition 2.2]{Rordam-1992-UHF2}, one has
\beq
(\alpha(g^{{1/2}}{{a}}g^{{1/2}})
-\epsilon/8)_+\otimes \sum_{i=1}^{n+1}e_{i,i}
\lesssim
(\af(g^{1/4} b g^{1/4})
-\sg)_+\otimes  \sum_{i=1}^{n}e_{i,i}.
\eneq

\noindent
{{Since}} $W(B)$ is almost unperforated,
one obtains
\beq
\label{L410-4-f}
(\alpha(g^{{1/2}}{{a}}g^{{1/2}})
-\epsilon/8)_+
\lesssim (\af(g^{1/4}{{b}}g^{1/4})
-\sg)_+.
\eneq
By \eqref{L410-n1+},
\cite[Proposition 2.2]{Rordam-1992-UHF2},
\eqref{L410-4-f},
and \eqref{L410-nn-1},
it follows that
\beq
\hspace{-0.2in}
(g^{{1/2}}{{a}}g^{{1/2}}-\ep/4)_+
&\lesssim&
(\alpha(g^{{1/2}}{{a}}g^{{1/2}})
-\epsilon/8)_+
\\
&\lesssim& (\af(g^{1/4}{{b}}g^{1/4})
-\sg)_+
\lesssim
g^{1/4}{{b}}g^{1/4}
\lesssim b.
\label{L410-6-f}
\eneq
{{By ($1'$) and the choice of $\dt,$}}
\beq
\label{L410-nn10}
a\approx_{\ep/16} (1-f)^{1/2}a(1-f)^{1/2}+f^{1/2} a f^{1/2}.
\eneq
\noindent
Choose
\beq
&&a_1:=(g^{{1/2}}a g^{{1/2}}-\ep/2)_+
=((1-f)^{1/2}a(1-f)^{1/2}-\ep/2)_+,
\andeqn\\
&&a_2:=f^{1/2}af^{1/2}.
\eneq
Then, {{by \eqref{L410-6-f},}}
one has $a_1\lesssim_A b$.
Note that  {{($3'$) above implies}}
$a_2 \lesssim_A c$.
Thus $a_1$ and $a_2$ satisfy (2) and (3) of the lemma.
By \eqref{L410-nn10} 
$$a\approx_{\ep/16} (1-f)^{{1/2}}a(1-f)^{{1/2}}+f^{{1/2}}af^{{1/2}}
\approx_{\epsilon/2} a_1+a_2.
$$
So (1) of the lemma  also holds and  the lemma follows.}}
\end{proof}



\begin{thm}\label{Tstcomp}
Let $A$ be a simple \CA\, which is
{{e}}.~tracially  in ${\cal W}$ (see Lemma \ref{L410}).
Then $A\in {\cal W}.$
\end{thm}

\begin{proof}
{{We may assume that $A$ is non-elementary.}}
Let $a,b\in M_m(A)_+\setminus \{0\}$ with $\|a\|=1=\|b\|$ for some integer $m\ge 1.$
Let $n\in\mathbb{N}$ and assume
 $(n+1)\la a\ra\leq n\la b\ra.$ To prove the theorem, it suffices to prove that $a\lesssim b.$

Note that,  if $B\in {\cal W},$ then, for each integer $m,$ $M_m(B)\in {\cal W}.$
It follows that
$M_m(A)$ is
{{e.~tracially}} in ${\cal W}.$ To simplify notation, \wilog, one may assume
 $a, b\in A_+.$

By Lemma 4.3 of  \cite{FL}, 
{{${\rm Her}(f_{1/4}(b))_+$}} contains $2n+1$ {{nonzero}} mutually orthogonal   elements
$b_0, b_1,
...,b_{2n}$ such that $\la b_i\ra =\la b_0\ra,$
$i=1, 2,
...,2n.$
\Wlog, we may assume that $\|b_0\|=1.$
If $b_0$ is a projection, choose $e_0=b_0.$
Otherwise, by replacing $b_0$ by $g_1(b_0)$ for some continuous function
$g_1\in C_0((0,1]),$ we may assume that there is a nonzero
$e_0\in A_+$ such that $b_0e_0=e_0b_0=e_0.$
Replacing $b$ by $g(b)$ for some $g\in C_0((0, 1]),$ one may assume that $bb_0=b_0b=b_0.$
Put $c=b-b_0.$
Note that
\beq\label{Tstcomp-1}
ce_0=(b-b_0)e_0=be_0-e_0=b_0e_0-e_0=0=e_0c.
\eneq
Keep in mind that
$b\ge c + e_0,$ $c
\perp e_0,$ and
$2n\la b_0\ra \le {{\la c\ra=\la b-b_0 \ra}}.$
One has
\beq
(2n+2)\la a\ra \le 2n \la b\ra \le 2n(\la b-b_0\ra +\la b_0\ra)\le 2n\la c\ra +\la c\ra=(2n+1)\la c \ra.
\eneq

\noindent
By Lemma \ref{tracially-almost-unperforated-lemma},
for any {{$\ep\in(0,1/2),$}} 
there exist $a_1,a_2\in A_+$ such that

(i) $a\approx_{\epsilon/2}a_1+a_2$,

(ii) $a_1\lesssim_A c,$ and

(iii) $a_2\lesssim_A e_0$.

\noindent
{{By}}  (i), (ii) and (iii), and applying  \cite[Proposition 2.2]{Rordam-1992-UHF2}, one obtains (recall $be_0=e_0b=e_0$)
\beq
(a-\ep)_+\lesssim a_1+a_2 \lesssim  c+e_0\le b.
\eneq
Since this holds for every
{{$\ep\in(0,1/2),$}} 
one concludes that  $a\lesssim b.$

\end{proof}

\begin{cor}\label{CZalmostunp}
Let $A$ be a simple \CA\, which is e.~tracially in ${\cal C}_{\cal Z}.$
Then $W(A)$ is  almost unperforated.

\end{cor}
\begin{proof}
It follows from 
\cite[Theorem 4.5]{Rorjs} and Theorem \ref{Tstcomp}.
\end{proof}

\begin{df}
Let $A$ be a \CA.
{{Let ${\cal T}$ denote the class of \CAs\ $A$ such that, for every
$a\in {\rm Ped}(A)_+\setminus\{0\},$
every 2-quasitrace of $\overline{aAa}$ is a trace.}}
\end{df}


\begin{prop}\label{Pquasit}
{{Let $A$ be a
simple 
\CA\ which is
{{e.~tracially}} in ${\cal T}.$ Then $A$ {{is in}} ${\cal T}.$}}

\end{prop}
\begin{proof}
Fix $a\in {\rm Per}(A)_+^1$ and let $C={\rm Her}(a).$
We will show that every 2-quasitrace of $C$ is a trace.
{{We may assume that $C$ is non-elementary.
Let $\tau\in QT(C).$
Fix $x, y\in  C_{sa}$ with $\|x\|,\,\|y\|\leq 1/2.$
Let $\ep\in(0,1/2).$
Let ${\cal F}:=\{x,y,x+y\}.$
Let $n\in\N$ such that $\ep>1/n.$
By \cite[Lemma 4.3]{FL}, there exist
mutually orthogonal norm one positive elements
$c_1,c_2,..., c_n\in A_+\backslash\{0\},$
such that $c_1\sim c_2\sim...\sim c_n.$
Then $d_{\tau} ( c_1 )\leq 1/n<\ep.$

Let $\dt\in (0,\ep)$ be such that,
for any $d\in C_+^1$ and $z\in C_{sa}^1,$
if $\|[d,z]\|<\dt,$ then
\beq
\label{f-qt=t-1}
z\approx_\ep (1-d)^{1/2}z(1-d)^{1/2}+d^{1/2}zd^{1/2}
\eneq
and (see  \cite[II.2.6]{BH}, note that
$\|[(1-d)^{1/2}z(1-d)^{1/2},d^{1/2}zd^{1/2}]\|$ can be sufficiently small
depending on $\dt$)
\beq
\label{f-qt=t-5}
\tau(z)\approx_\ep \tau((1-d)^{1/2}z(1-d)^{1/2})+\tau(d^{1/2}zd^{1/2}).
\eneq

\noindent
Note that ${\cal T}$ has property (H).
Since $A$ is simple and e.~tracially in ${\cal T},$    by Proposition \ref{Phered},
$C$ is also e~tracially in ${\cal T}.$
There exist an element $e\in C_+^1$ and  a  non-zero
$C^*$-subalgebra
$B\subset  C$ such that $B$ is in $\mathcal{T}$, and

(1)  $\|ez-ze\|<\dt$
for all $z\in\cal F$,

(2) $(1-e)^{1/2}z(1-e)^{1/2}\in_{\dt/2} B$ for all $z\in {\cal F},$
and

(3) $e\lesssim c_1.$

\noindent
We may choose $e_B\in {\rm Ped}(B)_+^1$ such that

($2'$) $(1-e)^{1/2}z(1-e)^{1/2}\in_{\dt} B_1:=\overline{e_BBe_B}$ for all $z\in {\cal F}.$

\noindent
Note that for $z\in{\cal F},$ $e^{1/2}ze^{1/2}$ is self-adjoint. One has
$(e^{1/2}ze^{1/2})_+,(e^{1/2}ze^{1/2})_-\in{\rm Her}_A(e).$
Then
\beq
|\tau(e^{1/2}ze^{1/2})|
&=&
|\tau((e^{1/2}ze^{1/2})_+)-\tau((e^{1/2}ze^{1/2})_-)|
\\
&\leq&
d_\tau((e^{1/2}ze^{1/2})_+)+d_\tau((e^{1/2}ze^{1/2})_-)
\leq 2d_\tau(e)\leq 2\ep.
\label{f-qt=t-2}
\eneq

\noindent
Then by (1), the choice of $\dt,$ \eqref{f-qt=t-1}, and \eqref{f-qt=t-5},
for $z\in{\cal F},$
\beq
\tau(z)
&\approx_{2\ep}&
\tau((1-e)^{1/2}z(1-e)^{1/2})+\tau(e^{1/2}ze^{1/2})
\\
\text{(by \eqref{f-qt=t-2})}&\approx_{2\ep}&
\tau((1-e)^{1/2}z(1-e)^{1/2}).
\label{f-qt=t-4}
\eneq
By ($2'$), there are $\bar x, \bar y\in (B_1)_{sa}$ such that
\beq
(1-e)^{1/2}x(1-e)^{1/2}\approx_{2\dt} \bar x,
\quad
(1-e)^{1/2}y(1-e)^{1/2}\approx_{2\dt} \bar y.
\label{f-qt=t-3}
\eneq
Then
\beq\nonumber
\tau(x+y)
&\overset{\eqref{f-qt=t-4}}{\approx_{{{4\ep}}}}&
\tau((1-e)^{1/2}(x+y)(1-e)^{1/2})
\\\nonumber
&\overset{\eqref{f-qt=t-3}}{\approx_{{4\dt}}}&
\tau(\bar x+\bar y)
\\\nonumber
\text{($\tau$ is a trace on $B_1$)}&=&
\tau(\bar x)+\tau(\bar y)
\\\nonumber
&\overset{\eqref{f-qt=t-3}}{\approx_{{4\dt}}}&
\tau((1-e)^{1/2}x(1-e)^{1/2})+
\tau((1-e)^{1/2}y(1-e)^{1/2})
\\\nonumber
&\overset{\eqref{f-qt=t-4}}{\approx_{{4\ep}}}&
\tau(x)+
\tau(y).
\eneq
Since $\ep$ (and $\dt$) are
 arbitrary small,
we have $\tau(x+y)=\tau(x)+\tau(y).$ therefore $\tau$ is a trace on $C.$}}
\end{proof}

\begin{df}\label{Dfinite}
Let $A$ be a \CA.
Recall that an element
$a\in {\rm Ped}(A)_+$ is said to be
infinite, if there are  nonzero elements $b, c\in {\rm Ped}(A)_+$ such that $bc=cb=0$
and $b+c\lesssim c$ and $c\lesssim a.$
$A$ is said to be finite,
if every element $a\in {\rm Ped}(A)_+$ is not infinite (see, for example, Definition 1.1 of  \cite{LZ}).
$A$ is stably finite, if $M_n(A)$ is finite for every integer $n\ge 1.$

Recall that a simple \CA\, $A$ is purely infinite if and only if every non-zero element in 
${\rm Ped}(A)_+$ is infinite
(see   Condition (vii)  and Theorem 2.2 of \cite{LZ}). 
Let ${\cal PI}$ be the class of 
 \CAs\ that every nonzero positive element is infinite.
 

\end{df}

{{
\begin{thm}\label{Ppurely}
Let $A$ be a simple \CA\, which is e.~tracially in
${\cal PI}.$
Then $A$ is purely infinite.
\end{thm}
\begin{proof}
Note $A$ has infinite dimension. 
Let $a\in {\rm Ped}(A)_+\setminus\{0\}$  with
$\|a\|=1.$

Since $\overline{f_{1/4}(a)Af_{1/4}(a)}$ is infinite dimensional simple \CA,
one may choose  $c, d\in \overline{f_{1/4}(a)Af_{1/4}(a)}_+\setminus \{0\}$
such that  $cd=dc=0.$

Since $A$ is e.~tracially in ${\cal PI},$
there exist a 
 sequence of positive elements $e_n\in A_+$ with $\|e_n\|\le 1$
and a  sequence of  $C^*$-subalgebra $B_n\subset A$ such that
$B$ in ${\cal PI},$ and

(1) $a\approx_{1/2^n} e_n^{1/2}ae_n^{1/2}+(1-e_n)^{1/2}a(1-e_n)^{1/2},$ and

(2) $(1-e_n)^{1/2}a(1-e_n)^{1/2}\in_{1/2^n} B_n,$ and

(3) $e_n\lesssim c.$

By (2), there is  $b_n\in {B_n}_+$ such that $b_n\approx_{1/2^n} (1-e_n)^{1/2}a(1-e_n)^{1/2}.$
Then by (1), 
\beq
a\approx_{2/2^n} b_n+e_n^{1/2}ae_n^{1/2}.
\eneq

Case (i): there exists a subsequence $\{n_k\}$  and $1>\dt>0$ such that
$\inf_k\{\|b_{n_k}\|\}\ge \dt>0.$
Choose $0<\ep<\dt/4.$ 

By 
\cite[Lemma 1.7]{Phi},  for all sufficiently large $k,$ 
we have 
\beq
0\neq(b_{n_k}-2\ep)_+\lesssim
(b_{n_k}+e_{n_k}^{1/2}ae_{n_k}^{1/2}-2\ep)_+\lesssim a.
\eneq

Note  $(b_{n_k}-2\ep)_+\in {\rm Ped}(B)_+\setminus \{0\}.$  Then there are $d_1, d_2\in {\rm Ped}(B)_+\setminus \{0\}$
such that $d_1\perp d_2,$ $d_1+d_2\lesssim d_2$ and 
\beq
d_1+d_2\lesssim (b_{n_k}-2\ep)_+\lesssim a.
\eneq
It follows that $a$ is infinite.

Case (ii):  $\lim_{n\to\infty}\|b_n\|=0.$   This happens only when 
$\lim_{n\to\infty}\|(1-e_n)^{1/2}a(1-e_n)^{1/2}\|=0.$ 
Then,  for any $0<\ep<1/16,$  by 
\cite[Proposition 2.2]{Rordam-1992-UHF2}, (1) and (2), 
for all large $n,$
\beq
f_\ep(a)\lesssim e_nae_n\lesssim c
\eneq
Note that $f_\ep(a)f_{1/8}(a)=f_{1/8}(a).$ Therefore
\beq
c+d\lesssim f_\ep(a)\lesssim c\lesssim a.
\eneq
It also follows that $a$ is infinite.  Therefore $A$ is purely infinite. 
%
\end{proof}}}


\begin{prop}[Corollary 5.1 of \cite{Rorjs}]\label{PRordam-2}
{{Let $A$ be a $\sigma$-unital simple \CA\, such that $W(A)$ is almost unperforated.
If $A$ is not purely infinite, then $\overline{aAa}$ has a
nonzero  {{2-quasitrace}}
for every $a\in {\rm Ped}(A)_+\setminus \{0\}.$
 Consequently $A$ is stably finite.}}

\end{prop}

\begin{proof}

This is a theorem of M. R\o rdam (Corollary 5.1 of \cite{Rorjs}).
Since we do not assume that $A$ is exact and will use only 2-quasitraces, some more explanation is in order.
The explanation follows of course exactly the same lines of the  proof of Corollary 5.1 in \cite{Rorjs}.

Let $a\in {\rm Ped}(A)_+^1$ and $B:=\overline{aAa}.$
Then $B$ is algebraically simple
{{(see, for example, \cite[II.5.4.2]{B-Encyclopaedia2006}).}}
Assume {{that}} $B$ has no {{nonzero}} 2-quasitraces.

Consider $W(B).$   Note that $W(B)\subset W(A)$ and
$W(B)$ has the property that, if $x\in W(B)$ and $y\in W(A)$ such that
$y\le x,$ then $y\in W(B).$  It follows that $W(B)$ is almost unperforated.
Since $B$ is algebraically simple, every element in $W(B)$ is  a strong order unit.

Let $t, t'\in W(B)$ (with $t$  a strong order unit).
The statement (and the proof)
of  Proposition 3.1 of {{\cite{Rordam-1992-UHF2}}} 
imply that if there is no state on $W(B)$ (with the strong order unit $t$),
then, there must be some integer $n\in \N$ and $u\in W(B)$
such that
\beq
n t' +u\le   nt +u.
\eneq
Then, Proposition 3.2 of {{\cite{Rordam-1992-UHF2}}} 
(see the proof also),
as $W(B)$ is almost unperforated,
\beq\label{Prorm2-n1}
t'{{~\leq~}}t.
\eneq
On the other hand, by II. 2.2 of \cite{BH}, every lower semicontinuous dimension
function on $W(B)$ is induced by a {{2-}}quasitrace on $B.$  Since $B$ is assumed to have no
nonzero 2-quasitraces,
combining Proposition 4.1 (as well as paragraph above it) of \cite{Rordam-1992-UHF2},  there is no state on $W(B).$
Therefore \eqref{Prorm2-n1} implies that, for any $b, c\in  B_+\setminus \{0\},$
$b\lesssim c.$ It follows that $B$ is purely infinite and so is $A.$

To see the last of the statement,
{{suppose that there
are $b, c\in {\rm Ped}(A)_+^1\setminus \{0\}$}} such that
$bc=cb=0$ and $b+c\lesssim c.$
Let $a=b+c$ and $B=\overline{aAa}.$
{{Note that $a\in {\rm Ped}(A)_+.$}}
Then $B$  has {{nonzero}} 2-quasitraces.

Therefore
\beq\label{Prom2-1}
d_\tau(c)\ge d_\tau(b+c)\rforal \tau\in QT(B).
\eneq
On the other hand, for any $\tau\in QT(B),$  for any
{{$1>\ep>0,$}}
\beq\label{Prorm2-n2}
\tau(f_{\ep}(b+c))=\tau(f_{\ep}(b)+f_{\ep}(c))=\tau(f_{\ep}(b))+\tau(f_{\ep}(c)).
\eneq
Fix $1>\ep_0>0$ such that $f_{\ep_0}(b)\not=0.$
Since $B$ is algebraically simple, $\tau(f_{\ep_0}(b))>0$ for all 2-quasitraces $\tau.$
Fix $\tau\in QT(B).$ Then, by  \eqref{Prorm2-n2},
\beq
d_\tau(b+c)\ge \tau(f_{\ep_0}(b))+d_\tau(c)>d_\tau(c).
\eneq
This contradicts with \eqref{Prom2-1}.
It follows that no such pairs $b$ and $c$ exist.
Thus $A$ is finite.

Since $M_n(A)$ has the same used property as $A,$  we conclude that $A$ is stably finite.

\end{proof}

\begin{cor}\label{Tstrict}
Let $A$ be a
$\sigma$-unital simple \CA\, such that $A$ is e.~tracially in ${\cal W}.$
Then $A$ has strict comparison.
\end{cor}

\begin{proof}
Fix $a\in A_+\setminus \{0\}$ and let $B:={\rm Her}(a).$ 
As in the proof of 
Proposition \ref{PRordam-2}, every lower semicontinuous dimension function 
on $W(B)$ is induced by a 2-quasitrace of $B.$
Therefore this corollary  follows from \cite[Corollary 4.7]{Rorjs}
and Theorem \ref{Tstcomp} above.
\end{proof}}}

\section{Essentially tracially ${\cal Z}$ stable \CA s}

Recall  from Notation \ref{Nota-WCZ} that ${\cal C}_{\cal Z}$ is the class of separable ${\cal Z}$-stable \CA s.

\begin{thm}\label{TZtocomp}
Let $A$ be a $\sigma$-unital  simple  \CA\, which is  {{e.~tracially}} in ${\cal C}_{\cal Z}.$
Then $A$ is either purely infinite, or stably finite.
Moreover, 
if $A$ is not purely infinite, then
$A$ has strict comparison for positive elements.

\end{thm}

\begin{proof}
It follows from {{Theorem 4.5}} of \cite{Rorjs}  that every \CA\, $B$ in ${\cal C}_{\cal Z}$ has almost unperforated $W(B).$
It follows from  Theorem \ref{Tstcomp} that $W(A)$ is almost unperforated.
By Proposition \ref{PRordam-2}, if $A$ is not purely infinite, then $A$ is stably finite and  has strict comparison
for positive elements.

\end{proof}

\begin{df}\label{Dtrdivisible}
Let $A$ be a  {{simple}} \CA. $A$ is said to  
be tracially approximately divisible,
if, for any $\ep>0,$ {{any}} finite subset ${\cal F}\subset A,$
any element $e_F\in A_+^1$ with
$e_Fx\approx_{\ep/4}x\approx_{\ep/4}xe_F$ {{for all $x\in {\cal F},$}}
 any
$s\in A_+\setminus \{0\},$  and any integer $n\ge 1,$
there are  $\theta\in A_+^1,$   a \SCA\, $D\otimes M_n\subset A$  and  a c.p.c.~map
$\bt: A\to A$
such that

(1)  $x\approx_\ep x_1+\bt(x)$ {{for all $x\in{\cal F}$}},
where $\|x_1\|\le \|x\|,$ $x_1\in {\rm Her}(\theta),$ and

(2) $\bt(x)\in_\ep D\otimes 1_n$ and $e_F\bt(x)\approx_{\ep} \bt(x)\approx_\ep \bt(x)e_F$  for all $x\in {\cal F},$  and


(3) $\theta \lesssim a. $

\end{df}
{

\begin{prop}\label{RDtrdiv}
{\rm (cf. \cite[5.3]{Linplm})}
Suppose that $A$ is a simple \CA\, which satisfies the following conditions:
 for any $\ep>0,$ 
{{any}} 
finite subset ${\cal F}\subset A,$
 any
$s\in A_+\setminus \{0\},$  and any integer $n\ge 1,$
there are  $\theta\in A_+^1$  and  \SCA\, $D\otimes M_n\subset A$
such that

(i)  $\theta x\approx_{\ep} x\theta$
for all $x\in {\cal F},$

(ii) ${{(1-\theta)}}x\in_{\ep} D\otimes 1_n$ for all $x\in {\cal F},$
 and

(iii) $\theta\lesssim s.$ 

\noindent
Then $A$ is tracially approximately divisible.
\end{prop}

\begin{proof}
Let ${\cal F}\subset A$ be a finite subset, $\ep>0$ and an element $e_F\in A_+^1$
such that $e_Fx\approx_{\ep/4} x\approx_{\ep/4} xe_F$
{{for all $x\in {\cal F},$}} 
let  
$s\in A_+\setminus \{0\}$
{{and}} an integer $n$ be given.
\Wlog, we may assume that ${\cal F}\subset A^1.$

Let {{$\dt\in(0,\ep/8)$}} be {{a}} positive number  such that, for any elements $z\in A^1$ and ${{w}}\in A_+^1,$
  $\|zw-wz\|<\dt$ implies that
\beq
{{\|(1-w)^{1/2}z-z(1-w)^{1/2}\|<\ep/8.}}
\eneq
{{P}}ut ${\cal F}_1={\cal F}\,\cup \{e_F\}.$ Suppose that there are $\theta\in A_+^1$ {{and}} $D$ as in the statement of
the proposition such that (i), (ii) and (iii) hold
for $\dt$ (in place of $\ep$) and ${\cal F}_1$ (in place of ${\cal F}$).

Thus (3) in {{Definition}} \ref{Dtrdivisible} holds.

Define $\bt: A\to A$ by $\bt(a)=(1-\theta)^{1/2}a(1-\theta)^{1/2}$ for all $a\in A.$ It is a c.p.c.~map.
Define, for each $x\in {\cal F},$
$x_1:=\theta^{1/2} x\theta^{1/2} \,\in {\rm Her}(\theta).$
Then $\|x_1\|\le \|x\|.$
Note
that, by the choice of $\dt,$ for all $x\in {\cal F},$
\beq
e_F\bt(x)=e_F(1-\theta)^{1/2}x(1-\theta)^{1/2}\approx_{\ep/8}(1-\theta)^{1/2}e_Fx(1-\theta)^{1/2}\approx_{\ep/8} \bt(x)
\approx_{\ep/4} \bt(x)e_F.
\eneq
Moreover,   for all $x\in {\cal F},$
\beq
\bt(x)=(1-\theta)^{1/2}x(1-\theta)^{1/2}
{{
\approx_{\ep/8}
(1-\theta)}}x\in_{\dt} D\otimes 1_n.
\eneq
So (2) in {{Definition}} \ref{Dtrdivisible} holds.
Also, by the choice of $\dt,$  for all $x\in {\cal F},$
\beq
x=\theta x+(1-\theta)x\approx_{\ep/4} \theta^{1/2}x\theta^{1/2}+(1-\theta)^{1/2}x(1-\theta)^{1/2}=x_1+\bt(x).
\eneq
Hence (1) in {{Definition}} \ref{Dtrdivisible} holds.
Thus $A$ is tracially approximately divisible.

\end{proof}

The following lemma is a convenient folklore. 

\begin{lem}\label{Lpositive}
Let $\dt>0.$ There is an integer $N(\dt)\ge 1$ satisfies the following:

For any \CA\ $A,$ any $e\in A_+^1,$ and any $x\in A,$ if
$x^*x\le e$ and $xx^*\le e,$
then
\beq
e^{1/n}x\approx_{\dt} x\approx_{\dt} xe^{1/n}\rforal n\ge N(\dt).
\eneq

\end{lem}

\begin{proof}
Let $\dt>0$ be given.
Choose ${{N(\dt)}}\ge 1$ such that
\beq\label{Posi-1}
\max\{|{{(1-t^{1/n})^2t}}|: t\in [0,1]\}<\dt^2
\rforal n\ge {{N(\dt)}}.
\eneq
Then for any \CA\ $A,$ any $e\in A_+^1,$ and any $x\in A$ satisfying
$x^*x\le e$ and $xx^*\le e,$
\beq
\|(1-e^{1/n})x\|
=
\|(1-e^{1/n})xx^*(1-e^{1/n})\|^{1/2}
\leq
\|(1-e^{1/n})e(1-e^{1/n})\|^{1/2}
<\dt
\eneq
for all $n\geq N(\dt).$
Similarly, we also have $\|x(1-e^{1/n})\|<\dt$ for all $n\geq N(\dt).$
Lemma follows.
\end{proof}


\begin{thm}\label{PherTD}
If $A$ is a simple
\CA\, which is tracially approximately divisible,
then every hereditary \SCA\, of $A$ is also tracially approximately divisible.
\end{thm}

\begin{proof}
Let $B$ be a hereditary \SCA\, of $A,$ ${\cal F}\subset B^1$ be a finite subset,  $\ep>0,$ 
$s\in B_+\setminus \{0\}$
be a  positive element,
and
let  $n\ge 1$ be an integer.

There are $b_0, b_1\in B_+^1$ such
that
\beq\label{PherTD-n1}
b_0b_1=b_1={{b_1b_0}}, \,\, b_1x\approx_{\ep/32} x \approx_{\ep/32} xb_1\rforal x\in {\cal F}.
\eneq
To simplify notation, \wilog, we may assume, by replacing $x$ by $b_1xb_1,$
that
\beq
\label{PherTD-n1-1}
b_0x=xb_0=x, \,\, b_0x^*=x^*b_0=x^*\rforal x\in {\cal F}.
\eneq

\noindent
Let ${\cal F}_1=\{b_0, b_1, b_0^{1/2}, x \in {\cal F}\}.$
Choose $\dt>0$ in Lemma 3.3 of \cite{eglnp} associated with $\ep/64$ (in place of $\ep$) and $\sigma=\ep/64.$
Set $\eta=\min\{\dt/4, \ep/256\}.$

We choose $N:= N(\dt_1)\ge 1$ be as in Lemma \ref{Lpositive}.

Let $0<\dt_1<\eta/2.$    Moreover,
we choose $\dt_1$ sufficiently small such that, if $C_1\subset C_2$ be any pair
of \CA s and $c\in C_2$ with $0\le c\le 1$  and $c\in_{\dt_1} C_1,$ if $0\le c_1, c_2\le 1$ and 
$c_1c_2\approx_{\dt_1} c_2\approx_{\dt_1}c_2c_1,$  then
\beq\label{Pher-nnnnn}
c^{1/N}\in_{\eta} C_1\andeqn  c_1c_2^{1/N}c_1\approx_\eta c_2^{1/N}.
\eneq

\noindent
Since $A$ is tracially approximately divisible, there are  $\theta_a\in A_+^1,$ a \SCA\, $D_a\otimes M_n\subset A$
and {{a}} c.p.c.~map $\bt: A\to A$
such that

(1) $x\approx_{\dt_1} x_1+
\bt(x) $ such that
$\|x_1\|\le 1$ {{and}}  $x_1\in {\rm Her}(\theta_a)$
for all $x\in {\cal F}_1,$

(2)  for all $x\in {\cal F}_1,$ $\bt(x)\in_{\dt_1} D_a\otimes 1_n,$
$$b_0\bt(x)\approx_{\dt_1} \bt(x)\approx_{\dt_1} \bt(x)b_0,\andeqn
b_0^{1/2}\bt(x)\approx_{\dt_1} \bt(x) \approx_{\dt_1} \bt(x)b_0^{1/2},\andeqn$$

(3) $\theta_a\lesssim s.$

\noindent
Choose $d(x)\in (D_a\otimes 1_n)^1$ such that
\beq\label{pher-8}
\|\bt(x)-d(x)\|<5\dt_1/4\rforal x\in {\cal F}_1.
\eneq
Let  $b_2=\bt(b_0)^{1/N}.$
By {{\eqref{PherTD-n1-1}}},
$\bt(b_0)\ge {{\bt(x)^*\bt(x)}}$ and $\bt(b_0)\ge \bt(x)\bt(x)^*$ for all $x\in {\cal F}$ (see, for example,
Corollary 4.1.3 of \cite{BK}).

  By  (2) above and the choice of
  {{$N$}}  and applying Lemma \ref{Lpositive},
\beq\label{pher-9}
b_2\bt(x)=\bt(b_0)^{1/N}\bt(x)
\approx_{\dt_1} \bt(x)\,\,\,\hspace{0.4in}\rforal x\in {\cal F}.
\eneq
Choose $d\in (D_a\otimes 1_n)_+$ such that
\beq
\|d-b_2\|<3\eta/2
\eneq
Then, with $b:=b_0b_2b_0,$ by also \eqref{Pher-nnnnn},
\beq
&&\|d-b\|<5\eta/2
\andeqn\\
\label{pher-12}
&&f_{\ep/64}(d)b\approx_{5\eta/2} f_{\ep/64}(d)d\approx_{\ep/64} d\approx_{5\eta/2} b.
\eneq
By the choice of $\eta,$ applying Lemma 3.3 of \cite{eglnp}, there is an isomorphism
$$
\phi: \overline{f_{\ep/64}(d)(D_a\otimes M_n)f_{\ep/64}(d)}\to \overline{bAb}
\subset B
$$ such that
\beq\label{pher-14}
\|\phi(y)-y\|<\ep/64\|y\|\rforal y\in \overline{f_{\ep/64}(d)(D\otimes 1_n)f_{\ep/64}(d)}.
\eneq
Note that $\overline{f_{\ep/64}(d)(D_a\otimes 1_n)f_{\ep/64}(d)}\cong D_1\otimes 1_n$
for some \SCA\, $D_1\subset D_a.$ Let $D_b=\phi(D_1).$
Define a c.p.c.~map $\af: B\to {{B}}$ by
\beq
\af(y)=b\bt(y)b\rforal y\in {{B}}.
\eneq
Then,  for all $x\in {\cal F},$ by \eqref{pher-12} and \eqref{pher-8},
\beq
\af(x)&=&b\bt(x)b\approx_{2(5\eta/2+\ep/64+5\eta/2)}  f_{\ep/64}(d) d\bt(x)d f_{\ep/64}(d)\\
&&\approx_{5\dt_1/4} f_{\ep/64}(d)d d(x) df_{\ep/64}(d)\in_{\ep/64}  D_b\otimes 1_n.
\eneq
For each $x\in {\cal F},$ let 
$x_1'=b_0x_1b_0.$
Then, for all $x\in {\cal F},$ by (1)  and (2) above,  and by \eqref{pher-9},
\beq
x&\approx_{\dt_1}& b_0(x_1+\bt(x))b_0=x_1' +b_0\bt(x)b_0\\
&\approx_{2\dt_1}& x_1'+\bt(x) \approx_{4\dt}x_1'+b\bt(x)b=x_1'+\af(x).
\eneq
Put $\theta_b=b_0\theta_ab_0.$ Then $x_1'\in \overline{\theta_bB\theta_b}.$
Moreover
\beq
\theta_b\lesssim \theta_a\lesssim s.
\eneq
{{Theorem}} follows.
\end{proof}

{{Recall  that a}} non-unital \CA\ is called almost has stable rank one, if 
for every hereditary $C^*$-subalgebra  $B\subset A,$ $B$ lies in the closure of 
invertible elements of $\wtd B$ (\cite[Definition 3.1]{Rob16}).


\begin{lem}
\label{f-nil}
Let $A$ be a \CA\, {{and}} $n\in \N.$ 
Let $e_1,...,e_n\in A_+$ be mutually orthogonal {{non-zero}} positive elements. 
Assume $d_1,...d_n\in A_+$ such that $d_i\lesssim e_i$ ($i=1,...,n$), 
and $e_i d_j=0$ {{whenever $i\le j$}} and $i,j=1,...,n.$ 
{{Then,}} 
for any $a\in \overline{d_1Ad_1}+...+\overline{d_nAd_n}$ and any $\ep>0,$ 
there are nilpotent elements 
$x,y\in A$ such that $\|a-yx\|<\ep.$
\end{lem}
\begin{proof}
Let $a\in \overline{d_1Ad_1}+...+\overline{d_nAd_n}$ and fix $\ep>0.$ 
Then there exist $a_1,...,a_n\in A$ and $\dt>0$ such that 
$a\approx_{\ep} f_\dt(d_1)a_1f_\dt(d_1)+...+f_{\dt}(d_n)a_nf_\dt(d_n).$
Let $x_1,...,x_n\in A$ such that $x_i^*x_i=f_\dt(d_i)$ and
 $x_ix_i^*\in \overline{e_iAe_i},$ $i=1,...,n$ (see \cite[Proposition 2.4]{Rordam-1992-UHF2}). 
For $i,j\in\{1,...,n\}$ and $i\leq j,$ 
$e_id_j=0$ implies $x_i^*x_ix_jx_j^*=0,$ thus 
\beq\label{f-nil-12}
x_ix_j=0 \quad(i\leq j).
\eneq
%

Claim 1: 
${{(x_1+x_2+\cdots +x_n)^{n+1}=0}}.$ 

Proof of Claim 1: 
{{Note that  $(x_1+x_2+\cdots +x_n)^{n+1}$
is a  sum of $n^{n+1}$ terms with   the form}} 
$x_{k_1}x_{k_2}\cdots x_{k_{n+1}}$ $(k_1,...,k_{n+1}\in \{1,...,n\}).$ 
Assume $x_{k_1}x_{k_2}...x_{k_n+1}\neq 0,$ then $x_{k_i}x_{k_{i+1}}\neq 0$ ($i=1,...,n$).
By \eqref{f-nil-12}, 
it follows that $k_{i+1}\leq k_i-1$  ($i=1,...,n$). {{In particular, $k_{n+1}\le k_n-1.$
Then $k_{n+1}\le k_n-1\le k_{n-1}-2.$  An induction implies that}}
$k_{n+1}\leq k_1-n\leq 0$   {{which leads}} a contradiction. 
Thus {{all  $n^{n+1}$ terms of the form $x_{k_1}x_{k_2}\cdots x_{n+1}$ are}} zero.   It follows that 
$(x_1+x_2+\cdots +x_n)^{n+1}=0.$

Claim 2: $(f_\dt(d_1)a_1x_1^*+...+f_{\dt}(d_n)a_nx_n^*)^{n+1}=0.$ 

Proof of Claim 2: 
Let $y_i=f_\dt(d_i)a_ix_i^*$ ($i=1,...,n$).
For $i\geq j,$  using \eqref{f-nil-12}, we have
\beq 
\label{f-nil-2}
y_iy_j
=
f_\dt(d_i)a_ix_i^*f_\dt(d_j)a_jx_j^*
=
f_\dt(d_i)a_i(x_i^*x_j^*)x_ja_jx_j^*
=0.
\eneq
Then,  {{as in the proof of}} Claim 1, 
we have $(y_1+...+y_n)^{n+1}=0,$ Claim 2 follows.

Let $x=x_1+...+x_n
$
and let $y=y_1+...+y_n=f_\dt(d_1)a_1x_1^*+...+f_{\dt}(d_n)a_nx_n^*
.$ 
Then by Claim 1 and Claim 2, both $x$ and $y$ are nilpotent elements. 
For $i,j\in\{1,...,n\}$ and $i\neq j,$ 
$e_ie_j=0$ implies $x_ix_i^*x_jx_j^*=0,$ thus 
$
x_i^*x_j=0 .
$
Then
$yx=f_\dt(d_1)a_1f_\dt(d_1)+...+f_{\dt}(d_n)a_nf_\dt(d_n)\approx_\ep a.$
\end{proof}

\begin{thm}\label{Tstablerank1}
Let $A$ be a simple \CA\, which is tracially approximately divisible.
Suppose that $A$ is stably finite and $W(A)$ is almost unperforated.
 Then $A$ {{has}} stable rank one if $A$ is unital, or $A$ almost has 
  stable rank one
if $A$ is not unital.

\end{thm}

\begin{proof}
{{We assume that $A$ is infinite dimensional. 
Fix an element $x\in A.$  Fix $\ep>0.$ 
We may assume that $x$ is not invertible. Since $A$ is finite, 
$x$ is not one sided invertible. 
{{To show that $x$ is a norm limit of invertible elements, it suffices to show that $ux$ is a norm limit of invertible elements
for some unitary $u\in {\wtd A}.$}}
{{Note that, since $A$ is simple, ${\wtd A}$ is prime.}}
Thus, by  Proposition 3.2 and Lemma 3.5 of \cite{Rordam-1991-UHF}, 
we may assume that there is 
$a'\in \wtd A_+\backslash\{0\}$  and 
$a'x=xa'=0.$  There is $e\in A_+$ such that $a'ea'\not=0.$ Put $a=a'ea'.$

Let $B_0=\{z\in A: az=za=0\}.$ Then $x\in B_0,$ and $B_0$ is a hereditary 
{{\SCA}}\, of $A.$ {{There is $e_b'\in {B_0}_+$ with $\|e_b\|=1$  such 
that $e_b'xe_b'\approx_{\ep/64}x.$   So $f_{\ep/64}(e_b')xf_{\ep/64}(e_b')\approx_{\ep/16}x.$
Put $e_b=f_{\ep/64}(e_b').$ 
Put $B={\rm Her}(e_b).$ \Wlog, we may further assume that $x\in B.$}}

Since we assume that $A$ is infinite dimensional, $\overline{aAa}$
contains  non-zero positive elements $a_0,a_1$
such that $a_0a_1=0.$

Since $A$ is simple, there is $c\in A$ such that $e_bc(a_1)^{1/2}\not=0$ 
(see the proof of \cite[1.8]{Cuntz77}).

{{Note, since $e_b\in {\rm Ped}(B).$ Then ${\rm Ped}(B)=B$
(see, for example, \cite[II.5.4.2]{B-Encyclopaedia2006}).
It follows that there are $y_1,y_2,...,y_m\in B$  such that
\beq
\sum_{i=1}^m y_i^*e_bca_1c^*e_by_i=e_b.
\eneq
It follows that $\la e_b\ra \le m\la a_1\ra.$
Put $n=2m.$}}

 {{For any  $z_1,z_2,...,z_n\in B_+$ which are 
 $n$ mutually orthogonal and mutually equivalent positive elements,
$$
n\la z_1\ra \le \la  e_b\ra \le  m\la  a_1\ra.
$$}}
Since $W(A)$ is almost unperforated,}}
\beq
z_1\lesssim a_1.
\eneq

Since $B$ is a hereditary $C^*$-subalgebra of $A,$  
by Theorem \ref{PherTD}, $B$ is also tracially approximately divisible.
 There are 
 $b\in B_+^1,$ a \SCA\, $D\otimes M_n\subset B,$ and {{a c.p.c. map}}
 $\bt: A\to A$ such
 that




(1) $x\approx_{\ep/8} x_0+\bt(x),$ where $x_0\in \overline{bAb},$

(2) $\bt(x)\in_{\ep/8} D\otimes 1_n,$ and

 (3) $b\lesssim a_0.$

 Thus, there is
 $x_1\in D\setminus\{0\}$
 such that
 \beq\label{NTstr1-pre-10}
\|x-(x_0+
x_1\otimes 1_n)\|<\ep/4.
\eneq
 %

\noindent
Choose a positive element $d\in D$ 
such that
\beq\label{NTstr1-pre-11}
\|dx_1d-x_1\|<\ep/4.
\eneq
 By the choice of
$n,$  $d\otimes e_{1,1}\lesssim a_1$
(where $\{e_{i,j}\}$ {{forms   a system of  matrix units for}}  $M_n$).

Define $g_1=a_0,$ $g_2=a_1,$  $g_{2+i}=d\otimes e_{i,i}$ ($i=1,...,n-1$). 

Define $h_1=b,$ $h_{1+i}= d\otimes e_{i,i}$ ($i=1,...,n$).  

Note that $h_i\lesssim g_i$ ($i=1,...,n+1$), 
and $g_ih_j=0,$  {{if $i\le j,$}}
and $i,j=1,...,n+1.$ 
Note that $x_0+dx_1d\otimes 1_n\in \overline{h_1Ah_1}+\overline{h_2Ah_2}+...+\overline{h_{n+1}Ah_{n+1}}.$ 
Then, by Lemma \ref{f-nil}, there are nilpotent elements $v,w\in A$ such that 
$x_0+dx_1d\otimes 1_n\approx_{\ep/4}vw.$ 
Choose $\dt>0$ such that $vw\approx_{\ep/4}(v+\dt)(w+\dt).$
Since $v,w$ are nilpotent elements, $v+\dt$ and $w+\dt$ are invertible. 
Then, {{combining  \eqref{NTstr1-pre-10} and \eqref{NTstr1-pre-11}, 
\beq
x\approx_{\ep/4}  x_0+x_1\otimes 1_n
\approx_{\ep/4}  x_0+dx_1d\otimes 1_n\approx_{\ep/2} (v+\dt)(w+\dt)
\in GL(\wtd A).
\eneq}}


Therefore  we have shown that $x\in \overline{GL(\wtd A)}.$ 
Thus, in the case that $A$ is unital, $A$ has stable rank one.  Since, by Theorem \ref{PherTD}, this works
for every hereditary \SCA\, of $A,$ $A$ almost has  
stable rank one in the case that $A$ is not unital.
\end{proof}

\begin{rem}\label{Rdiv=comp}
In \cite{FLIII}, we show that a separable simple \CA\, which is tracially approximately divisible has strict comparison for
positive elements.
So there is a redundancy in the assumption of 
Theorem \ref{Tstablerank1}. Since we do not use this fact here in  this
paper, we leave it for \cite{FLIII}.

\end{rem}

\begin{thm}\label{TTd=TZ}
Let $A$ be a simple \CA.
If  $A$ is e.tracially in ${\cal C}_{\cal Z}$  then $A$
  is  tracially approximately divisible.
\end{thm}

\begin{proof}
We assume that $A$ is infinite dimensional.
Let $A$ be a simple \CA\, which is e.~tracially in ${\cal C}_{\cal Z}.$
Let $\ep>0,$ ${\cal F}\subset A^1$ be a finite subset, $a\in A_+\setminus \{0\}$ and $n\ge 1$ be an integer.
Since $A$ is  infinite dimensional, choose $a_1, a_2\in {\rm Her}(a)_+\setminus \{0\}$ such that $a_1a_2=a_2a_1=0.$

There is $e_A\in A_+^1$ and
$\dt>0$ such that
\beq
f_\dt(e_A)x\approx_{\ep/4} x\approx_{\ep/4} xf_\dt(e_A)\rforal x\in {\cal F}.
\eneq
Note, by Theorem \ref{PherTD},  $A_1:=\overline{f_{\dt/2}(e_A)Af_{\dt/2}(e_A)}$ is also
a ($\sigma$-unital)  simple \CA\, which is e.tracially in ${\cal C}_{\cal Z}$
(as ${\cal C}_{\cal Z}$ has property (H), see Corollary 3.1 of \cite{TW07}).

Note also $f_{\dt/2}(e_A)af_{\dt/2}(e_A)\lesssim a.$
To simplify notation, by replacing $x$ by $f_\dt(e_A)xf_\dt(e_A)$ for all $x\in {\cal F},$ and
by replacing $a$ by $f_{\dt/2}(e_A)af_{\dt/2}(e_A),$  and $a_i$ by $f_{\dt/2}(e_A)a_if_{\dt/2}(e_A)$
($i=1,2$),
\wilog, we may assume
that $x, a, a_1, a_2\in A_1.$
We may also assume, \wilog,
\beq
e_1x=x=xe_1\rforal x\in {\cal F}
\eneq
for some strictly positive element $e_1\in 
A_1^{1}
.$
Note that $f_{\dt/2}(e_A)\in {\rm Ped}(A).$
{{Therefore}}  $A_1$ is algebraically simple and $f_{\dt/2}(e_A)$ is a strictly positive element of $A_1.$
There is an integer $l\ge 1$ and $x_i\in A_1,$ $i=1,2,...,l,$ such that
\beq\label{lZdiv-5}
\sum_{i=1}^l x_i^*a_1x_i=e_1.
\eneq
Set ${\cal F}_1={\cal F}\cup \{e_1\}.$
There exists $\theta_1\in A_+^1$ and a  ${\cal Z}$-stable \SCA\, $B$ of $A_1$
such that

(i) $\|\theta_1x-x\theta_1\|<\ep/64$ and $\|(1-\theta_1)^{1/2}x-x(1-\theta_1)^{1/2}\|<\ep/64$
for all $x\in {\cal F}_1,$

(ii) $(1-\theta_1)^{1/2}x(1-\theta_1)^{1/2}, (1-\theta_1)^{1/2}x, (1-\theta_1)x, x(1-\theta_1), (1-\theta_1)x(1-\theta_1)\in_{\ep/64} B$
for all $x\in {\cal F}_1,$  and

(iii) $\theta_1\lesssim a_2.$
\vspace{-0.06in}
$${\rm{Let}}\,\,\,\,{\cal F}_2=\{(1-\theta_1)^{1/2}x(1-\theta_1)^{1/2}, (1-\theta_1)^{1/2}x,
(1-\theta)x, x(1-\theta_1), (1-\theta_1)x(1-\theta_1): x\in {\cal F}\}.$$
For each $f\in {\cal F}_2,$
fix $b(f)\in B$ such
that $\|b(f)\|\le 1$  and
\beq\label{LZdvi-n1}
\|f-b(f)\|<\ep/32.
\eneq

Let ${\cal G}=\{b(f): f\in {\cal F}_2\}.$
We write $B=C\otimes {\cal Z}.$
Since ${\cal Z}$ is strongly self-absorbing, \wilog, we may assume  that
there is a finite subset ${\cal G}_1\subset C$ such that
${\cal G}=\{y\otimes 1_{\cal Z}: y\in {\cal G}_1\}\subset C\otimes 1_{\cal Z}.$
To further simplify notation, \wilog, we may assume
that there exists a strictly positive element  $e_C\in C$ such that
such
that
\beq\label{lZdiv-5+1}
e_by=y=ye_b\rforal y\in {\cal G}_1,
\eneq
where $e_b=e_C\otimes 1_{\cal Z}.$

For any integer $n,$ choose $m> l$ and $n$ divides $m.$  Let $\psi: M_m\to {\cal Z}$ be
an oder zero c.p.c~map such that
\beq\label{LZdvi-6}
1_{\cal Z}-\psi(1_m)\lesssim \psi(e_{1,1})
\eneq
(see (iv) implies (ii) of  Proposition 5.1 of \cite{RW-JS-Revisited}).
Define $\phi: M_m\to B$ by $\phi(c)=e_C\otimes \psi(c)$ for all $c\in M_m.$
Set
\beq\label{LZdiv-7}
\theta_2:=e_b-\phi(1_m)=e_C\otimes 1_{\cal Z}-e_C\otimes \psi(1_m)=e_C\otimes (1_{\cal Z}-\psi(1_{{m}}))\lesssim e_C\otimes \psi(e_{1,1}).
\eneq
Note that
\beq\label{LZdiv-8}
\theta_2 g=g\theta_2\rforal g\in {\cal G}.
\eneq

Define
$D=\overline{e_CCe_C\otimes \phi(e_{1,1})}$ and $D'$  the \SCA\,
generated by
\beq\label{LZdiv-9}
\{e_Cce_C\otimes \psi(z): c\in C\andeqn z\in M_m\}.
\eneq
Recall that $\psi$ gives a \hm\, $H:
 C^*(\psi(1_m))\otimes M_n\to {\cal Z}$ such 
that $H(f\otimes g)=f(\psi(1_A))H(g)$ 
for all $g\in M_n$ and $f\in C_0({\rm sp}(\psi(1_A))).$
It follows $D'\cong D\otimes M_m.$
Define $\bt_1: A\to A$ by
\beq
\bt_1(y)=(1_{\wtilde{A}}-\theta_2)^{1/2}y(1_{\wtilde{A}}-\theta_2)^{1/2}\rforal  y\in A
\eneq
(where $1_{\wtilde{A}}$ denotes the identity of $\wtilde{A}$  when $A$ is not unital, and
$1_{\wtilde{A}}$ is the identity of $A$ if $A$ has one{{)}}. Note also
$(1-\theta)^{1/2}$ is an element which has the form
$1+f(\theta)$ for $f(t)=(1-t)^{1/2}-1\in C_0((0,1])_+^1.$
If $g=y\otimes 1_{\cal Z}\in {\cal G},$ then (note $y\in {\cal G}_1\subset C$ and see also \eqref{lZdiv-5+1}),
\beq
\hspace{-0.2in}\bt_1(g)&=&(1-\theta_2)g
=g
-e_C\otimes (1_{\cal Z}-\psi(1_m))g\\
&=&(e_C\otimes 1_{\cal Z})g-e_C\otimes(1_{\cal Z}- \psi(1_m))g\\\label{LZdvi-11}
&=&(e_C\otimes \psi(1_m)(y\otimes 1_{\cal Z})=(e_C^{1/2}ye_C^{1/2})\otimes \psi(1_m)\in D\otimes 1_m.
\eneq

Define  a c.p.c map $\bt: A\to A$ by
\beq
\bt(x)=\bt_1((1-\theta_1)^{1/2}x (1-\theta_1)^{1/2})\rforal x\in A.
\eneq
For $x\in {\cal F},$ let $f=(1-\theta_1)^{1/2}x(1-\theta_1)^{1/2}.$   Then, by \eqref{LZdvi-n1},
\beq
\bt(x)=\bt_1((1-\theta_1)^{1/2}x (1-\theta_1)^{1/2})\approx_{\ep/32} \bt_1(b(f))\in D\otimes 1_m.
\eneq

Put $\theta=\theta_1+(1-\theta_1)^{1/2}\theta_2(1-\theta_1)^{1/2}.$
We have
\beq
0\le \theta\le \theta_1+(1-\theta_1)^{1/2}(1-\theta_1)^{1/2}=1.
\eneq

For $x\in {\cal F},$ let $f'=(1-\theta_1)x.$
Recall that we assume that $b(f')=y'\otimes 1_{\cal Z}$ for some $y'\in C^1.$ Then, for $x\in {\cal F},$
applying  \eqref{LZdvi-n1} and  \eqref{LZdiv-8} repeatedly,
\beq
(1-\theta)x&=&(1-\theta_1)x- (1-\theta_1)^{1/2}\theta_2(1-\theta_1)^{1/2} x\\
&\approx_{\ep/32} &(1-\theta_1)x- (1-\theta_1)x\theta_2=(1-\theta_1)x(1-\theta_2)\\
&\approx_{\ep/32}& b(f')(1-\theta_2)=(1-\theta_2)^{1/2} b(f')(1-\theta_2)^{1/2}\\
&=&\bt_1(b(f'))\approx_{\ep/32} \bt_1((1-\theta_1)^{1/2}x(1-\theta_1)^{1/2})=\bt(x).
\eneq
Therefore
\beq
x\approx_{\ep} \theta^{1/2} x\theta^{1/2}+\bt(x)\rforal x\in {\cal F}.
\eneq

Note
that,  by \eqref{LZdvi-6} and \eqref {lZdiv-5}, in $W(B),$
\beq
  m\la \theta_2\ra &=&m\la e_C\otimes (1_{\cal Z}- \psi(1_m))\ra\\
   &\le& m\la e_C\otimes \psi(e_{1,1})\ra \le
  \la e_C\otimes \psi(1_m)\ra\le
\la e_C\otimes 1_{\cal Z}\ra\le l \la a_1 \ra.
\eneq
By Theorem 4.5 of {{\cite{Rorjs},}} 
$W(B)$ is almost unperforated. Therefore, since $l<m,$
\beq\label{LZdiv-14}
\theta_2\lesssim a_1.
\eneq
It follows that (note $a_1a_2=a_2a_1=0$)
\beq\label{LZdiv-15}
\theta=\theta_1+(1-\theta_1)^{1/2}\theta_2(1-\theta_1)^{1/2}\lesssim a_2+a_1\lesssim a.
\eneq
Finally, the theorem 
follows from  \eqref{LZdiv-8}, \eqref{LZdvi-11} and \eqref{LZdiv-15},
 and the fact that $D\otimes 1_n$ embedded into $D\otimes 1_m$ unitally (as $n$ divides $m$).

\end{proof}

\begin{cor}\label{CTR-str1}
Let $A$ be a simple \CA\, which is e.tracially in ${\cal C}_{\cal Z}.$
If $A$ is not purely infinite, then $A$ has
stable rank one, if $A$ is unital, and $A$ almost has 
stable rank one,
if $A$ is not unital.
\end{cor}

\begin{proof}
By Theorem \ref{TTd=TZ}, $A$ is tracially approximately divisible.  By Theorem \ref{TZtocomp},
if $A$ is not purely infinite, then $A$ has strict comparison for positive elements.
It follows then form Theorem \ref{Tstablerank1} that $A$ has stable rank one, if $A$ is unital, and $A$ almost  has 
stable rank one,
if $A$ is not unital.

\end{proof}

\begin{rem}
{{For the rest of this paper, we will present non-amenable examples of \CA s which are possibly stably projectionless 
and are essentially tracially in the class ${\cal C}_{\cal Z},$ the class of ${\cal Z}$-stable \CA s.
}}
\end{rem}

\section{Construction of $A_z^C$}

In this section, we first fix  a separable  residually finite dimensional (RFD) \CA\, $C$ which may not be
exact.

Let $B$ be the unitization of $C_0((0,1], C).$
Since $C_0((0,1], C)$ is contractible,  $V(B)=\N\cup \{0\},$
$K_0(B)=\Z$ and $K_1(B)=\{0\}.$

Let us make the convention that $B$ includes
the case that $C=\{0\},$ i.e.,
$B=\C.$

Let $\p=p_1^{r_1}\cdot p_2^{r_2}\cdot \cdots $ be a supernatural number,
where $p_1, p_2,...,$ is a sequence (could be finite)  of distinct prime numbers and $r_i\in \N\cup\{\infty\}.$
 Denote by $\D_\p$ the subgroup
of $\Q$  generated by  finite sums of  rational numbers of the form
${\frac{m}{p_j^i}},$ where $m\in \Z$ and $i\in \N\cap [1, r_j].$
Suppose that $\q$ is another supernatural number.
Then we may identify $\D_\p\otimes \D_\q$ with $\D_{\p\q}.$

Denote by $M_\p$ the UHF-algebra
%
associated with
the supernatural number $\p.$

The following is a result of M. {D\u{a}d\u{a}rlat \cite{D2000}.

\begin{thm}\label{Tuhf}
Fix a supernatural number $\fd.$
There is a unital simple \CA\, $A_\fd$ which
is an inductive limit of $M_{m(n)}(B)$ such
that $(K_0(A_\fd), K_0(A_\fd)_+, [1_{{A_\fd}}
])=(\D_\fd, {\D_\fd}_+,1)$
and $K_1(A_\fd)=\{0\},$ $A_\fd$ has a unique tracial state and
$A_\fd$ has tracial rank zero.


\end{thm}

\begin{df}\label{Dapfq}
{\rm
Fix a supernatural number $\fd.$
One may write $A_\fd=\lim_{n\to\infty}(M_{d_n'}(B), \dt_n'),$ where
$d_{n+1}'=d_n\cdot d_n',$ $d_n, d_{n}'>1$ are integers, where
$\dt_n': M_{d_n'}(B)\to M_{d_{n+1}'}(B)$ is defined by
\beq\label{Ddtn}
\dt_n'(f)=\begin{pmatrix}  f & 0\\
                                      0 & \gamma_n(f)\end{pmatrix} \rforal f\in M_{d_n'}(B),
\eneq
and where $\gamma_n: B\to M_{d_{n}-1}$ is a $d_{n}-1$ dimensional representations (we then
use $\gamma_n$ for the extension $\gamma_n\otimes \id_{d_n'}: M_{d_n'}(B)\to M_{(d_{n}-1)d_n'}$).
We may also assume that
\beq\label{Dapfq-2}
\lim_{n\to\infty}d_{n}=\infty.
\eneq
This is possible, as we  may choose  a subsequence $\{d_n\}$ (and $d_{n+1}'$) 
and the fact 
that $f(0)\in \C$ for $f\in B$ (see the proof of Proposition 8 of \cite{D2000}).
Also, we assume, for any $n,$ $\{\gamma_m: m\ge n\}$ is a separating sequence
of finite dimensional representations.
For more specific construction of $A_\fd,$ the reader is referred to
{{\cite{D2000}}}.

It is important that, for any $\tau\in T(B),$
\beq\label{Dapfq-3}
\lim_{n\to\infty}|\tau\circ \dt_n'(a)-\tau(\gamma_n(a))|=0\rforal a\in M_{d_n'}(B).
\eneq
(Note, by $\tau,$ we mean $\tau\otimes {\rm tr}_{d'_{n}},$ where
${\rm tr}_{d_{n}'}$ is the tracial state of $M_{d_{n}'}.$)

Consider $\dt_{m,n}':={\dt}_{n-1}'\circ\dt_{n-2}'\circ \cdots    \circ \dt_m': M_{d_m'}(B)\to M_{d_n'}(B).$
Then, we may write
\beq
\dt_{m,n}'(f)=\begin{pmatrix}  f & 0\\
                                      0 & \gamma_{m,n}(f)\end{pmatrix} \rforal f\in M_{d_m'}(B),
\eneq
where $\gamma_{m,n}: B\to M_{d_n'/d'_m-1}
$
is a finite
dimensional representation (in what follows throughout the rest of the paper, we  also use $\gamma_{m,n}:=
\gamma_{m,n}\otimes \id_{d_m}: M_{N}(B)\to M_{(d_n'/d'_m
-1)N}$ for all integers $N\ge 1$).
Therefore, if we fix a finite subset ${\cal F}_m\subset M_{d_m'}(B),$
we may assume that
$\gamma_{m,n}(a)\not=0$ for some large $n\ge m.$  In fact, by first adding $|a|\in {\cal F}_m,$ and then considering $f_a(|a|)$
for a continuous function $f_a\in C_0((0,\|a\|),$ we may assume that
\beq\label{Dapfq-4}
\|{\gamma}_{m,n}(a)\|\ge (1-1/m)\|a\|.
\eneq
In fact, for any fixed $m$ and any nonzero element $a\in M_{d_m'}(B)\setminus\{0\},$
there is $n\ge m$ such that
\beq
\|{\gamma}_{m,n}(a)\|\not=0.
\eneq


\noindent
In what follows that $A_\fd=\lim_{n\to\infty}(M_{d_n'}(B), \dt_n')$ is the \CA\, in Theorem \ref{Tuhf}
and $\dt_n'$ as described in \eqref{Ddtn} such that \eqref{Dapfq-4} holds once ${\cal F}_m$ is chosen.

}
\end{df}


We wish to construct a unital simple \CA\, $A_z$
with a unique tracial state such that
$K_0(A_z)=\Z$ and $K_1(A_z)=\{0\}.$


The strategy is to have a Jiang-Su style inductive limit
of some \SCA s of\linebreak
 $C([0,1], M_p(B)\otimes M_q(B))$ for some nonnuclear  RFD algebra $B,$
or perhaps  some \SCA\,  of $C([0,1], M_{pq}(B)).$
However, there were several difficulties to resolve.
One should avoid to use
$M_p(B)\otimes M_q(B)$ 
as building blocks as there are different $C^*$-tensor products and potential
difficulties to compute the $K$-theory.
Other issues include the fact that, for  each $t\in [0,1],$ $M_m(B)$ is not simple.

We begin with the following building blocks.


\begin{df}\label{DDdmdrop}
Define, for a pair of integers $m,k\ge 1,$
\beq\nonumber
E_{m,k}=\{f\in C([0,1], M_{mk}(B)):
f(0)\in M_{m}(B)\otimes 1_{k} \andeqn f(1)\in 1_{m}\otimes M_{k}\}.
\eneq
Note here one views $M_m(B)\otimes 1_k, 1_m\otimes M_k\subset M_m(B)\otimes M_k\subset M_{mk}(B)$
as \SCA s.

\end{df}

Fix integers $m,n\ge 1.$
Let $D(m,k)=M_m(B)\oplus M_k.$
Define
$\phi_0: D(m,k)\to M_m(B)\otimes 1_k$ by $\phi_0((a,b))=a\otimes 1_k$ for all $(a,b)\in D(m,k)$
and $\phi_1: D(m,k)\to M_k$ by $\phi_1((a,b))=1_m\otimes b.$

Then
\beq\label{DEmkmaptorus}
E_{m,k}\cong \{(f,g)\in C([0,1], M_{mk}(B))\oplus D(m,k): f(0)=\phi_0(g)\andeqn f(1)=\phi_1(g)\}.
\eneq
Denote by $\pi_e: E_{m,k}
\to D(m,k)$ the quotient map
which maps $(f,g)$ to $g.$  Denote by $\pi_0: E_{m,k}\to M_m(B)\otimes 1_k$  the \hm\, defined by
$\pi_0((f,g))=\phi_0(g)=f(0)$ and $\pi_1: E_{m,k}\to 1_m\otimes M_k$ the \hm\,
defined by $\pi_1((f,g))=\phi_1(g)=f(1).$

\begin{lem}\label{LLprojectionless}
If $m$ and $k$ are relatively prime, then
$E_{m,k}
$ has no proper projections.

\end{lem}

\begin{proof}
Recall that $B=C_0\widetilde{((0,1], C)}$,  the unitization of  $C_0((0,1], C).$
Let $\tau_B$ be the tracial state on $M_m(B)$ induced by 
the quotient map $B\to B/C_0((0,1],C)\cong \mathbb{C},$
and let ${\rm tr}_k$ be the tracial state of $M_k.$
Let $\tau=\tau_B\otimes {\rm tr}_k.$

 Let $e\in E_{m,k}
 $ be a nonzero projection.  Note $E_{m,k}
 \subset C([0,1], M_{mk}(B)).$
 Note that $K_0(B)=\Z$ and $1_B$ is the only nonzero projection of $B.$
  Then, for each $x\in [0,1],$
$e(x)$ is a nonzero projection in $M_{mk}(B).$  One easily shows that
$\tau(e(x))$ is a constant function on $[0,1].$  Let $\tau(e(x))=r\in (0,1].$
But $\tau(e(0))\in \{i/m: i=0,1,...,m\}$ and $\tau(e(1))\in \{j/k,i=0,1,...,k\}.$
Since $m$ and $k$ are relatively prime, $\tau(e(0))=\tau(e(1))=1.$  Hence
$\tau(e(x))=1$ for all $x\in [0,1].$
This is possible only when $e=1_m\otimes 1_k.$

\end{proof}

\begin{lem}\label{LKofApnn}
Suppose that $m$ and $k$ are relatively prime. Then
$$
(K_0(E_{m,k}), K_0(E_{m,k})_+, [1_{E_{m,k}}])=(\Z, \N\cup \{0\}, 1)\andeqn
K_1(E_{m,k})=\{0\}.
$$
\end{lem}

\begin{proof}

Let
$$
I=\{f\in E(m,k): f(0)=f(1)=0\}.
$$
Then $I\cong C_0((0,1))\otimes M_{mk}(B):=S(M_{mk}(B)).$
It follows that
\beq
K_0(I)=K_1(M_{mk}(B))=\{0\}\andeqn K_1(I)=K_0(M_{mk}(B))=\Z.
\eneq
Consider the short exact sequence
\beq\label{UCTlater}
0\to I\stackrel{\iota_I}{\longrightarrow}  E_{m,k}\stackrel{\pi_e}{\longrightarrow} D(m,k)\to 0,
\eneq
where $\iota_I: I\to E_{m,k}$ is the embedding and
$\pi_e: E_{m,k}\to D(m,k)$ is the quotient map.
One obtains the following  six-term exact sequence
\beq
\begin{array}{ccccc}
K_0(I) & \stackrel{{\iota_I}_{*0}}\longrightarrow & K_0(E_{m,k}) & \stackrel{{\pi_e}_{*0}}{\longrightarrow} & K_0(D(m,k))\\
\uparrow_{\dt_1} & & & & \downarrow_{\dt_0}\\
K_1(D(m,k))) &  \stackrel{{\pi_e}_{*1}}{\longleftarrow} & K_1(E_{m,k}) & \stackrel{{\iota_I}_{*1}}{\longleftarrow} & K_1(I)\\
\end{array}
\eneq
which becomes
\beq
\begin{array}{ccccc}
0 & \stackrel{{\iota_I}_{*0}}\longrightarrow & K_0(E_{m,k}) & \stackrel{{\pi_e}_{*0}}{\longrightarrow} & \Z\oplus \Z\\
\uparrow_{\dt_1} & & & & \downarrow_{\dt_0}\\
0 &  \stackrel{{\pi_e}_{*1}}{\longleftarrow} & K_1(E_{m,k}) & \stackrel{{\iota_I}_{*1}}{\longleftarrow} & \Z\\
\end{array}
\eneq
Note that
$$
{\rm im}({\pi_e}_{*0})=\{(x, y)\in K_0(D(m,k)): {\phi_0}_{*0}(x)={\phi_1}_{*0}(y)\}.
$$
The lemma follows a  straightforward  computation.

Alternatively, 
let 
\beq
Z_{m,k}=\{f\in C([0,1], M_{mk}): f(0)\in M_{m}\otimes 1_{k}\andeqn f(1)\in 1_{m}\otimes M_{k}\}.
\eneq
We may view $Z_{m,k}\subset E_{m,k}$ by identifying $M_{mk}$  and $M_m$  with obvious unital \SCA s of $M_{mk}(B)$ 
and $M_m(B),$ respectively. 

For each $s\in [0,1],$ define $\theta_s: M_{mk}(B)\to M_{mk}(B)$ by
$\theta_s(f)(x)=f((1-s)x)$ for $x\in (0,1]$  and  $f\in M_{mk}(B)=M_{mk}(C_0((0,1], C)^{\widetilde{\,}}\,).$
Note that $\theta_1(f)(x)=f(0)\in M_{mk}$ for all $f\in M_{mk}(B).$
For each $s\in [0,1],$   define a  unital \hm\   {{$\Phi_s: E_{m,k}\to E_{m,k}$  by}}
\beq
\Phi_s(a)(t)=\theta_s(a(t))\rforal a\in E_{m,k}\subset C([0,1], M_{mk}(B))\andeqn t\in [0,1].
\eneq
Then $\Phi_0={\rm id}_{E_{m,k}}$ and $\Phi_1$ maps $E_{m,k}$ into $Z_{m,k}$ and 
${\Phi_1}|_{Z_{m,k}}={\rm id}_{Z_{m,k}}.$ 
In other words, $E_{m,k}$ is homotopic to $Z_{m,k}.$ 
It is known, when $m$ and $k$ are  relatively prime (see, for example, 
Lemma 2.3 of  \cite{JS1999}), and easy to check that
\beq\nonumber
(K_0(Z_{m,k}), K_0(Z_{m,k})_+, [1_{Z_{m,k}}], K_1(Z_{m,k}))=(\Z, \N\cup \{0\}, 1, \{0\}).
\eneq
%
\end{proof}

Let $I$ be the ideal in the proof of Lemma \ref{LKofApnn}. Then $I\cong C_0((0,1))\otimes M_{mk}(B).$
Let $\tau\in T(E_{m,k}).$ Then it is well known that $\tau|_{I}(f)=\int_{(0,1)} \sigma_t(f(t))d\mu$
for all $f\in C_0((0,1))\otimes M_{mk}(B),$ where $\sigma_t$ is a tracial state of $M_{mk}(B)$ and 
$\mu$ is a Borel measure on $(0,1)$ (with $\|\mu\|\le 1$). Since $E_{m,k}/I=M_m(B)\oplus M_k,$
as in  2.5 of \cite{Lncrell},
one may write
\beq\label{65+}
\tau(f)=\int_0^1 \sigma_t(f(t))d\nu\rforal f\in E_{m,k},
\eneq
where $\sigma_0$ is a tracial state on $M_m(B),$ $\sigma_1$ is a tracial state on $M_k,$ and 
$\nu$ is a probability Borel measure on $[0,1]$ (and $\nu|_{(0,1)}=\mu$).

\begin{nota}\label{NNotah}
Let $\gamma: B\to M_r$ be a finite dimensional representation
with rank $r,$ i.e., $\gamma$ is a finite direct sum of irreducible representations $\gamma_j: j=1,2,...,l,$
each of which has rank $r_j$ ($1\le j\le l$) such that $r=\sum_{j=1}^l r_j.$
We will also use $\gamma$ for $\gamma\otimes \id_m:M_m(B)\to M_{rm}$ for all integers $m\ge 1.$
In what follows we may also write $M_L$ for $M_L(\C\cdot 1_B)$ for all integers $L\ge 1.$
In this way $\gamma$ ($\gamma\otimes \id_m$) is a \hm\, from $M_m(B)$ into $M_{rm}\subset M_{rm}(B).$

Let $\xi_0, \xi_1,\xi_2,...,\xi_{{k-1}}: [0,1]\to [0,1]$ be continuous paths.
Define a \hm\,
$$H:  C([0,1], M_{mn}(B))\to C([0,1], M_{((k-1)r+1)mn}(B))\,\,\,{\rm  by}$$
\beq
H(f)=\begin{pmatrix} f\circ \xi_0  & 0  &\cdots  & 0\\
                                   0 & \gamma(f\circ \xi_1)  &\cdots  &0\\
                                  \vdots  &  \vdots & & \vdots\\
                                   0 &  0 & \cdots    & \gamma (f\circ \xi_{k-1})\end{pmatrix}
                                    \rforal f\in\, C([0,1], M_{mn}(B)).
                                    \eneq

Note that $H$ can be also defined on $E_{m,n}\subset C([0,1], M_{mn}(B)).$
But, in general, $H$ does not map $E_{m,n}$ into $E_{m,n}.$ However, with some restrictions on
the boundary (restriction on $\xi_i$'s),  it is possible  that $H$ maps $E_{m,n}$ into $E_{m,n}.$


\vspace{0.1in}
For the convenience of the  construction in the next lemma,
let us add some notation and terminologies.

Let $f, g\in M_n(B).$ We write $f=^s g,$  if there is a scalar unitary $w\in M_n$ such
that $w^*fw=g.$ Also, if $f, g\in C([0,1], M_n(B)),$ we write $f=^s g,$
if there is a unitary $w\in C([0,1], M_n)$ such that $w^*fw=g.$

\end{nota}

\begin{lem}\label{LCa}
There exists  an inductive limit
$A=\lim_{n\to\infty} (A_n, \phi_n),$
where each $A_n=E_{p_n, q_n}$ with $(p_n, q_n)=1,$ such that
each connecting map
$\phi_m: A_m\to A_{m+1}$
is  a unital  injective \hm\ of the form:
\beq\label{LCa-0}
\hspace{-0.2in}\phi_{m}(f)=u^*\begin{pmatrix} \Theta_{m}(f)  & 0  &\cdots  & 0\\
                                   0 & \gamma_{m}(f\circ \xi_1(t))\otimes 1_5  &\cdots  &0\\
                                  \vdots  &  \vdots & & \vdots\\
                                   0 &  0 & \cdots    & \gamma_m(f\circ \xi_k(t))\otimes 1_5
                                   \end{pmatrix} u
                                    \, \tforal f\in A_m,
                                     \eneq
where $u\in U(C([0,1], {{M_{p_{m+1}q_{m+1}}))}},$
$\Theta_{m}:
 A_m\to   C([0,1], M_{R(m,0)p_mq_m}(B))$ is a \hm,
$R(m,0)\ge 1$ is an integer,  $t\in [0,1],$ and  $\gamma_m: B\to M_{R(m)}$ is a
finite dimensional representation.   Moreover,

(1) each $\xi_i: [0,1]\to [0,1]$  is a continuous map which  has one of the following three forms:
\beq
&& \xi_i(t)=\begin{cases} (2/3)t & \text{if}\,\, t\in [0,3/4],\\
                                   1/2 & \text{if}\,\, t\in (3/4,1]\end{cases}, \,\,\,\,\,\,\,\,\,\, \,\,\,\,\xi_i(t)=1/2\tforal t\in [0,1],\andeqn\\
              &&  \xi_i(t)=\begin{cases} 1/2+(2/3)t & \text{if}\,\, t\in [0,3/4],\\
                                                                                                                      1 & \text{if}\,\, t\in (3/4,1],\end{cases}
\eneq
and each type  of $\xi_i$ appears in \eqref{LCa-0} at least once,

(2) $R(m,0)/5kR(m)<1/3^{m},$

(3) for a fixed finite subset ${\cal F}_m\subset A_m\setminus \{0\}\subset C([0,1], M_{p_mq_m}(B)),$\
     $$\|\gamma_m(f(t))\|> (1-1/2m)\|f\| \not=0\,\,\,{\rm for\,\, some}\,\,t \in [0,1] 
     ,$$

     (4) \beq\label{LCa-theta}
\Theta_m(f)=\diag(\theta^{(1)}(f) ,\cdots, \theta^{(k)}(f)),
\eneq
where $\theta^{(i)}: E_{p_m, q_m}\to C([0,1], M_{p_mq_m}(B))$ is defined by, if $\xi_i(3/4)=1/2,$ 
for each $f\in E_{p_m, q_m},$
$$
\theta^{(i)}(f)(t)=\begin{cases} f(\xi_i(t)) & \text{if}\,\,t\in [0,3/4],\\
                                                \theta_{4(t-3/4)}(f(\xi_i(t))) 
                                                & \text{if}\,\,t\in (3/4,1],\end{cases}
$$
and where
 $\theta_t: M_{p_mq_m}(B)\to M_{p_mq_m}(B)$ \rm{(}recall $B=\widetilde{C_0((0,1], C)}
 $ \rm{)} is a unital \hm\, defined by
\beq\label{LCa-3-0-0}
\theta_t(f)(x)=f((1-t)x)\tforal x\in (0,1]\andeqn {\text{for\,\,all}}\,\,t\in [0,1],
\eneq
and,  if  $\xi_i(3/4)=1,$ {\Green{for each $f\in E_{p_m,q_m},$}} 
$$
\theta^{(i)}(f)(t)=f(\xi_i(t)) \,\,\,{\Green{t\in [0,1].}}
$$

\end{lem}

\begin{proof}
The construction will be by induction.
Set $A_1=E_{3, 5}.$

Denote by ${\bar{\mathfrak{3}}}$
the supernatural number
$3^\infty.$
Write $A_{{\bar{\mathfrak{3}}}}=\lim_{n\to\infty} (M_{d_n'}(B), \dt_n')$
(see Theorem \ref{Tuhf}),
where
\beq\label{LCa-dapfq-2}
\dt_n'(f)=\begin{pmatrix}  f & 0\\
                                      0 & \gamma_n(f)\end{pmatrix} \rforal f\in M_{d_n}(B)
\eneq
as in  \eqref{Ddtn}  which also has the properties \eqref{Dapfq-2}, \eqref{Dapfq-3}
(with $d_n=1/3^{l}$ for some integer $l\ge 1$).
Hence, \wilog, we may assume, for all $n,$ 
\beq\label{LCa-3}
{1\over{d_n}} <1/3^n.
\eneq

Recall that $B=C_0(\widetilde{(0,1], C)}.$
Denote, for each $t\in [0,1],$  by $\theta_t: B\to B$
the \hm\, defined by, for all $f\in B,$
\beq\label{LCa-3-0}
\theta_t(f)(x)=f((1-t)x)\rforal x\in (0,1].
\eneq
Note also, for any integer $l\ge 1,$ we will use $\theta_t$ for
$\theta_t\otimes \id_{l}: M_{l}(B)\to M_{l}(B).$
Thus, if $f\in M_l(B),$
\beq\label{LCa-3+}
\theta_1(f)=f(0)\in  M_l.
\eneq
It should be noted that $\theta_0=\id_{M_l(B)}.$

To avoid potential complication of computing  relative primitivity of  integers,
we will have three-stage construction.

Stage 1:
Write $A_m=E_{p_m, q_m},$ where $(p_m, q_m)=1.$ Also
$(5, p_m)=1$ and $(3, q_m)=1.$

Fix any finite subset ${\cal F}_m\subset E_{p_m,q_m}\setminus \{0\}.$
One can choose a finite subset $S\subset [0,1]$ such that, for any $f\in {\cal F}_m,$ there is  $s\in S,$
$
\|f(s)\|> (1-1/2m)\|f\|\not=0.
$
Note that\\ ${\cal F}'=\{f(s): f\in {\cal F}_m\andeqn s\in S\}\setminus \{0\}$ is a finite subset of $M_{p_mq_m}(B).$
By passing to a subsequence
we may assume (replacing $\gamma_m$ by ${\gamma}_{m,n}$
as mentioned in \eqref{Dapfq-4}) that
\beq
\|\gamma_m(g)\|>(1-1/2m)\|g\|\not=0
\rforal g\in {\cal F}'.
\eneq
It follows that, for any $f\in {\cal F}_m,$
\beq\label{LCa-n1}
\|\gamma_m(f(s))\|  \ge (1-1/2m)\|f\| \not=0\,\,\,{\rm{for\,\,some}} \,\, s\in S\subset [0,1].
\eneq
Define $\psi_m': M_{p_m}(B)\otimes M_{q_m}\to M_{d_mp_m}(B)\otimes M_{5q_m}$ by
$
\psi_m'=
\dt_m\otimes s,$ where
\beq\label{LCa-6}
\dt_{m}(a)=\begin{pmatrix} a & 0\\
                                              0 & \gamma_m(a)\end{pmatrix}\rforal a\in M_{p_m}(B),\andeqn
        s(c)=c\otimes 1_5\rforal c\in M_{q_m}.
        \eneq
Define $\psi_m: E_{p_m, q_m}\to E_{d_{m}
p_m, 5q_m}$ by
\beq\label{LCa-7}
\psi_m(f)(t)=\psi_m'(f(t))\rforal f\in E_{p_m,q_m}\andeqn t\in [0,1].
\eneq
Let $f\in E_{p_m, q_m}.$   Then $f(0)=b\otimes 1_{q_m},$ where $b\in M_{p_m}(B).$
Thus,
\beq\label{LCa-8}
\psi_m(f)(0)=\psi_m'(f(0))=\dt_m
(b)\otimes (1_{q_m}\otimes 1_5)\in M_{d_mp_m}(B)\otimes 1_{5q_m}.
\eneq
On the other hand, $f(1)=1_{p_m}\otimes c,$ where $c\in M_{q_m}.$ Thus
\beq\label{LCa-9}
\psi_m(f)(1)=\psi_m'(f(1))=1_{d_mp_m}\otimes (c\otimes 1_5)\in 1_{d_mp_m}\otimes M_{5q_m}.
\eneq
So, indeed, $\psi_m$ maps $E_{p_m, q_m}$ into $E_{d_mp_m, 5q_m}.$

Note, for $t\in [0,1],$ we have, for all $f\in E_{p_m,q_m},$
\beq\label{LCa-10}
\psi_m(f)(t)=\psi_m'(f(t))=\begin{pmatrix}   f(t) & 0\\
                                                         0 & \gamma_m(f(t))\end{pmatrix} \otimes 1_5.
\eneq

Stage 2:
We will use a modified construction of Jiang-Su and define
$\phi_m$ on $[0,3/4].$

Choose (the first) two different prime numbers $k_0$ and $k_1$ such that
\beq\label{f-6.32}
k_0> 15 q_m\andeqn   k_1>15 k_0d_mp_m.
\eneq
In particular, $k_0, k_1\not=3, 5.$ 

Recall $(3,q_m)=1,$  $(5,p_m)=1$ and $d_m=3^{l_m}$ for some $l_m\ge 1.$
%
Therefore $(k_0d_mp_m, k_15q_m)=1.$
Let $p_{m+1}=k_0d_mp_m,$ $q_{m+1}=k_15q_m,$ and $k=k_0k_1.$
Write
\beq\label{LCa-15--1}
k=r_0+m(0)q_{m+1}\andeqn k=r_1+m(1)p_{m+1},
\eneq
where $m(0), r_0, m(1), r_1\ge 1$ are integers and
\beq
0<r_0<q_{m+1}, r_0\equiv k({\rm mod}\,
q_{m+1}),
\andeqn\\\label{r-1}
0<r_1< p_{m+1}, r_1\equiv k({\rm mod}\ 
p_{m+1}).
\eneq
Moreover, by \eqref{f-6.32},
\beq\nonumber
k-r_1-r_0&>&k-q_{m+1}-p_{m+1}=k-k_15q_m-k_0d_mp_m\\\nonumber
&=& k_1(k_0-5q_m)-k_0d_mp_m
\\\label{LCa-15}
&>&k_1(10q_m)-k_0d_mp_m> 0.
\eneq
We will construct paths $\xi_i.$
At $t=0,$ define
\beq\label{LCa-15+0}
\xi_i(0)=\begin{cases} 0 & {\text{if}}\,\,1\le i\le r_0,\\
                                    1/2 &   \text{if}\,\, r_0<i\le k\end{cases}.
\eneq
Note that, since
\beq\label{LCa-15+}
r_05q_m\equiv k5q_m\equiv k_0k_15q_m\equiv 0 ({\rm mod}\ 
 q_{m+1}),
\eneq
$r_05q_m=t_0q_{m+1}$ for some integer $t_0\ge 1.$
Note also, if $f\in E_{d_mp_m, 5q_m},$ then $f(0)=b\otimes 1_{5q_m}$ for some $b\in M_{d_mp_m}(B).$
Hence
$f(0)\otimes 1_{r_0}=b\otimes 1_{r_05q_m}=(b\otimes 1_{t_0})\otimes 1_{q_{m+1}}$
for any $f\in E_{d_mp_m, 5q_m}.$
On the other hand, for any $f\in  E_{d_mp_m,5q_m}
,$
\beq\label{LCa-18}
\diag(f(\xi_{r_0+1}(0)), \cdots, f(\xi_k(0)))=^s(f(1/2))\otimes 1_{m(0)q_{m+1}}.
\eneq
Therefore  there exists a unitary $v_0\in U(M_{p_{m+1}q_{m+1}})$ such that, for all $f\in E_{d_mp_m,5q_m},$
\beq\label{LCa-19+1}
\rho_0(f):=v_0^*\begin{pmatrix} f(\xi_1(0)) & 0  &\cdots  & 0\\
                                   0 & f(\xi_2(0))  &\cdots  &0\\
                                  \vdots  &  \vdots & & \vdots\\
                                   0 &  0 & \cdots    & f( 
                                   \xi_{k}(0))\end{pmatrix}v_0
                                     \eneq
                                  is in  $M_{p_{m+1}}(B)\otimes 1_{q_{m+1}}.$
So $\rho_0$ defines a \hm\, from $E_{d_mp_m, 5q_m}$ into $M_{p_{m+1}}(B)\otimes 1_{q_{m+1}}.$

 At $t=3/4,$
 define
 \beq\label{f-6.43}
 \xi_i(3/4)=\begin{cases} 1/2 & \text{if}\,\, 1\le i\le k-r_1\\
                                      1 & \text{if}\,k-r_1<i\le k \end{cases}.
 \eneq
As in  the case at $0,$ by \eqref{r-1},
$$
r_1d_mp_m\equiv kd_mp_m\equiv k_0k_1d_mp_m\equiv 0 ({\rm mod}\ 
p_{m+1}).
$$
So one may write $r_1d_mp_m=t_1p_{m+1}$ for some integer $t_1\ge 1.$
Let $f\in E_{d_mp_m, 5q_m}.$ Then $f(1)=1_{d_mp_m}\otimes c$ for some $c\in M_{5q_m}.$
It follows that $1_{r_1}\otimes f(1)=1_{p_{m+1}}\otimes(1_{t_1}\otimes c).$
Also,
$$
\diag(f(\xi_1(3/4)), \cdots, f(\xi_{k-r_1}(3/4)))=^s1_{m(1)p_{m+1}}\otimes f(1/2).
$$
Thus there is
a unitary $v_{3/4}\in U(M_{p_{m+1}q_{m+1}})$ such that (for $f\in E_{d_mp_m, 5q_m}$)
\beq\label{LCa-20}
\rho_{3/4}(f) :=v_{3/4}^*\begin{pmatrix} f(\xi_1(3/4)) & 0  &\cdots  & 0\\
                                   0 & f(\xi_2(3/4))  &\cdots  &0\\
                                  \vdots  &  \vdots & & \vdots\\
                                   0 &  0 & \cdots    & f(\xi_{k}(3/4))\end{pmatrix}v_{3/4}
                                     \eneq
defines a \hm\, from $E_{d_mp_m, 5q_m}$ to $1_{p_{m+1}}\otimes M_{q_{m+1}}(B).$

To connect 
$\xi_i(0)$ and $\xi_i(3/4)$ continuously, on $[0,3/4],$  let us define  (see \eqref{LCa-15})
\beq\label{LCa-20+}
\xi_i(t)=\begin{cases} 2t/3 & \text{if}\,\, 1\le i\le r_0,\\
                                     1/2 & \text{if}\,\, r_0<i\le k-r_1,\andeqn\\
                                      1/2+2t/3 &\text{if}\,\,\, k-r_1<i\le k\,.\end{cases}
\eneq
Let $v\in C([0, 3/4], M_{p_{m+1}q_{m+1}})$ be  a unitary such that
$v(0)=v_0$ and $v(3/4)=v_{3/4}.$
Now, on $[0,3/4],$ define, for all $f\in E_{p_m, q_m},$
\beq\label{LCa-21}
\phi_m(f)(t)=v(t)^*\begin{pmatrix} \psi_m(f)\circ\xi_1(t) & 0  &\cdots  & 0\\
                                   0 & \psi_m(f)\circ \xi_2(t)  &\cdots  &0\\
                                  \vdots  &  \vdots & & \vdots\\
                                   0 &  0 & \cdots    & \psi_m(f)\circ \xi_{k}(t)\end{pmatrix}v(t).
                                     \eneq
   Stage 3.      Connecting $3/4$ to $1$ (recall $\pi_1(E_{p_{m+1}, q_{m+1}})=1_{p_{m+1}}\otimes M_{q_{m+1}}$).

  We first extend $\xi_i$ by defining
  \beq\label{LCa-40}
  \xi_i(t)=\xi_i(3/4) \rforal t\in (3/4,1],\,\, i=1,2,...,k.
  \eneq
 Recall \eqref{LCa-10}, at
$t=3/4,$  for each $i$ and for $f\in E_{p_m,q_m},$
\beq\label{LCa-22}
 \psi_m(f)(\xi_i(3/4))=\begin{pmatrix} f(\xi_i(3/4))  & 0\\
                                                                         0 & \gamma_m(f(\xi_i(3/4)))\end{pmatrix} \otimes 1_5,
\eneq

\noindent
For $k-r_1< i\le k,$
define, for $t\in (
3/4, 1]$ and for $f\in E_{p_m,q_m},$
\beq\label{LCa-22+1}
\hspace{-0.2in}
{\tilde \psi}_{m,i}(f)(t)
&=&
\begin{pmatrix} f(\xi_i(3/4)) & 0\\
0& \gamma_m(f(\xi_i(3/4)))
\end{pmatrix}\otimes 1_5
=
\begin{pmatrix} f(1) & 0\\
0& \gamma_m(f(1))\end{pmatrix}\otimes 1_5
\\
&=^s& 1_{d_m}\otimes f(1)\otimes 1_5.
\eneq
Note $f(1)$ has the form $1_{p_m}\otimes c$  for some $c\in M_{q_m}.$
So, one may write
\beq
\diag({\tilde \psi}_{m,k-r_1+1}(f)(t), \cdots {\tilde \psi}_{m,k}(f)(t))=^s
1_{r_1d_mp_m}
\otimes
c
\otimes 1_5
=
1_{t_1p_{m+1}}
\otimes
c
\otimes 1_5.
\eneq

Now recall \eqref{LCa-3-0} for the definition of $\theta_t.$
For $1\le i\le k-r_1,$ define, for $t\in (3/4, 1],$
\beq\label{LCa-23}
{\tilde \psi}_{m,i}(f)(t)=\begin{pmatrix}\theta_{4(t-3/4)}(f(\xi_i(t)) & 0\\
                    0& \gamma_m(f(\xi_i(t)))\end{pmatrix}\otimes 1_5\rforal f\in E_{p_m, q_m}.
                    \eneq
Note that $\theta_1(f(\xi_i(3/4))=\theta_1(f(1/2))\in M_{p_mq_m}$ (see 
\eqref{f-6.43} and  \eqref{LCa-3+}).
Recall that $\phi_m(f)(3/4)\in 1_{p_{m+1}}\otimes M_{q_{m+1}}(B).$
Moreover (see also \eqref{LCa-15--1}),
\beq
\hspace{-0.2in}\diag({\tilde \psi}_{m,1}
(f)(1), \cdots, {\tilde \psi}_{m, k-r_1}
(f)(1))=^s1_{m(1)p_{m+1}}\otimes \begin{pmatrix} \theta_1(f(1/2)) & 0\\
                       0& \gamma_m(f(1/2))\end{pmatrix}\otimes 1_5.
\eneq
Thus, for $t=1,$  there is a unitary $v_1\in 1_{p_{m+1}}\otimes M_{q_{m+1}}$
such that
\beq\label{LCa-24}
\rho_1(f)=v_1^*\begin{pmatrix} {\tilde \psi}_{m,1}(f)(1) & 0  &\cdots  & 0\\
                                   0 & {\tilde \psi}_{m,2}(f)(1)  &\cdots  &0\\
                                  \vdots  &  \vdots & & \vdots\\
                                   0 &  0 & \cdots    & {\tilde \psi}_{m,k}(f)(1)\end{pmatrix}v_1
                                     \eneq
                                     defines a \hm\, from $E_{p_m, q_m}$ to $1_{p_{m+1}}\otimes M_{q_{m+1}}.$
There is a continuous path of unitaries $\{v(t): t\in [3/4, 1]\}\subset M_{p_{m+1}q_{m+1}}$
such that $v(3/4)$ is as defined and $v(1)=v_1$ (so now $v\in C([0,1], M_{p_{m+1}q_{m+1}})$
with $v(0)=v_0$ and $v(1)=v_1$ as well as $v(3/4)$ is consistent with previous definition).
Now define, for $t\in (3/4, 1],$
\beq\label{LCa-25}
\phi_m(f)(t)=v(t)^*\begin{pmatrix} {\tilde \psi}_{m,1}(f)(t) & 0  &\cdots  & 0\\
                                   0 & {\tilde \psi}_{m,2}(f)(t)  &\cdots  &0\\
                                  \vdots  &  \vdots & & \vdots\\
                                   0 &  0 & \cdots    & {\tilde \psi}_{m,k}(f)(t)\end{pmatrix}v(t)
                                     \eneq
for all $f\in E_{p_m, q_m}.$
Note, by \eqref{LCa-22}, \eqref{LCa-23}, \eqref{LCa-24} and \eqref{LCa-20},
\beq
\phi_m(f)(1)=\rho_1(f)\rforal f\in E_{p_m,q_m}.
\eneq
Hence  $\phi_m$ is a unital injective \hm\,
from $E_{p_m, q_m}$ to $E_{p_{m+1}, q_{m+1}}.$ (Note that injectivity
follows from the fact that $\cup_{i=1}^k \xi_i([0,1])=[0,1],$ as $r_0\ge 1$ and $k-r_1>0$.)

For the convenience of notation and for the later use,
let us define
$\tilde \psi_{m,i}: E_{p_m, q_m} \to C([0,1], M_{d_mp_m5q_m}(B))$ by
\beq\label{LCa-tildepsi}
\tilde \psi_{m,i}(f)(t)=\begin{cases} \psi_m(f\circ \xi_i(t)) &\text{if}\,\, t\in [0,3/4];\\
                                                        \tilde \psi_{m,i}(f)(t) &\text{if}\,\, t\in (3/4,1]\end{cases}\rforal f\in E_{p_m,q_m}.
\eneq
Then, we may write, for all $t\in [0,1],$ for all $f\in E_{p_m, q_m},$
\beq\label{LCa-phim}
\phi_m(f)(t)=v(t)^*\begin{pmatrix} {\tilde \psi}_{m,1}(f)(t) & 0  &\cdots  & 0\\
                                   0 & {\tilde \psi}_{m,2}(f)(t)  &\cdots  &0\\
                                  \vdots  &  \vdots & & \vdots\\
                                   0 &  0 & \cdots    & {\tilde \psi}_{m,k}(f)(t)\end{pmatrix}v(t).
                                     \eneq

Define  $\theta^{(i)}: E_{p_m, q_m}\to C([0,1], M_{p_mq_m}(B))$ by, for each $f\in E_{p_m, q_m},$
\beq\label{Dtheta-1}
\theta^{(i)}(f)(t)=\begin{cases} f(\xi_i(t)) & \text{if}\,\,t\in [0,3/4],\\
                                                \theta_{4(t-3/4)}(f(\xi_i(t))) 
                                                & \text{if}\,\,t\in (3/4,1],\end{cases}
\eneq
if $\xi_i(3/4)=1/2;$ 
and, if $\xi_i(3/4)=1,$
define
\beq
\theta^{(i)}(f)(t)=f(\xi_i(t))\rforal t\in [0,1].
\eneq

Define  $\Theta_m: E_{p_m, q_m}\to C([0,1], M_{kp_{m}q_{m}}(B))$ by, for each $f\in E_{p_m, q_m},$
\beq
\Theta_m(f)=\diag(\theta^{(1)}(f) ,\cdots, \theta^{(k)}(f))\rforal t\in [0,1].
\eneq
By \eqref{LCa-25}, \eqref{LCa-10}, \eqref{LCa-23} as well as the definition of $\Theta_m,$  and
by conjugating another unitary
in $C([0,1], M_{p_{m+1}q_{m+1}}),$ we may write
\beq
\phi_{m}(f)=u^*\begin{pmatrix} \Theta_{m}(f)  & 0  &\cdots  & 0\\
                                   0 & \gamma_{m}(f\circ \xi_1) \otimes 1_5 &\cdots  &0\\
                                  \vdots  &  \vdots & & \vdots\\
                                   0 &  0 & \cdots    & \gamma_m(f\circ \xi_k)\otimes 1_5
                                   \end{pmatrix} u
                                    \, \tforal f\in A_m.
                                     \eneq
So $\phi_m$ does have the form in \eqref{LCa-0}.    Condition (1) follows from the definition of $\xi_i,$  \eqref{LCa-20+}, \eqref{LCa-40}, and \eqref{LCa-15}.   Condition (2) follows from \eqref{LCa-3}  and \eqref{LCa-10}. Condition (3) follows from \eqref{LCa-n1}.
Finally, condition (4) follows from the definition of $\Theta_m.$
The construction is  then completed by induction.

\end{proof}

\begin{rem}\label{RTheta}
It should be noted that, if $f(0), f(1)\in M_{p_mq_m},$ then $\Theta(f(0))$ and $\Theta(f(1))$
are also scalar matrices.
\end{rem}

\section{Conclusion of the construction}

\begin{df}\label{Dcoll}
Let $\{\xi_i: 1\le i\le m\}$ be a collection of  maps described in (1) of Lemma \ref{LCa}.
The collection is said to be full, if each  of the three types occurs at least once.
Let $C_2:=\{\xi_{j(1)}\circ \xi_{j(2)}: 1\le j(1)\le m_1, 1\le j(2)\le m_2\}$ be a collection
of compositions of two maps in (1) of Lemma \ref{LCa}. The collection is called full,
if $\{\xi_{j(2)}: 1\le j(2)\le m_2\}$ is a full collection, and  for  each fixed $j(2),$
$\{\xi_{j(1)}: \xi_{j(1)}\circ \xi_{j(2)}\in C_2\}$ is also a full collection.
Inductively,
a collection of $n$
composition  of maps in (1) of Lemma \ref{LCa}
$$
C_n=\{\xi_{j(1)}\circ \xi_{j(2)}\circ \cdots \circ \xi_{j(n)}: 1\le j(i)\le m_i: i=1,2,..,n\}
$$
is called full, if $\{\xi_{j(n)}: 1\le j\le m_n\}$ is a full collection,
and, for each fixed $\xi_{j(n)},$
the collection
$$
\{\xi_{j(1)}\circ \xi_{j(2)}\circ\cdots \circ \xi_{j(n-1)}: \xi_{j(1)}\circ \xi_{j(2)}\circ\cdots \circ \xi_{j(n-1)}\circ \xi_{j(n)}\in C_n\}
$$
is full.

\end{df}

\begin{lem}\label{Lslopedense}
Let $\xi_j: [0,1]\to [0,1]$ be one of the three types of continuous maps given by (1) of Lemma \ref{LCa}.
Let $\Xi_j=\xi_{j(1)}\circ \xi_{j(2)}\circ \cdots \circ \xi_{j(n)}: [0,1]\to [0,1]$ be a composition of $n$
maps of some $\xi_j$'s. 
Then, for any $x, y\in [0,1],$
\beq
|\Xi_j(x)-\Xi_j(y)| \le  (2/3)^n.
\eneq
Moreover, if $\{\Xi_j: 1\le j\le l\}$ is a full  collection of compositions  of $n$ maps as above,
then
\beq\label{Lslope-1}
\cup^l_{j=1} \Xi_j([0,1])=[0,1],
\eneq
and, for each $t\in [0,1],$
$
\{\Xi_j(t): 1\le j\le l\}
$
is $(2/3)^n$-dense in $[0,1].$
\end{lem}

\begin{proof}
Note that, for each $i,$ and for any $x,y\in [0,1],$
$|\xi_i(x)-\xi_i(y)|\leq (2/3)|x-y|.$ Then, by induction,
for all $x, y\in [0,1],$
\beq
|\Xi_j(x)-\Xi_j(y)|
&=&
|\xi_{j(1)}\circ \xi_{j(2)}\circ \cdots \circ \xi_{j(n)}(x)-
\xi_{j(1)}\circ \xi_{j(2)}\circ \cdots \circ \xi_{j(n)}(y)|
\\
&\le&
(2/3)|\xi_{j(2)}\circ \cdots \circ \xi_{j(n)}(x)-
\xi_{j(2)}\circ \cdots \circ \xi_{j(n)}(y)|
\\
&\le&...
\le (2/3)^n|x-y|\le(2/3)^n.
\eneq
One already observes that $\cup_{j\in S}\xi_j([0,1])=[0,1],$ if $\{\xi_j: j\in S\}$
is a full collection. An induction shows that, if $\{\Xi_j: 1\le j\le l\}$ is a full collection,
then
$$
\cup_{j=1}^l \Xi_j([0,1])=[0,1].
$$
To show the last statement, fix $t\in [0,1].$ Let $x\in [0,1].$
Then, for some $y\in [0,1]$ and $j\in \{1,2,...,l\},$
$$
\Xi_j(y)=x.
$$
Now, by the first part of the statement, for any $t\in [0,1],$
$$
|\Xi_j(t)-\Xi_j(y)| \le 
(2/3)^n.
$$
It follows that
$$
|\Xi_j(t)-x|=|\Xi_j(t)-\Xi_j(y)| \le 
(2/3)^n.
$$

\end{proof}

\begin{thm}\label{TCCS}
The inductive limit $A$  in Lemma \ref{LCa} can be made into a unital simple \CA\, $A_z^C$ such
that
\beq
(K_0(A_z), K_0(A_z)_+, [1_{A_z}], K_1(A_z))=(\Z, \Z_+, 1, \{0\}).
\eneq
If $C$ is not exact, then $A_z^C$ is not exact.
\end{thm}

\begin{proof}
For convenience, one
makes an additional  requirement in the construction.
Let ${\cal F}_{m,1}\subset {\cal F}_{m,2},....,{\cal F}_{m,n},...$ be an increasing sequence of finite subsets
of $A_m$ such that $\cup_n {\cal F}_{m,n}$ is dense in $A_m.$

One requires   $\phi_m({\cal F}_{m,m+1}
)\subset {\cal F}_{m+1,1}$
and $\phi_m({\cal F}_{m,
,m+n})\subset {\cal F}_{m+1,n},$  $m, n=1,2,....$

This is done inductively as follows:
Choose any  increasing sequence of finite subsets ${\cal F}_{1,1}\subset {\cal F}_{1,2},..., \subset A_1$
such that $\cup_n {\cal F}_{1,n}$ is dense in $A_1.$
Specify ${\cal F}_1={\cal F}_{1,1}\setminus \{0\}.$
Choose $A_2$ and define $\phi_1: A_1\to A_2$ as in the construction of
Lemma \ref{LCa}.

Choose an increasing sequence of finite subsets ${\cal F}_{2,1}, {\cal F}_{2,2},...$ of $A_2$
such that $\phi_1({\cal F}_{1,n})\subset {\cal F}_{2,n},$ $n=1,2,....$
Specify ${\cal F}_2={\cal F}_{2,1}\setminus \{0\}.$

Once ${\cal F}_{m,1},{\cal F}_{m,2},...$ are determined.  Specified ${\cal F}_m={\cal F}_{m,1}\setminus
\{0\}.$
Then construct $A_{m+1}$ and $\phi_m: A_m\to A_{m+1}$ as in
Lemma \ref{LCa}.
Choose ${\cal F}_{m+1,1}, {\cal F}_{m+1,2},....$ so that
$\phi_m({\cal F}_{m,m+1})\subset {\cal F}_{m+1,1}$ and $\phi_m({\cal F}_{m,m+n})\subset {\cal F}_{m+1,n}$
as well as ${\cal F}_{m+1,n}\subset {\cal F}_{m+1,n+1}.$
Moreover $\cup_n {\cal F}_{m+1,n}$ is dense in $A_{m+1}.$  Choose ${\cal F}_{m+1}={\cal F}_{m+1,1}\setminus \{0\}.$
Thus the requirement  can be made.

Let us now prove that $A$ is simple.
For this, we will prove the following claim.

Claim:   for any fixed $i,$ and $g\in A_i\backslash \{0\},$ there exists $n>i$ such that $\phi_{i,n}(g)$
is full in $A_n.$ \Wlog, we may assume that $\|g\|=1.$
There is $j$ and $f\in {\cal F}_{i,j}$ such that $\|f-g\|<1/64.$


To simplify notation, \wilog, we may write $i=1.$
Set  $\phi_{j,j'}=\phi_{j'-1}\circ \cdots \phi_j$ for $j'>j.$
Then $\phi_{1, j'}(f)\, \in\,  
\phi_{j'}({\cal F}_{j'-1, j'})\subset {\cal F}_{j', 1}.$
Recall also that each $\phi_j$ is unital and injective.
To further simplify the notation, \wilog, we may write ${\cal F}_{i,j}\setminus \{0\}={\cal F}_m={\cal F}_{m,1}\setminus \{0\}.$
We assume that $m>128.$
By the construction, for some $t\in (0,1),$
\beq
\|\gamma_m(f(t))\|> (1-1/2m)\|f\|\not=0.
\eneq
By the continuity, there is $n(m)\ge 1$ such that, for any $(2/3)^{n(m)-1}$-dense set $S$ of $[0,1],$
\beq\label{LCCS-10}
\|\gamma_m(f(s))\| \ge (1-1/2m)\|f\|\not=0\,\,\,\rm{for\,\, some}\,\,s \in  S.
\eneq
For any $f\in C([0,1], M_{p_mq_m}(B))$ and $i,$  denote $h(t)=\gamma_m(f\circ \xi_i(t)){\otimes 1_5}$
(for $t\in [0,1]$). Then,  for any $k>m$
and $j$ ($t\in [0,1]$),
\beq\label{TCCS-11}
\gamma_k(h\circ \xi_j(t))=\gamma_k(\gamma_m(f\circ \xi_i\circ \xi_j(t))\otimes 1_5)=\gamma_m(f\circ\xi_i\circ \xi_j(t))\otimes 1_{5R(k)},
\eneq
where $R(k)$ is the rank of $\gamma_k$ and $\xi_i$ and $\xi_j$ are as defined in (1) of  Lemma \ref{LCa}.
Denote
\beq\label{TCCS-12}
{\bar \gamma}_{k+1}(f)(t)=\gamma_{k+1}(f(t))\otimes 1_5\rforal f\in C([0,1], M_{p_kq_k}(B))\,\, {\rm{(and}}\,\,t\in [0,1]{\rm )}.
\eneq
Therefore, from Lemma \ref{LCa} and \eqref{LCa-0} (also keep Remark \ref{RTheta} in mind), we may write,
for each $f\in A_m=E_{p_m,q_m},$
\beq
\label{TCCS-13}
&&\hspace{-0.7in}\phi_{m, m+2}(f)\\
&&\hspace{-0.7in}=w_1^*\begin{pmatrix}
H_0(f)
&&& 0\\
                               &\bar\gamma_m(f\circ \xi^{(2)}_1)\otimes 1_{R(m+1)}&&\\
                                &&\ddots &\\
                                           0&&&    \bar\gamma_m(f\circ \xi^{(2)}_{l(m+1)})\otimes 1_{R(m+1)}\end{pmatrix} w_1,
\eneq
where 
$H_0: A_m\to C([0,1], M_{L_0p_mq_m})$ is a homomorphism 
(for some integer $L_0\geq 1$),
$w_1\in C([0,1], M_{p_{m+2}q_{m+2}})$ is a unitary, $R(m+1)$ is the rank
of $\bar\gamma_{m+1},$
and $\{\xi^{(2)}_j: 1\le j\le l(m+1)\}$ is a full collection
of compositions of two $\xi_i$'s (maps in (1) of Lemma \ref{LCa}).

Therefore, by the induction,
for any $n>n(m)+m,$ one may write,  from the construction of Lemma \ref{LCa} (see \eqref{LCa-0}), for all $f\in A_m=E_{p_m,q_m},$
\beq\label{TCCS-n1}
\phi_{m,n}(f)= w^*\begin{pmatrix} H(f) &&&0 \\
                                          & \bar\gamma_m(f\circ \Xi_1)\otimes 1_{R(n,1)} &&\\
                                           &&\ddots &\\
                                           0&&&    \bar\gamma_m(f\circ \Xi_l)\otimes 1_{R(n,l)}\end{pmatrix} w,
\eneq
where $H: A_m\to C([0,1], M_{Lp_mq_m}(B))$ is a \hm\, (for some integer $L\ge 1$), $\Xi_j$ is a composition of
$n-m$ maps
in (1) of  Lemma \ref{LCa} such that the collection $\{\Xi_j: 1\le j\le l\}$ is full,
and $R(n,j)\ge 1$ is an integer, $j=1,2,...,l,$ and $w\in C([0,1], M_{p_nq_n})$
is a unitary.

 It follows from Lemma \ref{Lslopedense} that
\beq
|\Xi_i(x)-\Xi_i(y)|<(2/3)^{m-n}\rforal x, y\in [0,1],\,\, 1\le i\le l,\andeqn \cup_i^l \Xi_i([0,1])=[0,1].
\eneq
Fix any $t\in [0,1]$ and $x\in [0,1],$  by Lemma \ref{Lslopedense}, there is $y\in [0,1]$  and $j\in \{1,2,...,l\}$ such that $\Xi_j(y)=x.$
Then,
\beq
|\Xi_j(t)-x|=|\Xi_j(t)-\Xi_j(y)|<(2/3)^{n-m}<(2/3)^{n(m)}.
\eneq
It follows from the choice of $n(m)$ and \eqref{LCCS-10} that, for $f\in {\cal F}_m,$
\beq
\|\gamma_m(f\circ \Xi_j(t))\|\ge (1-1/m)\|f\|\ge ({63\over{64}})^2 \rforal t\in [0,1].
\eneq
Since $\|f(\Xi_j(t))-g(\Xi_j(t))\|<1/64,$
this implies that
\beq
\|\gamma_m(g(\Xi_j(t)))\|\ge {63^2-64\over{64^2}}\rforal t\in [0,1].
\eneq
%
Since, for each $t\in [0,1],$
$\gamma_m(g\circ \Xi_i(t))\in M_{p_mq_m},$   $i=1,2,...,l,$
$
\phi_{m,n}(g)(t)
$
is not in any closed ideal of $M_{p_nq_n}(B)$ for each $t\in [0,1].$  Therefore $\phi_{m,n}(g)$ is full in $E_{p_n,q_n}=A_n.$
This proves the claim.

It  follows from the claim that $A_z^C$ is simple.
To see this, let $I\subset A_z^C$ be an ideal  such that $I\not=A_z^C$ and put $C_n=\phi_{n, \infty}(A_n).$
Then $C_n\subset C_{n+1}$ for all $n.$
Let $a\in C_m\setminus \{0\}.$ By the claim,   there is $n'>m$ such that $a$ is full in $C_{n'}$ and therefore
$a$ is full in every $C_n$ for $n\ge n'.$
In other words,
$a\not\in C_n\cap I$  for all $n.$
It follows that $C_m\cap I=\{0\}$  as $C_m\subset C_n$ for all $n\ge m.$ It is then standard to show that $I=\{0\}.$
Thus $A_z^C$ is simple.

Since, for each $m,$  by Lemma \ref{LKofApnn},
$$
(K_0(A_m), K_0(A_m)_+, [1_{A_m}], K_1(A_m))=(\Z, \Z_+, 1, \{0\}),
$$
one concludes (as each $\phi_n$ is unital)  that
\beq
(K_0(A_z^C), K_0(A_z^C)_+, [1_{A_z^C}], K_1(A_z^C))=(\Z, \Z_+, 1, \{0\}).
\eneq
Finally if  $C$ is not exact, then $B$ is not exact since $B$ has quotients with the form
$\C\oplus C$ which is not exact.

Define $\Phi: B\to C([0,1], M_{15}(B))$ by
\beq
\Phi(f)(t)=\theta_t(f)\otimes 1_{15} 
\rforal f\in B \andeqn t\in [0,1],
\eneq
where $\theta_t: B\to B$ is defined in \eqref{LCa-3-0}.
Note, for $f\in B,$
\beq\label{Phi-0}
\Phi(f)(0)=\theta_0(f)\otimes 1_{15}=f\otimes 1_{15}\in M_3(B)\otimes 1_5\andeqn \Phi(f)(1)=f(0)\otimes 1_{15}\in \C\cdot 1_{15}.
\eneq
One then obtains a unitary $u\in C([0,1], M_{15})$ such that
\beq
u^*\Phi(f)u\in E_{3,5}.
\eneq
Define $\Psi(f)=u^*\Phi(f)u$ for all $f\in B.$ Then $\Psi$ is a unital  injective \hm.
In other words,  $B$ is embedded unitally into $A_1=E_{3,5}.$ Since each $\phi_m: A_m\to A_{m+1}$
is unital and injective, $B$ is embedded into $A_z^C.$ Since $B$ is not exact, neither is $A_z^C$ 
(see, for example, Proposition 2.6 of \cite{W}).
\end{proof}

\begin{prop}\label{Pnonnuclear}
If $C$ is exact but not nuclear, then $A_z^C$ is exact and not nuclear.
\end{prop}

\begin{proof}
Note that, since $C$ is non-nuclear and exact, so is $B.$
Note also that $A_n=E_{p_n,q_n}$ is a \SCA\,of exact \CA\, $C([0,1], M_{p_nq_n}(B)).$
So each $A_n$ is exact. By  \cite[2.5.5]{W},
 $A_z^C$ is exact.

Let $\Phi: B\to A_1=E_{3,5}$ be as in the end of the proof of  Theorem \ref{TCCS}.
Let $\pi_0^{(1)}: A_1\to M_3(B)\otimes M_5$ be the evaluation at $0,$
namely $\pi_0^{(1)}(f)=f(0)$ for all $f\in A_1.$
Let $\eta_1: M_3(B)\otimes {\rm 1}_5\to B$ by defining $\eta_1((b_{i,j})_{
3\times 3}\otimes 1_5
)=b_{1,1},$
where $b_{i,j}\in B,$ $1\le i,j\le 3.$
Then $\eta_1$ is a norm one c.p.c~map.
Define $\pi_0^{(1,1)}: A_1\to B$ by $\pi_0^{(1,1)}(f)=\eta_1\circ \pi_0^{(1)}.$
Note $\pi_0^{(1,1)}\circ \Phi$ is an isomorphism.
In fact, $\pi_0^{(1,1)}\circ \Phi(b)=\theta_0(b)=b$ (see \eqref{Phi-0}) for all $b\in B.$

The above is illustrated  in the diagram:
{\small{\beq
\label{Nulcear-p1}
\xymatrix{
B
 \ar[rrrr]^-{\Phi}
  \ar[ddrrrr]_-{{\rm id}_B}
 & &  &&
A_1  
\ar[d]^{\pi_0^{(1)}}   \\
&&&&M_3(B)\otimes 1_5
\ar[d]^{\eta_1}  
\\
&&&&B
}
\eneq}}
We will use the same diagram in the n-stage.

In (4) of Lemma \ref{LCa}, let us choose $\xi_1$ so that $\xi_1(t)=2t/3$ for $t\in [0,3/4]$ and
so $\theta^{(1)}(f(0))=f(0)$ for all $f\in E_{3,5}.$
So,  by \eqref{LCa-0} and \eqref{LCa-theta}, we may write
\beq
\phi_1(f)
(=\phi_{1,2}(f))
=
\diag(\theta^{(1,2)}(f)), H_1'(f))\rforal f\in A_1,
\eneq
where $\theta^{(1,2)}:=\theta^{(1)}$ and $\theta^{(1,2)}(f)(0)=f(0)$ for $f\in A_1$ and
$H_1': A_1=E_{3,5}\to C([0,1], M_{p_1q_1}(B))$ is a \hm\,
(note that the image of $H_1'$ is in a corner of $C([0,1], M_{p_1q_1}(B))$).
Similarly, by the formula \eqref{LCa-0} and \eqref{LCa-theta} again, we may also write
\beq
\phi_{1,3}
(f)=\diag(\theta^{(1,3
)}(f), H_2'(f))\rforal f\in A_1,
\eneq
where $\theta^{(1,3
)}(f)(0)=f(0)$ for $f\in A_1$ and
$H_2': A_1\to C([0,1], M_{p_2q_2}(B))$ is a \hm.  By induction, for any $n>1,$
we may write
\beq\label{Nulcear-10}
\phi_{1,n}(f)=\diag(\theta^{(1,n)}(f), H'_n(f))\rforal f\in A_1,
\eneq
where $\theta^{(1,n)}(f)(0)=f(0)$ and $H'_n: A_1\to C([0,1], M_{p_nq_n}(B))$ is a \hm.
(One should be warned that
$\diag(\theta^{(1,n)},0,...,0)$ is not in $A_n.$)

Now we prove that $A_z^C$ is not nuclear.   We follow the proof of 
Proposition 6 of \cite{D2000}.
Assume otherwise,
for any finite subset ${\cal F}\subset B$ and $\ep>0,$
if $A_z^C$ were nuclear, then $\phi_{1, \infty}\circ \Phi$ would be nuclear.
Therefore there would be  a finite dimensional \CA\, $D$ and c.p.c~maps $\af: B\to D$ and $\bt: D\to A_z^C$
such that
\beq
\|\phi_{1,\infty}\circ \Phi(b)-\bt\circ \af(b)\|<\ep/2\rforal b\in {\cal F}.
\eneq
Since $A_z^C$ is assumed to be nuclear,  by the Effros-Choi lifting theorem 
\cite{ChoiE76},
there exist an integer  $n\ge 1$  and a unital c.p.c~map $\bt_n: D\to A_n$ such that
\beq
\|\bt(x)-\phi_{n, \infty}\circ \bt_n(x)\|<\ep/2\rforal x\in \af({\cal F}).
\eneq
Thus
\beq
\|\phi_{n, \infty}(\phi_{1,n}\circ \Phi(b)-\bt_n\circ \af(b))\|<\ep.
\eneq
As $\phi_{n,\infty}$ is an isometry,  this implies that
\beq\label{Nuclear-11}
\|\phi_{1,n}\circ \Phi(b)-\bt_n\circ \af(b)\|<\ep\rforal b\in B.
\eneq
Let $\pi_0^{(n)}: E_{p_n,q_n}\to M_{p_n}(B)\otimes 1_{q_n}$ be the evaluation at $0$ defined
by $\pi_0^{(n)}(a)=a(0).$
We have, by \eqref{Nulcear-10},
\beq
\pi_0^{(n)}(\phi_{1,n}\circ \Phi(b))=\diag(\theta_0(b)
\otimes 1_{15}, H_n'(\Phi(f))(0))\rforal b\in B.
\eneq
Recall that $\theta_0(b)=b.$
Now a rank 1 projection $p$ corresponding the first $(1,1)$ corner  is in $M_{p_n}(B)\otimes 1_{q_n}.$
We now use the $n$-stage diagram of \eqref{Nulcear-p1}.
Define $\eta_n: M_{p_n}(B)\otimes 1_{q_n}\to B$ defined by
$\eta_n(x)=pxp$ for all $x\in M_{p_n}(B)\otimes 1_{q_n}$ which is a unital c.p.c~map
($\eta_n((b_{i,j})_{p_n\times p_n}\otimes 1_{q_n})=b_{1,1}$).
Note that
$\eta_n\circ \pi_0^{(n)}\circ  \phi_{1,n}\circ \Phi={\rm id}_B.$
By \eqref{Nuclear-11},
\beq
\|b-\eta_n\circ \pi_0^{(n)}\circ \bt_n\circ \af(b)\|=\|\eta_n\circ \pi_0^{(n)}(\phi_{1,n}\circ \Phi(b)-\bt_n\circ \af(b))\|<\ep\rforal b\in B.
\eneq
This would imply  that $B$ is nuclear.  Therefore $A_z^C$ is not nuclear.
The above could be illustrated by the following diagram: 
%
\beq\nonumber
{{
\xymatrix{
&&&M_{p_n}(B)\otimes 1_{q_n}\ar[dlll]_{\eta_n}&&\\
B  
 && B \ar[ll]_{\vspace{0.01in}\hspace{0.2in}\id_B}\ar[rrr]^{\phi_{1,n}\circ\Phi}\ar[dll]_{\af}\ar[drrr]_{\phi_{1, \infty}\circ \Phi}
 &&&
A_n  
\ar[d]^{\phi_{n,\infty}}   
\ar[ull]_{\pi_0^{(n)}}\\
D\ar@{..>}[u]^{\eta_n\circ \pi_0^{(n)}\circ \bt_n }
\ar[rrrrr]_{\bt}
\ar@{-->}[urrrrr]^{\hspace{-0.5in}\bt_n}
&&&&&A_z
^C
}
}}
\eneq
(which is only approximately commutative below the top triangle).
\end{proof}

\begin{thm}\label{LCCTrace}
The inductive limit $A_z^C$ in Theorem \ref{TCCS} has a unique tracial state.
\end{thm}

\begin{proof}
First we note each  unital \CA\, $A_m=E_{p_m, q_m}$ has at least one tracial state, say $\tau_m.$
Note that $\phi_{m, \infty}$ is injective \hm.  So we may view $\tau_m$ as a tracial state of
$\phi_{m, \infty}(A_m).$
Extend $\tau_m$ to a state $t_m$ on $A_z^C.$
Choose a weak*-limit of $\{t_m\},$ say $t.$ Then $t$ is a state of unital \CA\, $A_z^C.$
Note that  $\phi_{m, \infty}(A_m)\subset \phi_{n, \infty}(A_n)$ if $n>m.$
Then, for each pair $x, y\in \phi_{m,\infty}(A_m),$ and for any $n>m,$
$t_n(xy)=t_n(yx).$ It follows that $t$ is a tracial state of $A_z^C.$ In other words, $A_z^C$ has at least one tracial state.


Claim: for each $k$ and each $a\in A_k$ with $\|a\|\le 1,$ any each $\ep>0,$
there exists $N>k$ such that, for all $n\ge N,$
\beq
|\tau_1(\phi_{k, n}(a))-\tau_2(\phi_{k,n}(a))|<\ep\rforal \tau_1, \tau_2\in T(A_n).
\eneq
Fix $a\in A_k.$
To simplify the notation, \wilog, we may assume that $k=1.$

Choose $m>1$ such
that
\beq\label{TCCT-3}
1/3^{m-1}<\ep/4.
\eneq
Put $g=\phi_{1,m}(a).$ 
There is $\dt>0$ such that
\beq\label{TCCT-4+0}
\|g(x)-g(y)\|<\ep/4\rforal x,y\in [0,1]\andeqn |x-y|<\dt.
\eneq


\noindent
Recall that here we view $\gamma_m$ as a map from $M_{p_mq_m}(B)$ to $M_{R(m)p_mq_m}.$
Note,  for each $f\in A_m,$ since 
$\gamma_m(f(t))$
is a scalar matrix for all $t\in [0,1],$
$\gamma_m(f(t))(x),$  as an element in $M_{R(m)p_mq_m}(B),$ is a constant matrix (for $x\in (0,1]$)  in $M_{R(m)p_mq_m}(C_0({\widetilde{(0,1], C)}}).$
Hence (see \eqref{TCCS-12}  for $\bar\gamma_m$),
for $t\in [3/4,1],$ 
(recall that, $\xi_i(t)=\xi_i(3/4)$ for all $t$ in $[3/4,1]$),
\beq\label{LCCT-2}
\theta_{4(t-3/4)}(\bar\gamma_m(f(\Xi_j\circ \xi_i))(3/4))=\bar\gamma_m(f(\Xi_j\circ \xi_i)(3/4))=\bar\gamma_m(f(\Xi_j\circ \xi_i))(t).
\eneq
(Recall the definition of $\theta_t$ in \eqref{LCa-3-0-0}).
Therefore (see the definition of $\theta^{(i)}$ in \eqref{Dtheta-1}), for any $i$ with $\xi_i(3/4)\not=1,$
\beq
\theta^{(i)}(\bar\gamma_m(f\circ\Xi_j))(t)&=&\begin{cases} \bar\gamma_m(f\circ \Xi_j\circ \xi_i)(t) &\text{if}\,\, t\in [0,3/4],\\
                                                                                  \theta_{4(t-3/4)}(\bar\gamma_m(f\circ\Xi_j\circ \xi_i)(3/4)) &\text{if}\,\,
                                                                                  t\in (3/4,1] \end{cases}\\\label{LCCT-3}
                             &=&     \bar\gamma_m(f(\Xi_j\circ \xi_i))(t).
\eneq
For those $i$ so that $\xi_i(3/4)=1,$
one also has
\beq
\theta^{(i)}(\bar\gamma_m(f\circ \Xi_j))=\bar\gamma_m(f\circ\Xi_j\circ \xi_i).
\eneq
It follows that (recall (4) of Lemma \ref{LCa} for the definition of $\Theta_{m+1}$ and also keep Remark \ref{RTheta} in mind)
\beq\label{TCCT-3}
\hspace{-0.2in}\Theta_{m+1}(\phi_m(f))=u^*\begin{pmatrix} \Theta_{m+1}'(f)  & 0  &\cdots  & 0\\
                                   0 & \bar\gamma_{m}(f\circ \xi^{(2)}_1)  &\cdots  &0\\
                                  \vdots  &  \vdots & & \vdots\\
                                   0 &  0 & \cdots    & \bar\gamma_m(f\circ \xi^{(2)}_{k'})
                                   \end{pmatrix} u
                                    \, \tforal f\in A_m,
\eneq
where $u\in C([0,1], M_{R(m+1,0)p_{m+1}q_{m+1}})$ is a unitary,
$\Theta_{m+1}': A_m\to C([0,1], M_{T(0)p_mq_m}(B))$ is a \hm\, for some integer $T(0)\ge 1,$
and $\{\xi^{(2)}_j: 1\le j\le k'\}$ is a full collection of compositions of two maps in (1) of Lemma \ref{LCa}.
Moreover, by  (2) of  Lemma \ref{LCa},
\beq\label{TCCT-3+}
T(0)/5k'R(m)<1/3^m.
\eneq
Then, combing with \eqref{TCCS-11}, we may write $\phi_{m,m+2}: A_m\to A_{m+2}$ as
\beq\label{TCCT-3+2}
\hspace{-0.2in}\phi_{m,m+2}(f)=u_1^*\begin{pmatrix} H_{m+1}(f)  & 0  &\cdots  & 0\\
                                   0 & \bar\gamma_{m}(f\circ \xi^{(2)}_1)
                                   \otimes 1_{r(1)}  &\cdots  &0\\
                                  \vdots  &  \vdots & & \vdots\\
                                   0 &  0 & \cdots    & \bar\gamma_m(f\circ \xi^{(2)}_{l(m)})\otimes 1_{r(l(m))}
                                   \end{pmatrix} u_1
\eneq
for all $f\in A_m,$
where $u_1\in C([0,1], M_{p_{m+2}q_{m+2}})$ is a unitary,
$H_{m+1}: A_m\to C([0,1], M_{T(1)p_mq_m}(B))$ is a \hm\, for some integer $T(1)\ge 1,$
$\{\xi^{(2)}_j: 1\le j\le l\}$ is a full collection of compositions of  two maps  in (1) of Lemma \ref{LCa}, and
$r(l(j))\ge 1$ is an integer, $j=1,2,...,l(m).$ Moreover,
\beq\label{TCCT-3+3}
T(1)/5R(m)(\sum_{j=1}^{l(m)} r(l(j)))<1/3^m.
\eneq
Therefore, by Lemma \ref{LCa}, noting \eqref {LCa-0}, \eqref{LCa-theta}, \eqref{LCCT-3}, and
the proof of Theorem \ref{TCCS} (see \eqref{TCCS-11}),  as well as \eqref{TCCT-3}, repeatedly,
one may write, for each $n>m,$ for all $f\in A_m,$
\beq\label{TCCT-4}
\phi_{m,n}(f)=
w^*\begin{pmatrix} H_{m,n}(f) &&&0 \\
                                          & \bar\gamma_m(f\circ \Xi_1) &&\\
                                           &&\ddots &\\
                                           0&&&
                                       \bar\gamma_m(f\circ \Xi_L)\end{pmatrix} w,
\eneq
where $w\in C([0,1], M_{p_nq_n})$ is a unitary, $H_{m,n}: A_m\to C([0,1], M_{L(0)p_mq_m}(B))$
is a \hm\, for some integer $L(0)\ge 1,$
and
where $\Xi_j:[0,1]\to [0,1]$ is a composition
of $n-m$ many $\xi_i$' s,  and $\{\Xi_j:1\le j\le L\}$ is a full collection.
Moreover,
\beq\label{TCCT-5}
L(0)/5LR(m) <1/3^m.
\eneq
We choose $N$ such that $(2/3)^{N-m}<\dt$  and
choose any  $n\ge N.$

Let $\tau_i\in T(A_n)$ ($i=1,2$). Then, for any $f\in A_n,$
\beq\label{TCCT-6}
\tau_i(f)=\int_0^1 \sigma_i(t) (f(t))d\mu_i,\,\,\, i=1,2,
\eneq
where $\sigma_i(t)$ is a tracial state of $M_{p_nq_n}(B)$ for all $t\in (0,1),$
$\sigma_i(0)$ is a tracial state of $M_{p_n}(B)\otimes 1_{q_n},$  $\sigma_i(1)$ is a tracial state
of $1_{p_n}\otimes M_{q_n},$ and
$\mu_i$ is a probability Borel measure on $[0,1],$ $i=1,2.$
For each $t\in [0,1]$ and for $f(t)\in M_{p_nq_n}\subset M_{p_nq_n}(B),$
\beq\label{TCCT-7}
\sigma_i(t)(f(t))={\rm tr}(f(t)),\,\,i=1,2,
\eneq
where ${\rm tr}$ is the normalized trace on $M_{p_nq_n}$ (see \eqref{65+}).
For, each $j\in \{1,2,...,L\},$  by Lemma \ref{Lslopedense},
\beq\label{TCCT-8}
|\Xi_j(x)-\Xi_j(y)|<(2/3)^{n-m}<\dt\rforal x, y\in [0,1].
\eneq
By the choice of $\dt,$
\beq\label{TCCT-9}
\|g\circ \Xi_j(x) -g\circ \Xi_j(y)\|<\ep/4\rforal x, y\in [0,1].
\eneq
For each $f\in A_m,$ write
\beq\label{TCCT-10}
H'(f)(t)=\begin{pmatrix} H_{m+1}(f)(t)  & 0\\
                                    0 & 0\end{pmatrix}\rforal t\in [0,1].
\eneq
Then one has, for each $f\in A_m,$
\beq\label{TCCT-11}
\tau_i(\phi_{m,n}(f))&=&\int_0^1 \sigma_i(t)(\phi_{m,n}(f)) d\mu\\\label{TCCT-11+}
&=&\int_0^1\sigma_i(t)(H'(f)(t))d\mu +\int_0^1 {\rm tr}(\bigoplus_{j=1}^L(f\circ \Xi_j(t)))d\mu.
\eneq
By \eqref{TCCT-9} (recall that $\|g\|\le 1$)
\beq\label{TCCT-12}
\int_0^1|{\rm tr}(\bigoplus_{j=1}^L(g\circ \Xi_j(1/2))-\bigoplus_{j=1}^L
(g\circ \Xi_j(t)))|d\mu
<(\ep/4)\int_0^1d\mu=\ep/4.
\eneq
By \eqref{TCCT-5},
\beq\label{TCCT-13}
\int_0^1|\sigma_i(t)(H'(g)(t))|d\mu<(1/3)^m<\ep/4.
\eneq
Recall $\phi_{1,n}(a)=\phi_{m, n}(g).$
Thus, by  \eqref{TCCT-11}, \eqref{TCCT-11+},  \eqref{TCCT-12}, 
and \eqref{TCCT-13},
\beq\label{TCCT-14}
|\tau_i(\phi_{1,n}(a))-\sum_{j=1}^L{\rm tr}(g(\Xi_j(1/2)))|<\ep/2,\,\, i=1,2.
\eneq
Therefore
\beq\label{TCCT-15}
|\tau_1(\phi_{1,n}(a))-\tau_2(\phi_{1,n}(a))|<\ep.
\eneq
This proves the claim.

To complete the proof, let $s_1, s_2\in T(A_z^C).$
Let $a\in A_z^C$ and let $\ep>0.$ Then there is $f\in A_k$
for some $k\ge 1$ such that
\beq
\|a-\phi_{k,\infty}(f)\|<\ep/3.
\eneq
Let $\tau_{i,n}=s_i\circ \phi_{n,\infty}.$
Then, by the claim,  there exists $N\ge k$ such that, for all $n>N,$
\beq
|\tau_{1,n}(\phi_{k,n}(f))-\tau_{2
,n}(\phi_{k,n}(f))|<\ep/3.
\eneq
It follows that
\beq
|s_1(\phi_{k,\infty}(f))-s_2(\phi_{k,\infty}(f))|\le \ep/3.
\eneq
Therefore
\beq\nonumber
|s_1(a)-s_2(a)|&\le& |s_1(a)-s_1(\phi_{k,\infty}(f))|\\\nonumber
\hspace{0.3in}&&+|s_1(\phi_{k,\infty}(f))-s_2(\phi_{k,\infty}(f))|+|s_2(a)-s_2(\phi_{k,\infty}(f))|
<\ep.
\eneq
It follows that $s_1(a)=s_2(a).$ Thus
$A_z^C$ has a unique tracial state.

\end{proof}

\begin{rem}\label{RDJSZ}
Recall the construction allows $B=\C$ (with $C=\{0\}$).
In that case, of course $A_z^C={\cal Z}.$
Note that,  when $B=\C,$  $\theta_t(b)=b$ for all $b\in M_{p_mq_m}.$
In other words, $\theta_t=\id_B.$

 Let
\beq
Z_{p_m,q_m}=\{f\in C([0,1], M_{p_mq_m}): f(0)\in M_{p_m}\otimes 1_{q_m}\andeqn f(1)\in 1_{p_m}\otimes M_{q_m}\}.
\eneq
In general (when $C\not=\{0\}$), one has
$Z_{p_m,q_m}\subset  E_{p_m,q_m},$
as we view $\C\subset B$ and $M_{p_mq_m}\subset M_{p_mq_m}(B).$
Let
$\phi_m^z={\phi_m}|_{Z_{p_m,q_m}}.$
Then,   since $v\in C([0,1], M_{p_{m+1}q_{m+1}})$ (see  the line above \eqref{LCa-25}),
$\phi_m^z(Z_{p_m,q_m})\subset Z_{p_{m+1}, q_{m+1}}.$
Thus, one obtains a unital \SCA\, (of $A_z^C$)
\beq\label{DJSZ-1}
B_z=\lim_{n\to\infty}(Z_{p_m, q_m}, \phi_m^z).
\eneq
Then $B_z\cong {\cal Z}$ (\cite{JS1999}).

\end{rem}

\section{Regularity of $A_z^C$}

In this section, let $A_z^C$ be the $C^*$-algebra in Theorem \ref{TCCS}.

\begin{lem}\label{LCCtrnz}
$A_z^C$ has the following property.

(1) $A_z^C$ has a unital \SCA\, $B_z\cong {\cal Z},$

(2) for any finite subset ${\cal F}\subset A_m$ and  $\ep>0,$
there is $e\in (A_{m+1})_+^1\setminus\{0\}$
such that

(i) $e(t)\in M_{p_{m+1}q_{m+1}}$ for all $t\in [0,1]$ and $e(1)=0,$

(ii) $\|ex-xe\|<\ep\rforal x\in \phi_m({\cal F}),$

(iii) $\phi_{m+1, \infty}((1-e)^\bt\phi_m(f))\in_{\ep} B_z$ for all $f\in {\cal F},$
and, for any $\bt
\in (0, \infty),$
\beq\label{LCCtrnz-n1}
\|\phi_{m+1, \infty}((1-e)^{\bt}\phi_m(y)\|\ge (1-\ep)\|\phi_m(y)\|
 \rforal y\in {\cal F}_m\andeqn
 \eneq
%

(iv) $d_\tau(e)<1/3^m$ for all $\tau\in T(A_{m+1}).$

(Recall that ${\cal F}_m$ was constructed in the proof of Theorem \ref{TCCS}.)

\end{lem}

\begin{proof}
We will keep the notation used in the proof of Lemma \ref{LCa}.

For (i), we note that the \SCA\, $B_z=\lim_{n\to\infty}(Z_{p_m,q_m}, {\phi_m}|_{Z_{p_m,q_m}})$ has been identified in Remark \ref{RDJSZ},
where
\beq\label{LCCZ-1}
Z_{p_m,q_m}=\{f\in C([0,1], M_{p_mq_m}): f(0)\in M_{p_m}\otimes 1_{q_m}\andeqn f(1)\in 1_{p_m}\otimes M_{q_m}\}.
\eneq

\noindent
There is $\dt\in(0,\ep/2)
$ such that, if $|t-t'|<2\dt,$
\beq\label{LCCtrnz-0}
\|\phi_m(f)(t)-\phi_m(f)(t')\|
<\ep/4\rforal f\in {\cal F}.
\eneq
In particular,  there is $t_1\in (0,1)$  ($1-t_1<\dt$) such that
\beq\label{LCCtrnz-1}
\|\phi_m(f)(t)-\phi_m(f)(1)\|<\ep/4\rforal f\in {\cal F} \andeqn t\in (t_1,1).
\eneq
Choose a continuous function $g\in C([0,1])$ such
that $0\le g\le 1,$ $g(t)=1$ for all $t\in [0,t_1]$ and 
$g(t)=(1-t)/(1-t_1)$
for $t\in (t_1,1].$
Let $e_0(t)=g(t)\cdot 1_{A_m}$ 
for all $t\in[0,1].$
Note that $e_0(0)=1_{p_mq_m}\in M_{p_m}(B)\otimes 1_{q_m}$
and $e_0(1)=0\in 1_{p_m}\otimes 
M_{q_m}.$ So $e_0\in A_m.$
Moreover, $e_0$ is in the center of $A_m.$
Define $\sigma_0: M_{p_m}(B)\otimes M_{q_m}\to M_{d_mp_m}(B)\otimes M_{5q_m}$
by $\sigma_0'\otimes s,$  where
\beq
\sigma_0'(a)=\begin{pmatrix}\theta_1(a) &0\\
                                            0 & 0\end{pmatrix}\rforal a\in M_{p_m}(B)\andeqn s(c)=c\otimes 1_5
                                            \rforal c\in M_{q_m},
\eneq
where 
$\theta_1: M_{p_m}(B)\to M_{p_m}\subset M_{p_m}(B)$ is defined by
$\theta_1(c)(x)
=c(0)$ for 
$c\in M_{p_m}(B)=\, M_{p_m}(C_0(\widetilde{(0,1], C)}),$ 
and for all $x\in [0,1],$
and  the ``0" in the lower corner has the size of 
$(d_m-1)p_m\times (d_m-1)p_m.$
Then define $\sigma_1:
A_m\to C([0,1], M_{d_mp_m5q_m}(B))$
by
\beq\label{LCCtrnz-3}
\sigma_1(f)(t)=\sigma_0(f(t))\rforal f\in E_{p_m,q_m}
\andeqn t\in [0,1].
\eneq
It follows that,
 for all fixded $t\in [0,1],$
%
\beq\label{LCCtrnz-4}
\hspace{-0.3in}\sigma_1(e_0)(t)&=&\sigma_0(e_0(t))
=\sigma_0(g(t)\cdot 1_{A_m})
=\sigma_0(g(t)\cdot 1_{p_m}\otimes 1_{q_m})\\
&&\hspace{-0.9in}=\begin{pmatrix} (\theta_1(g(t)\cdot 1_{p_m})&0\\
                                            0 & 0\end{pmatrix}\otimes 1_{5q_m}
                                            =\begin{pmatrix} g(t)\cdot 1_{p_m}&0\\
                                            0 & 0\end{pmatrix}\otimes 1_{5q_m}=\begin{pmatrix} g(t)\cdot 1_{p_mq_m} &0\\
                                            0 & 0\end{pmatrix}\otimes 1_{5},
\eneq
where the last ``0" has the size $(d_m-1)p_mq_m)\times (d_m-1)p_mq_m$ in the last matrix above). Thus
\beq\label{LCCtrnz-5}
\sigma_1(e_0)(0)
                                            =b\otimes 1_{5q_m}
                                          \andeqn \sigma_1(e_0(1))=0,
                                          \eneq
                                                                                  where
                                         $ b=\begin{pmatrix} 1_{p_m} &0\\
                                            0 & 0\end{pmatrix}.$
                                                       It follows that $\sigma_1(e_0)\in E_{d_mp_m, 5q_m}.$
Note that, for each $\tau\in T(A_m),$ by \eqref{LCa-3}.
\beq\label{LCCtrnz-6-}
\tau(\sigma_1(e_0))<1/3^m.
\eneq
Let us recall the definition of $\tilde\psi_{m,i}$ in  the proof of Lemma \ref{LCa}, $1\le i\le k$ (see \eqref{LCa-tildepsi}).
Then,  for all $f\in A_m,$ by   \eqref{LCa-tildepsi}, \eqref{LCa-22+1}, \eqref{LCa-23}, and \eqref{Dtheta-1},
for each $t\in [0,1],$
\beq\label{LCCtrnz-6}
\tilde\psi_{m,i}(f)(t)\sigma_1(e_0)(t)
&=&\begin{pmatrix}   \theta^{(i)}(f)(t) & 0\\
                                                         0 & \gamma_m(f(t))\end{pmatrix} \otimes 1_5 \cdot \begin{pmatrix} g(t)\cdot 1_{p_mq_m} &0\\
                                            0 & 0\end{pmatrix}\otimes 1_{5}\\
                                            &=&\begin{pmatrix}   \theta^{(i)}(f)(t)\cdot g(t)\cdot 1_{p_mq_m}& 0\\
                                                         0 & 0\end{pmatrix}\otimes 1_5\\
                                                         &=&\begin{pmatrix}   g(t)\cdot 1_{p_mq_m} & 0\\
                                                         0 & 0\end{pmatrix} \otimes 1_5 \cdot \begin{pmatrix} \theta^{(i)}(f)(t) &0\\
                                            0 & \gamma_m(f(t))\end{pmatrix}\otimes 1_{5}\\
                                             &=& \sigma_1(e_0)(t){\tilde
                                             \psi}_{m,i}(f)(t).
\eneq
In other words,  for all $f\in E_{p_m,q_m},$
\beq\label{LCCtrnz-6+1}
\tilde\psi_{m,i}(f)\sigma_1(e_0)=\sigma_1(e_0)\tilde\psi_{m,i}(f),\,\,i=1,2,...,k.
\eneq
Define $\af: [0,1]\to [0,1]$ by
\beq
\af(t)=\begin{cases} {\frac{t}{t_1}}& \text{if}\,\, t\in [0,t_1];\\
                                     1& \text{if}\,\, t\in (t_1,1].\end{cases}
\eneq
Then $f\circ \af\in E_{p_j,q_j},$ if $f\in E_{p_j,q_j}$ for all $j.$
Moreover, by \eqref{LCCtrnz-0}, 
\beq\label{LCCtrnz-7}
\|\phi_m(f)\circ \af
-\phi_m(f)\|<\ep/4\rforal f\in {\cal F}.
\eneq
Therefore, for each $f\in A_m,$ and, for each $t\in [0,1],$   
each $\bt\in(0,\infty),$ with  $l=d_mp_m5q_m,$
\beq\label{LCCtrnz-8}
(1_l-\sigma_1(e_0))^\bt{\tilde\psi_{m,i}}(f)\circ \af
(t)=\begin{pmatrix} (1-g(t))^\bt\cdot 1_{p_mq_m}\cdot \theta^{(i)}(f)\circ \af(t) & 0\\
                                                        0& \gamma_m(f(t))\end{pmatrix} \otimes 1_5,
\eneq
for $i=1,2,...,k.$
For $t\in [0,t_1],$  by the definition of $g,$
\beq\label{LCCtrnz-8+1}
(1-g(t))^\bt\cdot 1_{p_mq_m}\cdot \theta^{(i)}(f)(t)=0.
\eneq
For $t\in (t_1,1],$
\beq\label{LCCtrnz-8+2}
(1-g(t))^\bt\cdot 1_{p_mq_m}\cdot \theta^{(i)}(f)\circ \af
(t)=(1-g(t))^\bt\cdot \theta^{(i)}(f)(1)\in M_{p_mq_m}.
\eneq
Hence
\beq\label{LCCtrnz-8+3}
(1_l-\sigma_1(e_0))^\bt\tilde\psi_{m,i}(f\circ \af)
\in C([0,1], M_{d_mp_m5q_m}).
\eneq
Moreover, by \eqref{LCCS-10}, for $f\in {\cal F}_m,$
\beq\label{LCCtrnz-8+3+1}
\|(1_l-\sigma_1(e_0))^\bt\tilde\psi_{m,i}(f)\circ \af
\|\ge (1-1/m)\|f\|.
\eneq
Define, using the same $v$ as in \eqref{LCa-25} 
($\sigma_1(e_0)(t)$ repeats $k$ times),
\beq\label{LCCtrnz-10}
e:=v(t)^*\begin{pmatrix} \sigma_1(e_0)(t) & 0  &\cdots  & 0\\
                                   0 & \sigma_1(e_0)(t)  &\cdots  &0\\
                                  \vdots  &  \vdots & & \vdots\\
                                   0 &  0 & \cdots    & \sigma_1(e_0)(t)\end{pmatrix}v(t).
                                     \eneq
                                     With $b$ in  the line below \eqref{LCCtrnz-5},
                                     $b\otimes 1_{5q_m}\otimes 1_{r_0}=b\otimes 1_{r_05q_m}=(b\otimes 1_{t_0})\otimes 1_{q_{m+1}}$
                                     (see    \eqref{LCa-15+}).
Since $\sigma_1(e_0)\in E_{d_mp_m,5q_m},$ as in \eqref{LCCtrnz-5}, 
\beq\label{LCCtrnz-11}
e(0)=v_0^*\begin{pmatrix} b\otimes 1_{5q_m} & 0  &\cdots  & 0\\
                                   0 & b\otimes 1_{5q_m}  &\cdots  &0\\
                                  \vdots  &  \vdots & & \vdots\\
                                   0 &  0 & \cdots    & b\otimes 1_{5q_m}\end{pmatrix}v_0\in M_{p_{m+1}}\otimes 1_{q_{m+1}}
                                   \eneq
                                   (see \eqref{LCa-19+1}).
Combining the fact $e(1)=0,$ one concludes $e\in E_{p_{m+1},q_{m+1}}=A_{m+1}.$
Moreover,  by \eqref{LCCtrnz-4} and the fact $v\in C([0,1], M_{p_{m+1}q_{m+1}})$
(see lines above \eqref{LCa-25}), $e(t)\in M_{p_{m+1}q_{m+1}}$ for each $t\in [0,1]$ and $e(1)=0.$
So (i) in part (2) of the statement of the lemma holds.

By \eqref{LCCtrnz-6+1} and \eqref{LCa-phim}, one computes, for all $f\in A_m,$
\beq\label{LCCtrnz-11+}
e\phi_m(f)&=&v^*\begin{pmatrix} \sigma_1(e_0)\tilde\psi_{m,1}(f) &  &&0 \\
                                    &  &  \ddots& \\
                                    0&&   & \sigma_1(e_0)\tilde\psi_{m,k}(f)\end{pmatrix}v\\
                                    &=&v^*\begin{pmatrix} \tilde\psi_{m,1}(f)\sigma_1(e_0) &  &&0 \\
                                    &  &  \ddots& \\
                                    0&&   & \tilde\psi_{m,k}(f)\sigma_1(e_0)\end{pmatrix}v
                                    =\phi_m(f) e
                                     \eneq
                                     In other words, (ii) of (2) in the statement holds.
Now
\beq\nonumber
(1-e)^\bt\phi_m(f\circ \af)
=v^*\begin{pmatrix} (1_l-\sigma_1(e_0))^\bt\tilde\psi_{m,1}(f)\circ \af
 &  &&0 \\
                                    &  &  \ddots& \\
                                    0&&   & (1_l-\sigma_1(e_0))^\bt \tilde\psi_{m,k}(f)\circ \af
                                    \end{pmatrix}v
                                     \eneq
for all $f\in A_m,$ where $l=d_mp_m5q_m.$  Note that $(1-e)^\bt\phi_m(f)\circ \af
\in E_{p_{m+1}, q_{m+1}}.$
It follows from  \eqref{LCCtrnz-8+3}
that
\beq
(1-e)^\bt\phi_m(f)\circ \af
\in Z_{p_{m+1}, q_{m+1}}\rforal f\in {\cal F}.
\eneq
Then, by \eqref{LCCtrnz-7},
\beq\label{LCCtrnz-13}
(1-e)^\bt\phi_m(f)\in_{\ep/4} Z_{p_{m+1}, q_{m+1}}\rforal f\in {\cal F}.
\eneq
It follows that
\beq\label{LCCtrnz-14}
\phi_{m, \infty}((1-e)^\bt f)\in_{\ep} B_z\rforal f\in {\cal F}.
\eneq
Moreover, by \eqref{LCCtrnz-8+3+1}, \eqref{LCCtrnz-n1} also holds.
So (iii) of part (2) of the statement holds.
It follows from \eqref{LCCtrnz-6-}  that (iv) in the statement of the lemma also holds.

\end{proof}

\begin{lem}\label{LCompam}
Let
\beq
E_{p,q}=\{(f,c): C([0,1], M_{pq}(B))\oplus (M_p(B)\oplus M_q): \pi_0(c)=f(0)\andeqn \pi_1(c)=f(1)\},
\eneq
where $\pi_0: M_p(B)\oplus M_q
\to M_p(B)\otimes 1_q\subset M_{pq}(B)$ defined by $\pi_0(c_1\oplus c_2)=c_1\otimes 1_q$
for all $c_1\in M_p(B)$ and $c_2\in M_q,$ and
$\pi_1: M_p(B)\oplus M_q 
\to 1_p\otimes M_q\subset M_{pq}(B)$ defined by $\pi_1((c_1\oplus c_2))=1_p\otimes c_2$
for all $c_1\in M_p(B)$ and $c_2\in M_q$ (see \eqref{DEmkmaptorus}).
Let
\beq\label{Lcompam-0}
L_{p,q}=\{(f,c): C([0,1], M_{pq})\oplus M_p: \pi_0|_{M_p}(c)=f(0)\},
\eneq
where $\pi_0|_{M_p}(c)
=c\otimes 1_q$ for all $c\in M_p.$

Suppose that $a, b\in {E_{p,q}}_+$ such that

(1) $a(t)\in C([0,1], M_{pq})$ and $a(1)=0,$

(2) there is $b_0\in C([0,1], M_{pq})_+$ such that $b_0(t)\le b(t)$ for all $t\in [0,1]$
and
\beq\label{Lcompam-1}
a\lesssim b_0\,\,\,{\rm in}\,\,\, L_{p,q},
\eneq
(i.e., there exists a sequence $x_n\in L_{p,q}$ such
that $x_n^*b_0x_n\to a$).
Then
\beq
a\lesssim b\,\,\,{\rm in}\,\,\, E_{p,q}.
\eneq

\end{lem}

\begin{proof}
Let $1>\ep>0.$
Consider continuous function $h_\dt \in E_{p,q},$
\beq
h_\dt(t)=\begin{cases} 1_{M_{pq}(B)} &\text{if}\,\, t\in [0,1-\dt],\\
                                       0&\text{if}\,\, t\in (1-\dt/2,1],\\
                                       {\rm linear,} & {\rm otherwise.}\end{cases}
\eneq
Since $a(1)=0,$ there exists $\dt_0>0$ such that
$\|a-h_{\dt_0}^{1/2}a\cdot h_{\dt_0}^{1/2}\|<\ep.$

Note that $h_{\dt_0}$ lies in the center of $C([0,1], M_{pq}(B)),$ and for any 
$f\in L_{p,q},$ any $n\in \N,$ $h^{1/n}_{\dt_0}\cdot f\in E_{p,q}.$
Then since $a\lesssim b_0$ in $L_{p,q},$  
one checks  $h_{\dt_0}^{1/2}a h_{\dt_0}^{1/2}\lesssim h_{\dt_0}^{1/2}b_0 \cdot h_{\dt_0}^{1/2}\le b_0\le b$ in $E_{p,q}.$ 
This implies $a\approx_{\ep} h_{\dt_0}^{1/2}a\cdot h_{\dt_0}^{1/2}\lesssim b$ in $E_{p,q}.$ 
Since   this holds for any $1>\ep>0,$  one concludes 
$a\lesssim b$ in $E_{p,q}.$
\end{proof}

\begin{df}\label{Detrn=1}
In the spirit of  Definition \ref{DTrapp}, a  simple \CA\, $A$ is said to
 have essential tracial nuclear dimension at most $n,$   if $A$ is essentially tracially in
 ${\cal N}_n,$ the class of \CA s with nuclear dimension at most $n,$ i.e.,
  if, for any $\ep>0,$ any finite
 subsets ${\cal F}\subset A$ and $a\in A_+\setminus \{0\},$
there exist an element $e\in A_+^1$ and
 a \SCA\, $B\subset A$ which has nuclear dimension at most $n$ such that

 (1) $\|ex-xe\|<\ep$ for all $x\in {\cal F},$

 (2) $(1-e)x\in_{\ep} B$ 
 and $\|(1-e)x\|\ge \|x\|-\ep$ for all $x\in {\cal F},$
 and

 (3) $e\lesssim  a.$ 
\vspace{0.1in}
%
\end{df}

Let us denote ${\cal N}_{{\cal Z},s,s}$ the class of separable 
nuclear  simple ${\cal Z}$-stable  \CA s.

\begin{thm}\label{TCCtrZ}
The unital simple \CA\, $A_z^C$ is essentially tracially  in ${\cal N}_{{\cal Z},s,s}$ and
has essential  tracial  nuclear dimension at most 1,
has stable rank one, and strict comparison for positive elements.
Moreover,
$A_z^C$ has a unique tracial state and has no 2-quasitraces other than the unique tracial state, and
\beq
(K_0(A_z^C), K_0(A_z^C)_+, [1_{A_z^C}], K_1(A_z^C))=(\Z, \Z_+, 1, \{0\}).
\eneq
\end{thm}
(Recall that, if $C$ is  exact and not nuclear, then $A_z^C$ is exact and not nuclear 
(Theorem \ref{TCCS}), and
if $A_z^C$ is not exact, then $A_z^C$ is not exact 
(Proposition \ref{Pnonnuclear})).
\begin{proof}
We will first show that $A_z^C$ is essentially tracially in ${\cal N}_{{\cal Z},s,s}.$
We will retain the notation in the construction of $A_z^C.$

Fix a finite subset ${\cal F}$ and  an element $a\in {A_z^C}_+$
with $\|a\|=1.$
To verify $A_z$ has the said property, \wilog, we may assume
that there is a finite subset ${\cal G}\subset A_1^1$ such that $\phi_{1, \infty}({\cal G})={\cal F}.$
By the first few lines of the proof of Theorem \ref{TCCS}, to further simplify notation,
\wilog, we may assume that ${\cal G}={\cal F}_{1,1}={\cal F}_1\cup \{0\}.$
\Wlog, we may assume that there is  $a'\in (A_1)_+^1$
 with $\|a'\|=1$
such that
\beq\label{TCCtrZ-1}
\|\phi_{1, \infty}(a')-a\|<1/4.
\eneq
It follows from Proposition 2.2 of \cite{Rordam-1992-UHF2} that
\beq\label{TCCtrZ-2}
\phi_{1, \infty}(f_{1/4}(a'))=f_{1/4}(\phi_{1, \infty}(a'))\lesssim a.
\eneq
Put $a_0'=f_{1/4}(a')\,(\neq 0).$  Since $A_z$ is simple, there are $x_1,x_2,...,x_k\in A_z$
such that
\beq
\sum_{i=1}^k x_i^* \phi_{1, \infty}(a_0')x_i=1.
\eneq
It follows that, for some large $n_0,$  there are  $y_1,y_2,...,y_k\in A_{n_0}$ and $n_1\ge n_0$ such that
\beq
\|\sum_{i=1}^k \phi_{n_0, n_1}(y_i^*) \phi_{1, n_1}(a_0')\phi_{n_0, n_1}(y_i)-1_{A_{n_1}}\|<1/4.
\eneq
It follows that $a_0:=\phi_{1, n_1}(a_0')$ is a full element in $A_{n_1}.$

Set
\beq\label{TCCtrZ-n1}
d=\inf\{d_\tau(a_0): \tau\in T(A_{n_1})\}.
\eneq
Since $a_0$ is full in $A_{n_1}$ and $a_0\in (A_{n_1})_+,$ $d>0.$
Choose $m>n_1$ such that
\beq\label{TCCtrZ-10}
d/4>1/3^{m-1}.
\eneq

By applying  Lemma \ref{LCCtrnz}, we obtain
$e\in (A_{m+1})_+^1\setminus\{0\}$
such that

(i) $e(t)\in M_{p_{m+1}q_{m+1}}$ for all $t\in [0,1]$ and $e(1)=0,$

(ii) $\|ex-xe\|<\ep\rforal x\in \phi_m(\phi_{1, m}({\cal G})),$

(iii) $\phi_{m+1, \infty}((1-e)\phi_m(\phi_{1, m}(x)))\in_{\ep} B_z,$ and
$\|\phi_{m+1, \infty}((1-e)\phi_m(\phi_{1, m}(x)))\|\ge (1-1/m)\|\phi_{1,m}(x)\|$
for all $x\in {\cal F}_{1,1}.$

(iv) $d_\tau(e)<1/3^m$ for all $\tau\in T(A_{m+1}).$

Denote $a_1=\phi_{n_1,m}(a_0).$ It is full in $A_m.$
Write, as in  Theorem \ref{LCa} and \eqref{LCa-0},
\beq\label{TCCtrZ-11}
\phi_{m}(a_1)=u^*\begin{pmatrix} \Theta_{m}(a_1)  & 0  &\cdots  & 0\\
                                   0 & \gamma_{m}(a_1\circ \xi_1)\otimes 1_5  &\cdots  &0\\
                                  \vdots  &  \vdots & & \vdots\\
                                   0 &  0 & \cdots    & \gamma_m(a_1\circ \xi_k)\otimes 1_5
                                   \end{pmatrix} u,
                                     \eneq
                                where $u\in U(C([0,1], M_{p_{m+1}q_{m+1}}(B))),$
$\Theta_{m}:
 A_m\to   C([0,1], M_{R(m,0)p_mq_m}(B))$ is a \hm,
$R(m,0)\ge 1$ is an integer, and  $\gamma_m: B\to M_{R(m)}$ is a
finite dimensional representation.  Moreover
\beq\label{TCCtrZ-12}
R(m,0)/5kR(m)<1/3^{m}.
\eneq
Let
\beq\label{TCCtrZ-13}
b_0=u^*\begin{pmatrix} 0 & 0  &\cdots  & 0\\
                                   0 & \gamma_{m}(a_1\circ \xi_1)\otimes 1_5  &\cdots  &0\\
                                  \vdots  &  \vdots & & \vdots\\
                                   0 &  0 & \cdots    & \gamma_m(a_1\circ \xi_k)\otimes 1_5
                                   \end{pmatrix} u.
                                     \eneq
Note that
$b_0\in C([0,1], M_{p_{m+1}q_{m+1}}).$
 Moreover, since $a_1\in E_{p_m,q_m},$
$a_1(0)=a_1'\otimes 1_{q_m}$ for some $a_1'\in M_{p_m}(B).$
Therefore
\beq\label{TCCtrZ-14}
\gamma_m(a_1(0))=\gamma_m(a_1')\otimes 1_{q_m}.
\eneq

 Put
\beq\label{TCCtrZ-15}
c_0'&=&\begin{pmatrix} 0 & 0\\
                                 0 &\gamma_m(a_1')\end{pmatrix}\andeqn\\
c_0&=&\begin{pmatrix} 0 & 0\\
                                 0 &\gamma_m(a_1(0))\end{pmatrix}\otimes 1_5=\begin{pmatrix} 0 & 0\\
                                 0 &\gamma_m(a_1')\end{pmatrix}\otimes 1_{5q_m}=c_0'\otimes 1_{5q_m}.
\eneq
Note $c_0'\in M_{d_mp_m}.$
Put
\beq\label{TCCtrZ-16}
c_i(t)&=&\begin{pmatrix} 0 & 0\\
                                 0 &\gamma_m(a_1\circ \xi_i(t)))\end{pmatrix}\otimes 1_5, \,\,\,i=r_0+1,r_0+2, ...,k.
\eneq
Recall that
at $t=0$ (see \eqref{LCa-15+0}),
\beq
\xi_i(0)=\begin{cases} 0 & {\text{if}}\,\,1\le i\le r_0,\\
                                    1/2 &   \text{if}\,\, r_0<i\le k\end{cases}.
\eneq
Recall (see the line below \eqref{LCa-15+}) also that
$
r_05q_m=t_0q_{m+1}$ for some integer $t_0\ge 1.$
Hence
$(c_0'\otimes 1_{5q_m})\otimes 1_{r_0}=c_0'\otimes 1_{r_05q_m}=(c_0'\otimes 1_{t_0})\otimes 1_{q_{m+1}}.$
On the other hand, since  $k=r_0+m(0)q_{m+1}$ (see  \eqref{LCa-15--1}),
\beq
\diag(c_{r_0+1}(0), \cdots c_k(0))
=^s\left(\begin{pmatrix} 0 & 0\\
                                 0 &\gamma_m(a_1(1/2)))\end{pmatrix}\otimes 1_5\right)\otimes 1_{m(0)q_{m+1}}.
\eneq
Note that ``$=^s$'' is implemented by the same scalar unitary as in  \eqref{LCa-18}(see also the end of \ref{NNotah}
for the notation ``$=^s$").
As in \eqref{LCa-19+1} (and the line next to it),  since $b_0\in C([0,1],M_{p_{m+1}q_{m+1}})$
(mentioned earlier),
this implies that $b_0\in L_{p_{m+1}, q_{m+1}}$ (see \eqref{Lcompam-0}).

Since $a_1\ge 0,$ $b_0(t)\le a_1(t)$ for all $t\in [0,1]$ (see \eqref{TCCtrZ-13}).
Since $\phi_k$ is an  injective unital \hm\, for all $k,$ by \eqref{TCCtrZ-n1},
we also have
\beq\label{TCCtrZ-20}
\inf\{d_\tau(\phi_m(a_1): \tau\in T(A_{m+1})\}=\inf\{d_\tau(\phi_{n_1, m+1}(a_0)): \tau\in T(A_m)\}\ge d.
\eneq
By \eqref{TCCtrZ-11}, \eqref{TCCtrZ-13},   \eqref{TCCtrZ-12}, and \eqref{TCCtrZ-10},
we conclude, for each $t\in (0,1),$
\beq
d_\sigma(e(t))<d_\sigma(b_0(t))\rforal \sigma\in T(M_{p_{m+1}q_{m+1}}),\\
d_{\tau_0}(e(0))<d_{\tau_0}(b_0(0))\andeqn d_{\tau_1}(e(1))<d_{\tau_1}(\phi_m(a_1)(1)),
\eneq
where $\tau_0$ is the unique tracial state of $M_{p_{m+1}}\otimes 1_{q_{m+1}}$ and
$\tau_1$ is the unique tracial state of $1_{q_{m+1}}\otimes M_{q_{m+1}}.$
Note that $e(1)=0.$
It follows that, for all $\tau\in T(L_{p_{m+1}, q_{m+1}}),$
\beq\label{TCCtrZ-25}
d_\tau(e)<d_\tau(b_0).
\eneq
By, for example, Theorem 3.18 of \cite{GLN},
\beq\label{TCCtrZ-26}
e\lesssim b_0 \,\,{\rm in}\,\, L_{p_{m+1},q_{m+1}}.
\eneq
By Lemma \ref{LCompam},
\beq
e\lesssim \phi_m(a_1)\,\, {\rm in}\,\, E_{p_{m+1},q_{m+1}}=A_{m+1}.
\eneq
It follows that
\beq
e\lesssim \phi_{m, \infty}(a_1)=f_{1/4}(\phi_{1, \infty}(a'))\lesssim a.
\eneq
Combining this with (ii)  and (iii) above, we conclude that
$A_z^C$ is essentially  tracially in ${\cal N}_{{\cal Z},s,s}.$
Since $B_z\cong {\cal Z}$ which has nuclear dimension 1 (
\cite[Theorem 6.4.]{SWW15}),  
$A_z^C$ has essential tracial nuclear dimension at most 1.
Since ${\cal Z}$ is in ${\cal T},$ $A_z^C$ is e.tracially in ${\cal T}.$
By  Proposition \ref{Pquasit}, every 2-quasitrace of $A_z^C$ is a tracial state. 
By Corollary \ref{CTR-str1}, 
$A$ has stable rank one.


\end{proof}

\begin{rem}\label{RAz}
Note that Theorem \ref{TCCtrZ} actually shows that $A_z^C$ is essentially tracially approximated by ${\cal Z}$ itself, as
$B_z\cong {\cal Z}.$
\end{rem}


\begin{thm}\label{Tmodel}
Let $(G, G_+, g)$ be a countable weakly unperforated simple ordered  group,
$F$ be a countable abelian group, $\Delta$ be a
metrizable Choquet simplex, and
$\lambda: G\to \Aff_+(\Delta)$ be a \hm.

Then, there is a unital simple non-exact \CA\ $A$   which
is e.~tracially in ${\cal  N}_{{\cal Z}, s,s}^1,$
has
essential tracial nuclear dimension at most 1,
 stable rank one, and strict comparison for positive
elements such that
\beq
(K_0(A), K_0(A)_+, [1_A], K_1(A), T(A), \rho_A)=(G, G_+, g, F, \Delta, \lambda).
\eneq

\end{thm}

\begin{proof}
It follows from Theorem 13.50 of \cite{GLN}
that there is a unital simple $A_0$ with finite nuclear dimension which satisfies the UCT
such that
\beq
(K_0(A_0), K_0(A_0)_+, [1_{A_0}], K_1(A_0), T(A_0), \rho_{A_0})=(G, G_+, g, F, \Delta, \lambda).
\eneq
Let $A_z^C$ be the $C^*$-algebra in Theorem \ref{TCCS} with 
$A_z^C$ is non-exact.
Then define $A=A_0\otimes A_z^C.$   {{Note that, since $A_0$ is a separable amenable
\CA\, satisfying the UCT,  
by \cite[Theorem 2.14]{Sch}, the K\"{u}nneth formula holds for the tensor product $A=A_0\otimes A_z^C.$}}
Since {{the only normalized 2-quasitrace of $A_z^C$ is the unique tracial state and}} 
$$
(K_0(A_z^C), K_0(A_z^C)_+, [1_{A_z^C}], K_1(A_z^C))=(\Z, \Z_+, 1, 0),
$$
one computes, {{applying the K\"{u}nneth formula,}}  that
\beq
(K_0(A), K_0(A)_+, [1_{A}], K_1(A), T(A), \rho_{A})=(G, G_+, g, F, \Delta, \lambda).
\eneq
We will show that $A$ is essentially tracially in ${\cal N}_{{\cal Z},s,s}$ 
and has e.tracial nuclear dimension at most 1.  Once this is done, 
by Definition \ref{Detrn=1}, $A$ has essentially tracial nuclear dimension at most 1, 
and, by Corollary \ref{CTR-str1}, has stable rank one and strict comparison for positive elements.

To see that $A$ is essentially tracially in ${\cal N}_{z,s,s},$  let $\ep>0,$
${\cal F}\subset A$ be a finite subset set and let $d\in A_+\setminus \{0\}.$ 

\Wlog, we may assume that there are  $n\in\N,$
$M>0,$ and a finite subset ${\cal F}_0\subset A$ and 
${\cal F}_1\subset A_z^C$ such that 
\beq
{\cal F}=\{\sum_{i=1}^na_i\otimes b_i, a_i\in {\cal F}_0, b_i\in {\cal F}_1\}\andeqn\\
\|f_0\|, \|f_1\|\le M,\,\,{\rm if}\,\, f_0\in {\cal F}_0\andeqn f_1\in {\cal F}_1.
\eneq

By Kirchberg's Slice Lemma (see, for example, \cite[Lemma 4.1.9]{Rordam-2002}), 
there are $a\in (A_0)_+\setminus \{0\}$ and $b\in (A_z^C)_+\setminus \{0\}$
such that $a\otimes b\lesssim d.$

Let us identity $A_0$ with $A_0\otimes {\cal Z},$ see \cite[Corollary 7.3]{W-2012-pure-algebras}.
In $A_0\otimes {\cal Z},$ choose $1_{A_0}\otimes a_z$ with $a_z\in {\cal Z}_+\setminus \{0\}$ 
such that $1_{A_0}\otimes a_z\lesssim_{A_0} a.$ 
Choose $b_z\in (B_z)_+\setminus \{0\}\subset A_z^C$  with $B_z\cong {\cal Z}$ (see Remark \ref{RAz})
such that $b_z\lesssim_{A_z^C} b.$ 

Note $B_z\cong {\cal Z}_b\otimes B_z,$ where ${\cal Z}_b\cong {\cal Z}.$ 
 Put $c_0:=\sigma(a_z)\otimes b_z\in B_z,$ 
where $\sigma: 1_{A_0}\otimes {\cal Z}(\subset A_0)\to {\cal Z}_b\otimes 1_{B_z}$ is an isomorphism. 
Consider $D_0:=A_0\otimes \sigma(1_{A_0}\otimes {\cal Z})\otimes 1_{B_z}\subset A_0\otimes B_z.$
We may also write $D_0=(A_0\otimes {\cal Z})\otimes \sigma(1_{A_0}\otimes {\cal Z})\otimes 1_{B_z}.$
There is a sequence of unitaries $v_n\in D_0$ 
such that
\beq
\lim_{n\to\infty}v_n^*(1_{A_0}\otimes \sigma(a_z)\otimes 1_{B_z})v_n=1_{A_0}\otimes a_z\otimes 1_{B_z}.
\eneq
It follows that 
\beq\label{Tmodle-10}
1_{A_0}\otimes c_0=1_{A_0}\otimes \sigma(a_z)\otimes b_z\lesssim a\otimes b.
\eneq
By  Remark \ref{RAz}, there exists $e_1\in A_z^C$ with $0\le e_1\le 1$ such that, for all $f\in {\cal F}_1,$
\beq
&&\|e_1f-fe_1\|<\ep/2(nM)^2, \,\,(1_{A_z^C}-e_1)f\in_{\ep/2nM^2} B_z,\\\label{Tmodle-10+}
&&
\|(1-e)f\|\ge (1-\ep/2(nM)^2)\|f\|
 \andeqn e_1\lesssim c_0.
\eneq
Put $B=A_0\otimes B_z.$ Then, since $A_0$ has nuclear dimension at most 1 
(see, for example, \cite[Theorem B]{CETWW-2019}) 
and $A_0$ is ${\cal Z}$-stable,  
$B\cong A_0.$ Therefore $B$ is ${\cal Z}$-stable and has nuclear dimension at most 1.

Put $e=1_{A_0}\otimes e_1.$ Then $0\le e\le 1.$     For any $f\in {\cal F},$ $f=\sum_{i=1}^n a_i\otimes b_i$ 
for some $a_i\in {\cal F}_0$ and $b_i\in {\cal F}_1.$ It follows that 
\beq
\|ef-fe\|=\|e(\sum_{i=1}^n a_i\otimes b_i)-\sum_{i=1}^n a_i\otimes b_i)e\|
=\|(\sum_{i=1}^n a_i\otimes (e_1b_i-b_ie_1)\|<\ep.
\eneq
Also
\beq
&&(1-e)f=(1-1_{A_0}\otimes e_1)(\sum_{i=1}^n a_i\otimes b_i)=\sum_{i=1}^na_i\otimes (1_{A_z^C}-e_1)b_i\in_\ep A_0\otimes B_z.
\eneq
Moreover, by \eqref{Tmodle-10} and \eqref{Tmodle-10+},
\beq
e\lesssim a\otimes b\lesssim d.
\eneq
These imply that $A_0\otimes A_z^C$ is essentially tracially in ${\cal N}_{{\cal Z},s,s}.$
Since $B$ has nuclear dimension no more than 1 (see \cite[Theorem B]{CETWW-2019}).

Since $A_z^C$ embedded into $A_0\otimes A_z^C$ and  $A_z^C$ is not exact, $A_0\otimes A_z^C$ is also not exact (see, for example, Proposition 2.6 of \cite{W}). 

\end{proof}

\begin{rem}
(1) Let $A_0$ be a unital separable nuclear purely infinite simple \CA\, in the UCT class.
Then the proof of Theorem \ref{Tmodel} also shows that $A:=A_0\otimes A_z^C$ 
is an non-exact purely infinite simple \CA\, which has  essential  tracial nuclear dimension 1
and ${\rm Ell}(A)={\rm Ell}(A_0).$

(2) If the RFD \CA\, $C$ at the beginning of section 6 is amenable
then 
$C_0((0,1], C)$ is a nuclear contractible \CA\, which satisfies the UCT. It follows that the unitization 
$B$ of $C_0((0,1], C)$ also satisfies the UCT.  
Therefore $D(m,k)$ and $I=C_0((0,1), M_{mk}(B))$ in \eqref{UCTlater} satisfy the UCT. Thus 
$E_{m,k}$ is nuclear and satisfies the UCT.   It follows that $A_z^C$ is a unital  amenable separable 
simple 
stable rank one \CA\, with a unique tracial state which also has strict comparison for positive elements and 
satisfies the UCT.  By 
\cite[Theorem 1.1]{MS2012_strict_comparison_and_z_stability} $A_z^C$ 
is ${\cal Z}$-stable. By \cite[Theorem 1.1]{MS2014} $A_z^C$ has finite 
nuclear dimension. 
Then by \cite[Corollary 4.11]{EGLN},
$A_z^C$ is classifiable by the Elliott invariant. Since $A_z^C$ has the same Elliott invariant of ${\cal Z},$ it follows that $A_z^C\cong {\cal Z}.$ 

\end{rem}

xlf@fudan.edu.cn\\

hlin@uoregon.edu

\end{document}